\documentclass[leqno,11pt]{amsart}
\usepackage[top=1.40in, bottom=1.40in, left=1.40in, right=1.40in]{geometry}
\usepackage{amscd,amsfonts,mathrsfs,amsthm,enumerate}
\usepackage{amssymb, amsmath,bigints}
\usepackage{stmaryrd}
\usepackage{epsfig}
\usepackage[sorted-cites,nobysame]{amsrefs}
\usepackage{latexsym}
\usepackage[usenames,dvipsnames]{color}
\usepackage[colorlinks=true,linkcolor=blue,citecolor=BrickRed,urlcolor=RoyalBlue]{hyperref}
\newtheorem{thm}{Theorem}[section]
\newtheorem{proposition}[thm]{Proposition}
\newtheorem{corollary}[thm]{Corollary}
\newtheorem{definition}[thm]{Definition}
\newtheorem{theorem}[thm]{Theorem}
\newtheorem{lemma}[thm]{Lemma}
\newtheorem{remark}[thm]{Remark}
\def\p{\mathsf{p}}

\makeatletter
\@namedef{subjclassname@2020}{\textup{2020} Mathematics Subject Classification}
\makeatother

\title[Maximal $L^{q}$-regularity for the Laplacian on manifolds with edges]{Maximal $L^{q}$-regularity for the Laplacian on manifolds with edges}
\author{Nikolaos Roidos}
\address{Department of Mathematics, University of Patras, 26504 Rio Patras, Greece}
\email{roidos@math.upatras.gr}

\begin{document}

\date{\today}
\thanks{The research project is implemented in the framework of H.F.R.I call ``Basic research Financing (Horizontal support of all Sciences)'' under the National Recovery and Resilience Plan ``Greece 2.0" funded by the European Union - NextGenerationEU (H.F.R.I. Project Number: 14758).}

\begin{abstract}\hspace{-6pt}
We introduce an $R$-sectoriality perturbation technique for non-commuting operators defined in Bochner spaces. Based on this and on bounded $H^{\infty}$-functional calculus results for the Laplacian on manifolds with conical singularities, we show maximal $L^{q}$-regularity for the Laplacian on manifolds with edge type singularities in appropriate weighted Sobolev spaces. As an application, we consider the porous medium equation on manifolds with edges and show short time existence, uniqueness and maximal regularity for the solution. We also provide space asymptotics near the singularities in terms of the local geometry.
\end{abstract}

\maketitle
\tableofcontents

\section{Introduction}

 Let $n,\nu\in\mathbb{N}$ and $\mathbb{V}=(\mathcal{V},\mathfrak{m})$ be an $n$-dimensional smooth closed (i.e. compact without boundary) and connected Riemannian manifold. Moreover, let $\mathcal{B}$ be a $\nu$-dimensional smooth compact manifold with smooth closed and possibly disconnected boundary $\partial\mathcal{B}=\mathcal{Y}$. We call $\mathcal{F}=\mathcal{V}\times \mathcal{B}$ {\em manifold with edge type singularities} or {\em manifold with edges} or simply {\em edge manifold}, when it is equipped with a globally defined, positive definite in the interior and degenerate at the boundary metric $\mathfrak{e}=\mathfrak{m}+\mathfrak{g}$, where $\mathfrak{g}$ is a family over $z\in\mathcal{V}$ of Riemannian metrics on $\mathcal{B}$ such that, on a collar neighborhood of the boundary diffeomorphic to $[0,1)\times\mathcal{Y}$, we have
$$
\mathfrak{g}(z)|_{[0,1)\times \mathcal{Y}}=dx^{2}+x^{2}\mathfrak{h}(z,x), \quad z\in \mathcal{V}, x\in [0,1),
$$
with $\mathbb{V}\times [0,1)\ni (z,x) \mapsto \mathfrak{h}(z,x)$ being a smooth family of Riemannian metrics on $\mathcal{Y}$. Denote $\mathbb{F}=(\mathcal{F},\mathfrak{e})$ and let $\mathbb{F}^{\circ}=(\mathcal{F}\backslash\partial\mathcal{F},\mathfrak{e})$ be the interior of $\mathbb{F}$. We identify the boundary $\partial\mathcal{F}=\mathcal{V}\times \{0\}\times \mathcal{Y}$ with the singularities of $\mathbb{F}$.

The Laplacian $\Delta_{\mathfrak{e}}$ on $\mathbb{F}$ is expressed as 
\begin{equation}\label{deltaF}
\Delta_{\mathfrak{e}}=\Delta_{\mathfrak{m}}+\Delta_{\mathfrak{g}}+D,
\end{equation}
where $\Delta_{\mathfrak{m}}$ is the Laplacian on $\mathbb{V}$, $\Delta_{\mathfrak{g}}$ is the Laplacian on $(\mathcal{B},\mathfrak{g})$ and
$$
D=\Bigg\{\begin{array}{lll} 0 & \text{if}& \nu=1\\
\sum_{i,j=1}^{n}\frac{\partial_{z_{i}}(\det[\mathfrak{g}])}{2\det[\mathfrak{g}]}\mathfrak{m}^{i,j}\partial_{z_{j}} & \text{if}& \nu>1,
\end{array}
$$
where $(z_{1},\dots,z_{n})$ are local coordinates in $\mathcal{V}$. While $D$ has smooth coefficients, $\Delta_{\mathfrak{g}}$ in local coordinates $(z,x,y)\in \mathcal{V}\times(0,1)\times\mathcal{Y}$ is given by
\begin{equation}\label{Deltaz}
\Delta_{\mathfrak{g}}=\Bigg\{\begin{array}{lll} x^{-2}((x\partial_{x})^{2}-(x\partial_{x})) & \text{if}& \nu=1\\
x^{-2}\Big((x\partial_{x})^{2}+\Big(\nu-2+\frac{x\partial_{x}\det[\mathfrak{h}(z,x)]}{2\det[\mathfrak{h}(z,x)]}\Big)(x\partial_{x})+\Delta_{\mathfrak{h}(z,x)}\Big) & \text{if}& \nu>1,
\end{array}
\end{equation}
where $\Delta_{\mathfrak{h}(z,x)}$ is the Laplacian on $\mathcal{Y}$ induced by the metric $\mathfrak{h}(z,x)$.

We have defined $\mathbb{F}$ being the product of $\mathbb{V}$ and a manifold with isolated conical singularities. Based on this, endow $\mathcal{B}$ with a globally defined, positive definite in the interior and degenerate at the boundary metric $\mathfrak{p}$, which, in local coordinates $(x,y)\in[0,1)\times\partial\mathcal{B}$ on a collar neighborhood of the boundary, is of the form
\begin{equation}\label{metricg}
\mathfrak{p}=dx^{2}+x^{2}\mathfrak{q}(x),
\end{equation}
where $[0,1)\ni x \mapsto \mathfrak{q}(x)$ is a smooth family of Riemannian metrics on $\partial\mathcal{B}=\mathcal{Y}$. Denote $\mathbb{B}=(\mathcal{B},\mathfrak{p})$; we call $\mathbb{B}$ {\em conic manifold} or {\em manifold with conical singularities}, which are identified with the set $\{0\}\times\partial\mathcal{B}$. Furthermore, denote $\partial\mathbb{B}=(\partial\mathcal{B},\mathfrak{q}(0))$ and let $\partial\mathcal{B}=\sqcup_{j=1}^{k_{\mathcal{B}}}\partial\mathcal{B}_{j}$, for certain $k_{\mathcal{B}}\in\mathbb{N}$, where $\partial\mathcal{B}_{j}$ are closed, smooth and connected. 

The Laplacian $\Delta_{\mathfrak{p}}$ on $\mathbb{B}$, in local coordinates on the collar part $[0,1)\times\partial\mathcal{B}$, is given by \eqref{Deltaz} with $\mathfrak{h}(z,x)$ replaced by $\mathfrak{q}(x)$. It is a second order degenerate differential operator which belongs to the class of {\em cone differential operators} or {\em Fuchs type operators}, see Section 3. Such operators act naturally on scales of {\em Mellin-Sobolev spaces} $\mathcal{H}_{p}^{s,\gamma}(\mathbb{B})$, $p\in(1,\infty)$, $s,\gamma\in\mathbb{R}$, see Definition \ref{MellSob}. Denote by $H_{q}^{s}(\cdot;\cdot)$ or $H_{q}^{s}(\cdot)$, $q\in(1,\infty)$, $s\in\mathbb{R}$, the usual Bessel potential space (in particular we have $H_{q}^{0}=L^{q}$) and let $H_{q,loc}^{s}(\cdot)$ be the space of locally $H_{q}^{s}(\cdot)$-functions. If in particular $s\in \mathbb{N}_{0}=\mathbb{N}\cup\{0\}$, then $\mathcal{H}_{p}^{s,\gamma}(\mathbb{B})$ is the space of all functions $u$ such that $u\in H^{s}_{p,loc}(\mathbb{B}^{\circ})$ and
\begin{equation}\label{MellSobint}
x^{\frac{\nu}{2}-\gamma}(x\partial_{x})^{k}\partial_{y}^{\alpha}(\omega u) \in L^{p}\big([0,1)\times\partial\mathcal{B},\sqrt{\det[\mathfrak{q}(x)]}\frac{dx}{x}dy\big),\quad k+|\alpha|\leq s,
\end{equation}
where $\omega\in C^{\infty}(\mathbb{B})$ is a fixed cut-off function near the boundary, i.e. a smooth non-negative function on $\mathcal{B}$ with $\omega=1$ near $\{0\}\times\partial\mathcal{B}$ and $\omega=0$ on $\mathcal{B}\backslash([0,1)\times\partial\mathcal{B})$. Furthermore, when it is expressed in local coordinates $(x,y)\in [0,1)\times\partial\mathcal{B}$, we assume that $\omega$ depends only on $x$. We consider $\omega$ as a smooth function on $\mathbb{F}$ by extending it by constant on $\mathcal{V}$. Moreover let $\mathbb{C}_{\omega}$ be the space of all $C^{\infty}(\mathbb{B})$-functions $c$ that vanish on $\mathcal{B}\backslash([0,1)\times\partial\mathcal{B})$ and on each component $[0,1)\times\partial\mathcal{B}_{j}$, $j\in\{1,\dots,k_{\mathcal{B}}\}$, are of the form $c_{j}\omega$, where $c_{j}\in\mathbb{C}$, i.e. $\mathbb{C}_{\omega}$ consists of smooth functions on $\mathbb{B}$ that are locally constant close to the boundary. Endow $\mathbb{C}_{\omega}$ with the norm $c\mapsto (\sum_{j=1}^{k_{\mathcal{B}}}|c_{j}|^{2})^{1/2}$. Note that $\mathbb{C}_{\omega}\hookrightarrow \mathcal{H}_{p}^{s,\gamma}(\mathbb{B})$ if and only if $\gamma<\nu/2$, so that, by $\mathcal{H}_{p}^{s,\gamma}(\mathbb{B})\oplus\mathbb{C}_{\omega}$ we will denote $\mathcal{H}_{p}^{s,\gamma}(\mathbb{B})$ in case of $\gamma<\nu/2$.

Now for $p,q\in(1,\infty)$, $r\geq0$, $s,\gamma\in\mathbb{R}$ and $J\in\{\{0\},\mathbb{C}_{\omega}\}$, we consider the Mellin-Sobolev space valued Bessel potential spaces 
$$
\mathcal{H}_{q,p,\gamma}^{r,s}(\mathbb{F})_{\oplus J}=H_{q}^{r}(\mathbb{V};\mathcal{H}_{p}^{s,\gamma}(\mathbb{B})\oplus J),
$$ 
where, if $J=\{0\}$, then we simply denote $\mathcal{H}_{q,p,\gamma}^{r,s}(\mathbb{F})_{\oplus J}$ by $\mathcal{H}_{q,p,\gamma}^{r,s}(\mathbb{F})$. More precisely, if $\kappa_{j}:V_{j}\rightarrow \mathbb{R}^{n}$, $j\in\{1,\dots,N\}$, $N\in\mathbb{N}$, is a covering of $\mathcal{V}$ by coordinate charts and $\{\phi_{j}\}_{j\in\{1,\dots,N\}}$ a subordinate partition of unity, then 
\begin{equation}\label{nrmHF}
\|\cdot\|_{\mathcal{H}_{q,p,\gamma}^{r,s}(\mathbb{F})_{\oplus J}}=\sum_{j=1}^{N}\|(\kappa_{j})_{\ast}(\phi_{j}\cdot)\|_{H_{q}^{r}(\mathbb{R}^{n};\mathcal{H}_{p}^{s,\gamma}(\mathbb{B})\oplus J)},
\end{equation}
where $(\cdot)_{\ast}$ is the push-forward of distributions. The above space is, up to an equivalence of norms, independent of the choice of the covering $\{\kappa_{j}\}_{j\in\{1,\dots,N\}}$ and the partition $\{\phi_{j}\}_{j\in\{1,\dots,N\}}$.

Let $\mathcal{R}(\theta)$ be the class of {\em $R$-sectorial operators of angle $\theta\in[0,\pi)$}, see Definition \ref{r1sec}. For any $z\in\mathcal{V}$ denote by $\lambda_{z,1}<0$ the greatest nonzero eigenvalue of $\Delta_{\mathfrak{h}(z,0)}$, and let
\begin{equation}\label{lambda1}
\lambda_{1}=\sup_{z\in\mathcal{V}} \lambda_{z,1}<0,
\end{equation}
where the above inequality is due to Lemma \ref{lambda1bound}. Moreover, denote by $\mathcal{D}(A)$ the domain of a linear operator $A$ acting between Banach spaces. Our first result concerns the edge Laplacian $\Delta_{\mathbb{F}}$ and reads as follows.

\begin{theorem}[$R$-sectoriality for the edge Laplacian]\label{thsec33}
Let $p,q\in(1,\infty)$, $r,s\geq0$ and
\begin{equation}\label{gamma2}
\frac{\nu-4}{2}<\gamma< \Bigg\{\begin{array}{lll} 1/2 & \text{if} & \nu=1 \\ \min\Big\{-1+\sqrt{\Big(\frac{\nu-2}{2}\Big)^{2}-\lambda_{1}},\frac{\nu}{2}\Big\} & \text{if} & \nu>1,\end{array}
\end{equation}
where $\lambda_{1}$ is defined in \eqref{lambda1}. Then, the operator $\underline{\Delta}_{\mathfrak{e}}$, defined by 
\begin{equation}\label{DeltaF}
\mathcal{H}_{q,p,\gamma}^{r+2,s}(\mathbb{F})\cap \mathcal{H}_{q,p,\gamma+2}^{r,s+2}(\mathbb{F})_{\oplus\mathbb{C}_{\omega}}=\mathcal{D}(\underline{\Delta}_{\mathfrak{e}})\ni u\mapsto \Delta_{\mathfrak{e}}u \in \mathcal{H}_{q,p,\gamma}^{r,s}(\mathbb{F}),
\end{equation}
is closed. Moreover, for any $\theta\in[0,\pi)$ there exists a $c>0$ such that $c-\underline{\Delta}_{\mathfrak{e}}\in\mathcal{R}(\theta)$.
\end{theorem}

For the proof of the previous theorem we introduce an $R$-sectoriality perturbation technique for non-commuting operators defined in Bochner spaces. It combines the freezing-of-coefficients method, the Kalton-Weis operator valued functional calculus results and $R$-sectoriality perturbation theory; see the proof of Theorem \ref{t1} below. Denote by $W^{k,p}(\cdot;\cdot)$, $k\in\mathbb{N}_{0}$, $p\in(1,\infty)\cup\{\infty\}$, the usual Sobolev space, by $\mathcal{L}(\cdot,\cdot)$ or $\mathcal{L}(\cdot)$ the space of bounded linear operators acting between Banach spaces and let $\mathcal{H}^{\infty}(\theta)$, $\mathcal{RH}^{\infty}(\theta)$ be the class of operators having {\em bounded $H^{\infty}$-calculus} and {\em $R$-bounded $H^{\infty}$-calculus of angle $\theta\in[0,\pi)$}, respectively, see Definition \ref{hinftycal} and Definition \ref{Rbhinftycal}. Moreover recall the class of UMD Banach spaces, i.e. the Banach spaces that possess the {\em unconditionality of martingale differences property}, see e.g. \cite[Section III.4.4]{Am} for a precise definition.

\begin{theorem}\label{t1}
Let $\mu\in\mathbb{N}$, $q\in(1,\infty)$, $\theta_{A},\theta_{B}\in(0,\pi)$ with $\theta_{A}+\theta_{B}>\pi$, $\mathbb{K}\in \{\mathbb{R}^{n},\mathbb{V}\}$ and let $X_{1}\hookrightarrow X_{0}$ be a continuously and densely injected complex Banach couple with $X_{0}$ being UMD. Moreover, let $A$ be an $\mu$-th order differential operator with coefficients in $W^{1,\infty}(\mathbb{K};\mathbb{C})$ and $B(\cdot)\in C(\mathbb{K};\mathcal{L}(X_{1},X_{0}))$, such that, when $\mathbb{K}=\mathbb{R}^{n}$ we assume that $B(\cdot)$ is constant on $\mathbb{R}^{n}\backslash \mathcal{K}$, for certain compact set $\mathcal{K}\subset\mathbb{R}^{n}$. Consider the operators
$$
A:\mathcal{D}(A)=H_{q}^{\mu}(\mathbb{K};X_{0})\rightarrow L^{q}(\mathbb{K};X_{0})\quad\text{and} \quad B:\mathcal{D}(B)=L^{q}(\mathbb{K};X_{1})\rightarrow L^{q}(\mathbb{K};X_{0}),
$$ 
where $(Bu)(z)=B(z)u(z)$ for almost all $z\in\mathbb{K}$ when $u\in \mathcal{D}(B)$. Assume that for any $z\in\mathbb{K}$ there exists a $c\geq0$ depending on $z$, such that at least one of the following two conditions holds true:\\
{\bf (i)} $A\in \mathcal{H}^{\infty}(\theta_{A})$ and $B(z)+c \in \mathcal{RH}^{\infty}(\theta_{B})$.
{\bf (ii)} $A\in \mathcal{RH}^{\infty}(\theta_{A})$ and $B(z)+c \in \mathcal{H}^{\infty}(\theta_{B})$.\\
Then, $A+B$ with domain $H_{q}^{\mu}(\mathbb{K};X_{0})\cap L^{q}(\mathbb{K};X_{1})$ in $L^{q}(\mathbb{K};X_{0})$ is closed and furthermore there exists a $c_{0}\geq 0$ such that $A+B+c_{0}\in \mathcal{R}(\theta_{A+B})$, where $\theta_{A+B}=\min\{\theta_{A},\theta_{B}\}$.
\end{theorem}

There are plenty of examples of operators in the class $\mathcal{H}^{\infty}(\theta)$, $\theta\in[0,\pi)$, see e.g. \cite{Am4} for $L^{ p}$-realizations of general elliptic systems on $\mathbb{R}^{n}$ or on compact manifolds without boundary, under weak conditions on their coefficients. Furthermore, due to \cite[Theorem 5.3]{KaW}, an operator in $\mathcal{H}^{\infty}(\theta)$ also belongs to $\mathcal{RH}^{\infty}(\theta)$ if the underlying space admits Pisier's {\em property $(\alpha)$}, see e.g. \cite[Definition 4.2.7]{PS2} for the definition of the property $(\alpha)$.

Recall that in UMD Banach spaces $R$-sectoriality characterizes maximal $L^{q}$-regularity, i.e. the well-posedness of the related abstract linear parabolic Cauchy's problem in the $L^{d}$-setting, $d\in(1,\infty)$, and furthermore, allows us to obtain classical solutions for quasilinear parabolic problems, see Section 2. In this direction, as an application, we consider the {\em porous medium equation} (PME for short) on $\mathbb{F}$, namely the following semilinear parabolic problem 
\begin{eqnarray}\label{PME1}
u'(t)-\Delta_{\mathfrak{e}} u^{m}(t) &=& f(u(t),t),\quad t\in(0,T],\\\label{PME2}
u(0)&=&u_{0}.
\end{eqnarray}
Here $(\cdot)'=\partial_{t}(\cdot)$, $m>0$ is a fixed parameter, $\mathcal{U}\times[0,T_{0}]\ni(\lambda,t)\mapsto f(\lambda,t)\in\mathbb{C}$ is holomorphic in $\lambda$ and Lipschitz in $t$, for some open neighborhood $\mathcal{U}$ of $\mathrm{Ran}(u_{0})\subseteq\mathbb{C}$ and some $T_{0}>0$, and $T\in(0,T_{0}]$. The evolution described by \eqref{PME1}-\eqref{PME2} can model heat transfer, fluid flow, or gas diffusion in porous media; it can also be regarded as a nonlinear version of the heat equation.

For an introduction to the theory of the PME considered on domains in $\mathbb{R}^{n}$ and on smooth compact manifolds, we refer to \cite{AP}, \cite{BG}, \cite{DHa}, \cite{Otto}, \cite{Va} and to the references therein, where the problem has been extensively studied under many aspects. Concerning singular spaces, in \cite{GrMu}, \cite{GMP}, \cite{GMV}, \cite{GMV2} and \cite{Va0} the PME is considered on manifolds with negative curvature. Among other properties, it is shown existence, uniqueness, regularity, smoothing effects, long time behaviour and bounds for solutions, as well as logarithmic growth of the location of the free boundary. Furthermore, in \cite{RS2}, \cite{RS}, \cite{RSS} and \cite{SS} the PME is studied on manifolds with isolated conical singularities. In addition to existence, uniqueness and regularity results for solutions, it is also obtained their asymptotic behaviour near the singular points in terms of the local geometry. Concerning the PME on manifolds with edges, we show the following.

\begin{theorem}\label{pmeonF}
Assume that $\lambda_{1}$ defined in \eqref{lambda1} satisfies $-\lambda_{1}>\nu-1$ when $\nu>1$. Let $r\in\mathbb{N}_{0}$, $s\geq0$, choose $\delta\in(0,1)$ sufficiently small and $p,q,d\in(1,\infty)$ sufficiently large such that $\nu<p$, $n<q$ and
$$
\frac{2}{d}+\frac{n}{q}<\frac{1}{d}+\delta\Big(1-\frac{1}{d}\Big)<\Bigg\{\begin{array}{lll} 1-\frac{1}{p} & \text{if} & \nu=1 \\ \min\Big\{1-\frac{\nu}{p},-\frac{\nu}{2}+\sqrt{\Big(\frac{\nu-2}{2}\Big)^{2}-\lambda_{1}}\Big\} & \text{if} & \nu>1,\end{array}
$$
and let
\begin{equation}\label{gam2ma}
\frac{\nu-2}{2}+\frac{1}{d}+\delta\Big(1-\frac{1}{d}\Big)<\gamma< \Bigg\{\begin{array}{lll} 1/2 & \text{if} & \nu=1 \\ \min\Big\{-1+\sqrt{\Big(\frac{\nu-2}{2}\Big)^{2}-\lambda_{1}},\frac{\nu}{2}\Big\} & \text{if} & \nu>1.\end{array}
\end{equation}
Then, for any strictly positive $u_{0}$ such that $u_{0}^{m},u_{0}^{\frac{m-1}{m}}\in \mathcal{H}_{q,p,\gamma+2}^{r+2,s+2}(\mathbb{F})_{\oplus\mathbb{C}_{\omega}}$ there exists a $T>0$ and a unique $u: [0,T]\times\mathbb{F}\rightarrow \mathbb{C}$ with the following properties:
\begin{equation}\label{regofu1}
u\in H_{d}^{1}(0,T;\mathcal{H}_{q,p,\gamma}^{r,s}(\mathbb{F})),
\end{equation}
\begin{eqnarray}\nonumber
\lefteqn{u^{m}\in L^{d}(0,T;\mathcal{H}_{q,p,\gamma}^{r+2,s}(\mathbb{F})\cap \mathcal{H}_{q,p,\gamma+2}^{r,s+2}(\mathbb{F})_{\oplus\mathbb{C}_{\omega}})}\\\label{regofu2}
&&\cap \bigcap_{\varepsilon>0} C([0,T]; \mathcal{H}_{q,p,\gamma+(1-\delta)(1-\frac{1}{d})-\varepsilon}^{r+(1+\delta)(1-\frac{1}{d})-\varepsilon,s+(1-\delta)(1-\frac{1}{d})-\varepsilon}(\mathbb{F})_{\oplus\mathbb{C}_{\omega}})
\end{eqnarray}
and $u$ satisfies \eqref{PME1}-\eqref{PME2}.
\end{theorem}

Note that there is no $\lambda_{1}$-assumption in the above theorem in the case of $\nu=1$. Furthermore, $u$ is a classical solution of \eqref{PME1}-\eqref{PME2}. Theorem \ref{pmeonF} together with the Sobolev embedding theorem for the $\mathcal{H}_{q,p,\gamma}^{r,s}(\mathbb{F})_{\oplus\mathbb{C}_{\omega}}$-spaces provides us with the asymptotic behaviour of the solution $u$ near the singularities of $\mathbb{F}$. More precisely, by \eqref{regofu2} and Lemma \ref{propF} (v), 
$$
u^{m}\in C([0,T];C(\mathbb{V};\mathcal{H}_{p}^{s+(1-\delta)(1-\frac{1}{d})-\varepsilon,\gamma+(1-\delta)(1-\frac{1}{d})-\varepsilon}(\mathbb{B})\oplus \mathbb{C}_{\omega}))\hookrightarrow C([0,T];C(\mathbb{F}))
$$ 
and there exists a $c\in C([0,T];C(\mathbb{F}))$ such that, in local coordinates $(z,x,y)\in \mathcal{V}\times[0,1)\times\mathcal{Y}$, we have $c=c(t,z)$, 
$$
u^{m}(t,z,\cdot,\cdot)-\omega c(t,z)\in \mathcal{H}_{p}^{s+(1-\delta)(1-\frac{1}{d})-\varepsilon,\gamma+(1-\delta)(1-\frac{1}{d})-\varepsilon}(\mathbb{B})\hookrightarrow C(\mathbb{B})
$$ 
and
\begin{eqnarray*}
\lefteqn{|u^{m}(t,z,x,y)-\omega c(t,z)|}\\
&\leq& Cx^{\gamma+(1-\delta)(1-\frac{1}{d})-\frac{\nu}{2}-\varepsilon}\|u^{m}(t,z,\cdot,\cdot)-\omega c(t,z)\|_{\mathcal{H}_{p}^{s+(1-\delta)(1-\frac{1}{d})-\varepsilon,\gamma+(1-\delta)(1-\frac{1}{d})-\varepsilon}(\mathbb{B})},
\end{eqnarray*}
for some constant $C>0$ depending only on $\nu$, $p$, $q$, $d$, $s$, $\delta$ and $\varepsilon$. Hence, in particular, the interplay of the local geometry near the singularity with the evolution can be seen by the choice \eqref{gam2ma}.

Concerning manifolds with incomplete edge singularities, in \cite{BaVe1}, \cite{BaVe2}, \cite{LyVe} and \cite{Vertm} it has been shown short or long time existence and convergence for certain geometric flows on such spaces. However, both the underlying singular analysis theory as well as the linear and quasilinear PDE techniques differ from those in this article. We also refer to \cite{KaSc} for a comprehensive theory providing parametrices construction of edge boundary value problems and elliptic regularity in appropriate scales of weighted spaces. Finally, we refer to \cite{GKM3} for the closed extensions and resolvents of elliptic operators on manifolds with edges and to \cite{KM} for the study of the Friedrichs extension of second order elliptic wedge operators. Note that all of the above situations involve manifolds with edges in a broader or more general sense. 

The rest of the article is organised as follows. In Section 2 we recall some basics on the maximal $L^{q}$-regularity theory for linear and quasilinear parabolic problems and prove Theorem \ref{t1}. In Section 3 we present some previous results for the Laplacian on manifolds with conical singularities, regarded as a cone differential operator, and then provide certain properties of the $\mathcal{H}_{q,p,\gamma}^{r,s}(\mathbb{F})$-spaces, useful for treating nonlinear PDE. In Section 4, using the results of Section 3 together with the perturbation technique from the proof of Theorem \ref{t1}, we prove Theorem \ref{thsec33}. An application to the PME on edge manifolds and a proof of Theorem \ref{pmeonF} is provided in Section 5.

\section{Maximal $L^{d}$-regularity for parabolic equations}

Let $X_{1}\hookrightarrow X_{0}$ be a continuously and densely injected complex Banach couple. 

\begin{definition}[Sectoriality]\label{secdef}
Let $\mathcal{P}(K,\theta)$, $K\geq1$, $\theta\in[0,\pi)$, be the class of all closed densely defined linear operators $A$ in $X_{0}$ such that 
$$
S_{\theta}=\{\lambda\in\mathbb{C}\,|\, |\arg(\lambda)|\leq\theta\}\cup\{0\}\subset\rho{(-A)}
$$
and
$$
(1+|\lambda|)\|(A+\lambda)^{-1}\|_{\mathcal{L}(X_{0})}\leq K, \quad \lambda\in S_{\theta}.
$$
The elements in $\mathcal{P}(\theta)=\cup_{K\geq1}\mathcal{P}(K,\theta)$ are called {\em invertible sectorial operators of angle $\theta$}. If $A\in \mathcal{P}(K,\theta)$ for some $K\geq1$ and $\theta\in[0,\pi)$, then $K$ is called {\em a sectorial bound} of $A$.
\end{definition}

\begin{remark}\label{secextnd}
If $A\in\mathcal{P}(\theta)$ for some $\theta\in[0,\pi)$, then there exist $r>0$ and $\phi\in(\theta,\pi)$ such that
$$
\Omega_{r,\phi}=\{\lambda\in \mathbb{C}\, |\, |\lambda|\leq r\}\cup S_{\phi} \subset \rho(-A)
$$
and
$$
(1+|\lambda|)\|(A+\lambda)^{-1}\|_{\mathcal{L}(X_{0})}\leq C, \quad \lambda\in \Omega_{r,\phi},
$$
for certain $C\geq1$, see e.g. \cite[(III.4.7.12)-(III.4.7.13)]{Am}.
\end{remark}

We start by the holomorphic functional calculus for sectorial operators which is defined by the Dunford integral formula, see e.g. \cite[Theorem 1.7]{DHP}. If $A\in \mathcal{P}(\theta)$, due to Remark \ref{secextnd}, we can always assume that $\theta>0$. Moreover, for any $\rho\geq0$ and $\theta\in(0,\pi)$, consider the counterclockwise oriented path
$$
\Gamma_{\rho,\theta}=\{re^{-i\theta}\in\mathbb{C}\,|\,r\geq\rho\}\cup\{\rho e^{i\phi}\in\mathbb{C}\,|\,\theta\leq\phi\leq2\pi-\theta\}\cup\{re^{+i\theta}\in\mathbb{C}\,|\,r\geq\rho\}.
$$
Simply denote $\Gamma_{0,\theta}$ by $\Gamma_{\theta}$, let $-\Gamma_{\rho,\theta}=\{\lambda\in\mathbb{C}\, |\, -\lambda\in \Gamma_{\rho,\theta}\}$ and, for each $c\in\mathbb{R}$, let $c+\Gamma_{\rho,\theta}=\{c+\lambda\in \mathbb{C}\, |\, \lambda\in \Gamma_{\rho,\theta}\}$. Note that $-\Gamma_{\theta}=\Gamma_{\pi-\theta}$. A typical example of the functional calculus are the complex powers; for $\mathrm{Re}(z)<0$ they are defined by
\begin{equation}\label{cp}
A^{z}=\frac{1}{2\pi i}\int_{\Gamma_{\rho,\theta}}(-\lambda)^{z}(A+\lambda)^{-1}d\lambda,
\end{equation}
where $\rho>0$ is sufficiently small. The family $\{A^{z}\}_{\mathrm{Re}(z)<0}\cup\{A^{0}=I\}$ is a strongly continuous analytic semigroup on $X_{0}$, see e.g. \cite[Theorem III.4.6.2 and Theorem III.4.6.5]{Am}. Furthermore, each $A^{z}$, $\mathrm{Re}(z)<0$, is injective and the complex powers for positive real part $A^{-z}$ are defined by $A^{-z}=(A^{z})^{-1}$, see e.g. \cite[(III.4.6.12)]{Am}. For further properties of the complex powers of sectorial operators we refer to \cite[Theorem III.4.6.5]{Am}. Also, recall the following decay property of the resolvent of a sectorial operator. 

\begin{remark}\label{TaLnLem}
If $A\in\mathcal{P}(\theta)$, $\theta\in[0,\pi)$, in $X_{0}$, then for any $\rho\in[0,1]$ there exists a $C>0$, depending only on $\theta$, the sectorial bound of $A$ and $\rho$, such that 
$$
\|A^{\rho}(A+\lambda)^{-1}\|_{\mathcal{L}(X_{0})}\leq\frac{C}{1+|\lambda|^{1-\rho}}, \quad \lambda\in S_{\theta}.
$$
This can be found e.g. in \cite[Lemma 7.1]{RShao}.
\end{remark}

We next consider a certain property of operators in the class $\mathcal{P}(\theta)$. If $\phi\in[0,\pi)$, denote by $H_{0}^{\infty}(\phi)$ the space of all bounded analytic functions $f:\mathbb{C}\backslash S_{\phi}\rightarrow \mathbb{C}$ satisfying 
$$
|f(\lambda)|\leq c \Big(\frac{|\lambda|}{1+|\lambda|^{2}}\Big)^{\eta} \quad \text{for any} \quad \lambda\in \mathbb{C}\backslash S_{\phi},
$$
and some $c,\eta>0$ depending on $f$.

\begin{definition}[Bounded $H^{\infty}$-calculus]\label{hinftycal}
If $\theta\in(0,\pi)$, $\phi\in[0,\theta)$ and $A\in\mathcal{P}(\theta)$, then any $f\in H_{0}^{\infty}(\phi)$ defines an element $f(-A)\in \mathcal{L}(X_{0})$ by 
$$
f(-A)=\frac{1}{2\pi i}\int_{\Gamma_{\theta}}f(\lambda)(A+\lambda)^{-1} d\lambda.
$$
We say that the operator $A$ {\em has bounded $H^{\infty}$-calculus of angle $\phi$}, and we denote by $A\in \mathcal{H}^{\infty}(\phi)$, if there exists some $C>0$ such that
$$
\|f(-A)\|_{\mathcal{L}(X_{0})}\leq C\sup_{\lambda\in\mathbb{C}\backslash S_{\phi}}|f(\lambda)| \quad \mbox{for any} \quad f\in H_{0}^{\infty}(\phi).
$$
\end{definition}

Consider now the following abstract parabolic first order linear Cauchy problem
\begin{eqnarray}\label{app1}
u'(t)+Au(t)&=&f(t), \quad t\in(0,T),\\\label{app2}
u(0)&=&0,
\end{eqnarray}
where $-A:X_{1}\rightarrow X_{0}$ is the infinitesimal generator of an analytic semigroup on $X_{0}$ and $f\in L^{d}(0,T;X_{0})$ for some $d\in(1,\infty)$, $T>0$. We say that the operator $A$ has {\em maximal $L^{d}$-regularity} if for any $f\in L^{d}(0,T;X_{0})$ there exists a unique $u\in H_{d}^{1}(0,T;X_{0})\cap L^{d}(0,T;X_{1})$ solving \eqref{app1}-\eqref{app2}; if this is the case, then $u$ depends continuously on $f$ and the above property is independent of $q$ and $T$. 

\begin{definition}[$R$-boundedness]
Let $\{\epsilon_{k}\}_{k\in\mathbb{N}}$ be the sequence of Rademacher functions. A set $E\subset \mathcal{L}(X_{0})$ is called {\em $R$-bounded} if there exists a $C>0$ such that
$$
\|\sum_{k=1}^{N}\epsilon_{k}T_{k}x_{k}\|_{L^{2}(0,1;X_0)} \leq C \|\sum_{k=1}^{N}\epsilon_{k}x_{k}\|_{L^{2}(0,1;X_0)},
$$
for every $N\in\mathbb{N}$, $T_{1},\dots,T_{N}\in E$ and $x_{1},\dots,x_{N}\in X_0$. The constant $C$ is called {\em an $R$-bound} of $E$.
\end{definition} 

Based on the above, we introduce a boundedness property of the resolvent of a sectorial operator that is related to the maximal $L^{d}$-regularity. 

\begin{definition}[$R$-sectoriality]\label{r1sec}
Let $\mathcal{R}(\theta)$, $\theta\in[0,\pi)$, be the class of all operators $A\in \mathcal{P}(\theta)$ in $X_{0}$ such that the set $\{\lambda(A+\lambda)^{-1}\, |\, \lambda\in S_{\theta}\backslash\{0\}\}$ is $R$-bounded. If $A\in \mathcal{R}(\theta)$ then $A$ is called {\em (invertible) $R$-sectorial of angle $\theta$}; in this case, an $R$-bound $C\geq1$ of the above family is called {\em $R$-sectorial bound} of $A$ and we write $A\in \mathcal{R}(C, \theta)$. 
\end{definition} 

All Banach spaces we will consider in the next sections belong to the class of UMD. In this class the following fundamental result holds true.

\begin{theorem}[{\rm Kalton and Weis, \cite[Theorem 6.5]{KaW} or \cite[Theorem 4.2]{Weis}}]\label{KaWeTh}
If $X_{0}$ is UMD and $A\in\mathcal{R}(\theta)$ in $X_{0}$ with $\theta>\pi/2$, then $A$ has maximal $L^{d}$-regularity. 
\end{theorem}

We recall next a subclass of the $H^{\infty}$-calculus operators, quite efficient in perturbation theory.

\begin{definition}[$R$-bounded $H^{\infty}$-calculus]\label{Rbhinftycal}
Let $\mathcal{RH}(\theta)$, $\theta\in[0,\pi)$, be the class of all operators $A\in \mathcal{H}(\theta)$ in $X_{0}$ such that the set $\{f(A)\, |\, f\in H_{0}^{\infty}(\theta), \sup_{\lambda\in\mathbb{C}\backslash S_{\theta}}|f(\lambda)|\leq1\}$ is $R$-bounded. Any $A\in \mathcal{RH}(\theta)$ is said to have {\em$R$-bounded $H^{\infty}$-calculus of angle $\theta$.} 
\end{definition} 

Concerning the subclass of $\mathcal{P}(\theta)$ we have introduced, in UMD spaces the following inclusion holds
\begin{equation}\label{sectinclutions}
\mathcal{H}(\theta) \subseteq \mathcal{R}(\phi),
\end{equation}
where $0<\phi<\theta<\pi$, see e.g. \cite[Theorem 4.5]{DHP}. Next we show a decay property of the resolvent of the sum of two sectorial operators.

\begin{lemma}\label{Lem1}
Let $A$, $B$ be two resolvent commuting linear operators {\em(}see e.g. \cite[(III.4.9.1)]{Am}{\em)} such that $A\in\mathcal{H}^{\infty}(\theta_{A})$ and $B\in\mathcal{R}(\theta_{B})$, for some $\theta_{A}, \theta_{B}\in (0,\pi)$ satisfying the parabolicity condition $\theta_{A}+\theta_{B}>\pi$. Then: \\
{\em {\bf (i)}} $A+B:\mathcal{D}(A)\cap\mathcal{D}(B) \rightarrow X_{0}$ is closed and belongs to $\mathcal{P}(\theta_{A+B})$, where $\theta_{A+B}=\min\{\theta_{A},\theta_{B}\}$.\\
{\em {\bf (ii)}} For any $\rho\in(0,1)$ we have
$$
\|(A+B)(A+B+\lambda)^{-1}A^{-\rho}\|_{\mathcal{L}(X_{0})}\leq \frac{K}{1+|\lambda|^{\rho}},\quad \lambda\in S_{\theta_{A+B}}, 
$$
for some constant $K>0$ depending only on $\rho$ and on the sectorial bounds of $A$ and $B$.
\end{lemma}
\begin{proof}
{\bf (i)} The closedness and the sectoriality of $A+B$ follow by \cite[Theorem 3.1]{PS}; alternatively they follow by the classical results \cite[Lemma 3.5 and Theorem 3.7]{DG} and \cite[Theorem 6.3]{KaW}. \\
{\bf (ii)} In particular, see e.g. \cite[(3.11)]{DG} or \cite[Theorem 2.1]{Roi}, the resolvent of $A+B$ is given by the formula 
$$
(A+B+\lambda)^{-1}=\frac{1}{2\pi i}\int_{\Gamma_{\theta_{A}}}(A+z)^{-1}(B+\lambda-z)^{-1}dz, \quad \lambda\in S_{\theta_{A+B}}.
$$
By the invertibility of $A+B$, we can assume that $|\lambda|\geq1$. Then, by Remark \ref{secextnd} and \eqref{cp}, for sufficiently small $\varepsilon, \delta>0$ we have
\begin{eqnarray*}
\lefteqn{(A+B+\lambda)^{-1}A^{-\rho}}\\
&=&\frac{1}{2\pi i}\int_{\Gamma_{\theta_{A}-\varepsilon}}(B+\lambda-z)^{-1}\Big(\frac{1}{2\pi i}\int_{-\delta+\Gamma_{\theta_{A}}}(-\mu)^{-\rho}(A+z)^{-1}(A+\mu)^{-1}d\mu\Big)dz\\
&=&\frac{1}{2\pi i}\int_{\Gamma_{\theta_{A}-\varepsilon}}(B+\lambda-z)^{-1}\\
&&\times\Big(\frac{1}{2\pi i}\int_{-\delta+\Gamma_{\theta_{A}}}(-\mu)^{-\rho}(\mu-z)^{-1}((A+z)^{-1}-(A+\mu)^{-1})d\mu\Big)dz\\
&=&\frac{1}{(2\pi i)^{2}}\int_{\Gamma_{\theta_{A}-\varepsilon}}\int_{-\delta+\Gamma_{\theta_{A}}}(-\mu)^{-\rho}(\mu-z)^{-1}(B+\lambda-z)^{-1}(A+z)^{-1}d\mu dz\\
&&-\frac{1}{(2\pi i)^{2}}\int_{-\delta+\Gamma_{\theta_{A}}}\int_{\Gamma_{\theta_{A}-\varepsilon}}(-\mu)^{-\rho}(\mu-z)^{-1}(B+\lambda-z)^{-1}(A+\mu)^{-1}dz d\mu,
\end{eqnarray*}
where we have used Fubini's theorem in the last term. By Cauchy's theorem, the first term on the right-hand side of the above equation is zero. Therefore, by first applying Cauchy's theorem to the second term, changing $\mu\mapsto \mu-\delta$ and then using dominated convergence theorem as $\delta\rightarrow 0$, we get 
\begin{equation*}
(A+B+\lambda)^{-1}A^{-\rho}=\frac{1}{2\pi i}\int_{\Gamma_{\theta_{A}}}(-\mu)^{-\rho}(A+\mu)^{-1}(B+\lambda-\mu)^{-1}d\mu,
\end{equation*}
which implies
\begin{eqnarray*}\label{ef}
\lefteqn{(A+B)(A+B+\lambda)^{-1}A^{-\rho}}\\
&=&\frac{1}{2\pi i}\int_{\Gamma_{\theta_{A}}}(-\mu)^{-\rho}(A(A+\mu)^{-1}(B+\lambda-\mu)^{-1}+B(B+\lambda-\mu)^{-1}(A+\mu)^{-1})d\mu,
\end{eqnarray*}
since the last integral converges absolutely. By changing variables $\mu=|\lambda|z$, we obtain
\begin{eqnarray*}\label{ef}
\lefteqn{(A+B)(A+B+\lambda)^{-1}A^{-\rho}}\\
&=&\frac{|\lambda|^{1-\rho}}{2\pi i}\int_{\Gamma_{\theta_{A}}}(-z)^{-\rho}(A(A+|\lambda|z)^{-1}(B+\lambda-|\lambda|z)^{-1}\\
&&+B(B+\lambda-|\lambda|z)^{-1}(A+|\lambda|z)^{-1})dz,
\end{eqnarray*}
and by Cauchy's theorem we deduce
\begin{eqnarray*}\label{ef}
\lefteqn{(A+B)(A+B+\lambda)^{-1}A^{-\rho}}\\
&=&\frac{|\lambda|^{1-\rho}}{2\pi i}\int_{-\Gamma_{r,\pi-\theta_{A}-\varepsilon}}(-z)^{-\rho}(A(A+|\lambda|z)^{-1}(B+\lambda-|\lambda|z)^{-1}\\
&&+B(B+\lambda-|\lambda|z)^{-1}(A+|\lambda|z)^{-1})dz,
\end{eqnarray*}
for some fixed sufficiently small $r>0$ independent of $\lambda$. For such $r$, due to $\theta_{A}+\theta_{B}>\pi$, $\lambda-|\lambda|z$ always belongs to a sector slightly larger than $S_{\theta_{B}}$, where, due to Remark \ref{secextnd}, $B$ is still sectorial. Therefore, by the sectoriality of $A$ in $S_{\theta_{A}+\varepsilon}$ and the sectoriality of $B$ in a sector slightly larger than $S_{\theta_{B}}$, we estimate
\begin{eqnarray*}\label{ef}
\lefteqn{\|(A+B)(A+B+\lambda)^{-1}A^{-\rho}\|_{\mathcal{L}(X_{0})}}\\
&\leq&\frac{|\lambda|^{-\rho}}{2\pi}\int_{-\Gamma_{r,\pi-\theta_{A}-\varepsilon}}|z|^{-\rho}\Big(\frac{K_{1}}{|\lambda|^{-1}+|\frac{\lambda}{|\lambda|}-z|}+\frac{K_{2}}{|\lambda|^{-1}+|z|}\Big)dz,
\end{eqnarray*}
for some $K_{1}, K_{2}>0$ that only depend on the sectorial bounds of $A$ and $B$. The result follows since the above integral is uniformly bounded in $\lambda\in S_{\theta_{A+B}}$ with $|\lambda|\geq1$.
 \end{proof}
 
If $\eta\in(0,1)$ and $p\in(1,\infty)$, denote by $(\cdot,\cdot)_{\eta,p}$ the real interpolation functor of exponent $\eta$ and parameter $p$.

{\bf Proof of Theorem \ref{t1}.} Let $E_{0}=L^{q}(\mathbb{K};X_{0})$ and $E_{1}=L^{q}(\mathbb{K};X_{1})$.
For any fixed $z\in\mathbb{K}$ denote again by $B(z)$ the natural extension of $B(z)\in \mathcal{L}(X_{1},X_{0})$ in $E_{0}$ with domain $E_{1}$. By the definition of $\mathcal{H}^{\infty}$-calculus combined with Kahane's inequality \cite[Theorem 2.4]{KW1}, each $B(z)+c$ belongs again to $\mathcal{RH}^{\infty}(\theta_{B})$ or to $\mathcal{H}^{\infty}(\theta_{B})$, depending on whether {\bf (i)} or {\bf (ii)} is satisfied. We split the proof in several steps.\\
{\em Step 1 {\em(}Freezing the coefficients{\em)}}. Fix $z\in\mathbb{K}$ and consider the operator $A+B(z)$ with domain $\mathcal{D}(A)\cap E_{1}$ in $E_{0}$. Since $A$ and $B(z)$ are resolvent commuting, see e.g. \cite[(III.4.9.1)]{Am} for the definition of resolvent commuting operators, and the parabolicity condition $\theta_{A}+\theta_{B}>\pi$ holds, by \cite[Theorem 3.1]{PS} the sum $A+B(z)+c$ is closed, invertible sectorial and moreover satisfies $A+B(z)+c\in\mathcal{H}^{\infty}(\theta_{A+B})$. In particular, $A+B(z)+c\in\mathcal{R}(\theta_{A+B})$ due to \eqref{sectinclutions}.\\
{\em Step 2 {\em (}Uniform boundedness of the spectral shifts{\em)}}. For any $z\in\mathbb{K}$ denote
$$
\underline{c}_{z}=\inf\{c\geq0 \, |\, \text{$A+B(z)+c\in\mathcal{R}(\theta_{A+B})$}\}.
$$
Assume that there exists some sequence $\{w_{j}\}_{j\in\mathbb{N}}$ in $\mathbb{K}$ such that $\underline{c}_{w_{j}}\rightarrow\infty$ as $j\rightarrow\infty$. By the assumption, $\{w_{j}\}_{j\in\mathbb{N}}$ lies in a compact set, so that, by possibly passing to a subsequence, we can assume that $w_{j}\rightarrow w\in\mathbb{K}$ as $j\rightarrow\infty$. Let $\tilde{c}\geq0$ such that $A+B(w)+\tilde{c}\in\mathcal{R}(\theta_{A+B})$ and choose $\delta>0$ such that $d(z,w)<\delta$ implies
$$
\|(B(z)-B(w))(A+B(w)+\tilde{c})^{-1}\|_{\mathcal{L}(E_{0})}<\frac{1}{2}\min\Big\{\frac{1}{1+K_{w}},\frac{1}{1+R_{w}}\Big\},
$$
where $d(\cdot,\cdot)$ is the geodesic distance on $\mathbb{K}$ and $K_{w}$, $R_{w}$ is a sectorial and an $R$-sectorial bound of $A+B(w)+\tilde{c}$, respectively. By the formula 
\begin{eqnarray*}
\lefteqn{(A+B(z)+\tilde{c}+\lambda)^{-1}}\\
&=&(A+B(w)+\tilde{c}+\lambda)^{-1}\sum_{k=0}^{\infty}\Big((B(w)-B(z))(A+B(w)+\tilde{c}+\lambda)^{-1}\Big)^{k},
\end{eqnarray*}
valid for $z$ as above and for $\lambda\in S_{\theta_{A+B}}$, we deduce that $S_{\theta_{A+B}}\subset \rho(-(A+B(z)+\tilde{c}))$ and $A+B(z)+\tilde{c}\in\mathcal{P}(2K_{w},\theta_{A+B})\cup \mathcal{R}(2R_{w},\theta_{A+B})$. Hence, we get a contradiction. Moreover, see e.g. \cite[Lemma 2.6]{RS2}, if for an operator $T$ in a Banach space $E$ we have $T\in \mathcal{R}(\beta,\theta)$, for some $\beta\geq1$ and $\theta\in[0,\pi)$, then $T+\tau\in \mathcal{R}(\widetilde{\beta},\theta)$ for any $\tau\geq0$, where
\begin{equation}\label{Stheta}
\widetilde{\beta}=\beta\Big(1+\frac{2}{S(\theta)}\Big)\quad \text{and} \quad S(\theta)=\Big\{\begin{array}{lll} \sin(\theta) & \text{if} &\theta\in[\pi/2,\pi)\\ 1 & \text{if} & \theta\in [0,\pi/2).\end{array}
\end{equation}
By the above we conclude that there exists some $\underline{c}\geq0$ with the following property: for any $c_{1}\geq \underline{c}$ we have that $A+B(z)+c_{1}\in\mathcal{R}(\theta_{A+B})$ for all $z\in\mathbb{K}$.\\
{\em Step 3 {\em(}Uniform boundedness of the $R$-sectorial bounds{\em)}}. For any $z\in\mathbb{K}$ denote by $\underline{R}_{z}\geq1$ the infimum of all possible $R$-sectorial bounds of $A+B(z)+\underline{c}$ with respect to the angle $\theta_{A+B}$. Assume that there exists some sequence $\{\xi_{j}\}_{j\in\mathbb{N}}$ in $\mathbb{K}$ such that $\underline{R}_{\xi_{j}}\rightarrow\infty$ as $j\rightarrow\infty$. By the assumption, $\{\xi_{j}\}_{j\in\mathbb{N}}$ lies in a compact set, so that, by possibly passing to a subsequence, we may assume that $\xi_{j}\rightarrow \xi\in\mathbb{K}$ as $j\rightarrow\infty$. Let $\varepsilon>0$ sufficiently small such that $ d(z,\xi)<\varepsilon$ implies
$$
\|B(z)-B(\xi)\|_{\mathcal{L}(E_{1},E_{0})}<\frac{1}{2(1+\underline{R}_{\xi})\|(A+B(\xi)+\underline{c})^{-1}\|_{\mathcal{L}(E_{0},E_{1})}}.
$$
Then, by the formula
\begin{eqnarray*}
\lefteqn{(A+B(z)+\underline{c}+\lambda)^{-1}}\\
&=&(A+B(\xi)+\underline{c}+\lambda)^{-1}\sum_{k=0}^{\infty}\Big((B(\xi)-B(z))(A+B(\xi)+\underline{c}+\lambda)^{-1}\Big)^{k},
\end{eqnarray*}
valid for $z$ as above and for $\lambda\in S_{\theta_{A+B}}$, or equivalently, by the formula $A+B(z)+\underline{c}=A+B(\xi)+\underline{c}+B(z)-B(\xi)$ and the $R$-sectoriality perturbation result \cite[Theorem 1]{KW}, we find that $A+B(z)+\underline{c}\in \mathcal{R}(2\underline{R}_{\xi},\theta_{A+B})$ when $d(z,\xi)<\varepsilon$. This provides us a contradiction. By taking into account \eqref{Stheta}, we conclude that there exists some $K_{0}\geq1$ such that $A+B(z)+c_{1}\in\mathcal{R}(K_{0},\theta_{A+B})$ for all $c_{1}\geq \underline{c}$ and $z\in\mathbb{K}$. \\
{\em Step 4 {\em(}$R$-sectorial local approximations{\em)}}. By a contradiction argument similar to that one in Step 2, i.e. by using the continuity of $B(\cdot)$, the fact that $B(\cdot)$ is constant on $\mathbb{R}^{n}\backslash\mathcal{K}$ when $\mathbb{K}=\mathbb{R}^{n}$ and the elementary perturbation $B(z)=B(\tilde{z})+B(z)-B(\tilde{z})$, $z,\tilde{z}\in\mathbb{K}$, we can show that there exists a $c_{2}\geq\underline{c}$ such that the norm 
\begin{equation}\label{betanorm}
\|(B(z)+c_{2})^{-1}\|_{\mathcal{L}(E_{0},E_{1})} 
\end{equation}
is uniformly bounded in $z\in\mathbb{K}$. Moreover, recall that for fixed $z\in\mathbb{K}$, closedness and invertibility of the operator 
$$
A+B(z)+c_{2} : \mathcal{D}(A)\cap E_{1} \rightarrow E_{0}
$$
also follow by the result of Kalton and Weis \cite[Theorem 6.3]{KaW}. Furthermore, by Step 3, the $R$-sectorial bounds of $B(z)+c_{2}$ can be chosen uniformly bounded in $z\in\mathbb{K}$. Thus, by \cite[Theorem 2.1]{DH}, the norm 
\begin{equation}\label{unibound1}
\|(B(z)+c_{2})(A+B(z)+c_{2})^{-1}\|_{\mathcal{L}(E_{0})}
\end{equation}
is also uniformly bounded in $z\in\mathbb{K}$. \\
For any $z\in\mathbb{K}$ and $r>0$, denote by $U_{r}(z)$ the open geodesic ball in $\mathbb{K}$ of radius $r$, centered at $z$. Consider two open covers $\{U_{j}\}_{j\in\{1,\dots,N\}}$, $\{V_{j}\}_{j\in\{1,\dots,N\}}$, $N\in\mathbb{N}$, $N\geq2$, of $\mathbb{K}$ as follows:\\
When $\mathbb{K}=\mathbb{V}$: $U_{j}=U_{r}(z_{j})$, $V_{j}=U_{2r}(z_{j})$, $j\in\{1,\dots,N\}$, for certain $z_{1},\dots,z_{N}\in\mathbb{V}$ and $r>0$. \\
When $\mathbb{K}=\mathbb{R}^{n}$: $U_{1}=\mathbb{R}^{n}\backslash \overline{U_{2r_{0}}(0)}$, $V_{1}=\mathbb{R}^{n}\backslash \overline{U_{r_{0}}(0)}$ and $U_{j}=U_{r}(z_{j})$, $V_{j}=U_{2r}(z_{j})$, $j\in\{2,\dots,N\}$, for certain $z_{2},\dots,z_{N}\in\mathbb{R}^{n}$ and $r_{0},r>0$. Moreover, choose $r_{0}$ sufficiently large such that $\mathcal{K}\subset U_{r_{0}}(0)$ and take $z_{1}\in \mathbb{R}^{n}\backslash \mathcal{K}$ arbitrary. 

Let $\{\chi_{j}\}_{j\in\{1,\dots,N\}}$ be a collection of smooth positive functions bounded by one such that $\chi_{j}=1$ on $U_{j}$ and $\chi_{j}=0$ outside $V_{j}$, $j\in\{1,\dots,N\}$. For each $j\in\{1,\dots,N\}$ define 
$$
A_{j}=\chi_{j}(A+B+c_{2})+(1-\chi_{j})(A+B(z_{j})+c_{2})=A+B(z_{j})+c_{2}+\chi_{j}(B-B(z_{j}))
$$
with domain $\mathcal{D}(A)\cap E_{1}$ in $E_{0}$. Due to the uniform continuity of $B(\cdot)$, for any $\varepsilon>0$ we can choose $r$ sufficiently small, and possibly increase $N$, such that 
\begin{equation}\label{e1toe0normsmall}
\|\chi_{j}(B-B(z_{j}))\|_{\mathcal{L}(E_{1},E_{0})}\leq \sup_{z\in\mathbb{K}}\|\chi_{j}(z)(B(z)-B(z_{j}))\|_{\mathcal{L}(X_{1},X_{0})}<\varepsilon
\end{equation}
for each $j\in\{1,\dots,N\}$. Then, by choosing $\varepsilon$ sufficiently small and using the formula
\begin{eqnarray}\nonumber
\lefteqn{(A_{j}+\lambda)^{-1}=(A+B(z_{j})+c_{2}+\lambda)^{-1}}\\\label{kkk}
&&\times\sum_{k=0}^{\infty}\Big(\chi_{j}(B(z_{j})-B)(B(z_{j})+c_{2})^{-1}(B(z_{j})+c_{2})(A+B(z_{j})+c_{2}+\lambda)^{-1}\Big)^{k},
\end{eqnarray}
valid for any $\lambda\in S_{\theta_{A+B}}$, which holds due to Step 3, \eqref{betanorm} and \eqref{unibound1}, we get that $A_{j}\in\mathcal{P}(\theta_{A+B})$ for each $j\in\{1,\dots,N\}$. \\
Furthermore, $A_{j}\in\mathcal{R}(\theta_{A+B})$ for each $j\in\{1,\dots,N\}$. Indeed, let $\lambda_{1},\dots,\lambda_{\eta}\in S_{\theta_{A+B}}\backslash\{0\}$, $u_{1},\dots,u_{\eta}\in E_{0}$, $\eta\in\mathbb{N}$, and let $\{\epsilon_{k}\}_{k\in\mathbb{N}}$ be the sequence of Rademacher functions. Given $\varepsilon>0$ choose $r>0$ sufficiently small, and possibly $N$ large enough, such that, due to \eqref{betanorm} and \eqref{e1toe0normsmall}, we have
$$
\|\chi_{j}(B-B(z_{j}))(B(z_{j})+c_{2})^{-1}\|_{\mathcal{L}(E_{0})}<\varepsilon
$$
for all $ j\in\{1,\dots,N\}$. Then, by \eqref{kkk} and the $R$-sectoriality of each $A+B(z_{j})+c_{2}$, we estimate
\begin{eqnarray*}
\lefteqn{\|\sum_{\ell=1}^{\eta}\epsilon_{\ell}\lambda_{\ell}(A_{j}+\lambda_{\ell})^{-1}u_{\ell}\|_{L^{2}(0,1;E_{0})}}\\
&\leq&\|\sum_{\ell=1}^{\eta}\epsilon_{\ell}\lambda_{\ell}(A+B(z_{j})+c_{2}+\lambda_{\ell})^{-1}\\
&&\times\sum_{k=0}^{\infty}\Big(\chi_{j}(B(z_{j})-B)(B(z_{j})+c_{2})^{-1}(B(z_{j})+c_{2})(A+B(z_{j})+c_{2})^{-1}\\
&&\times(A+B(z_{j})+c_{2})(A+B(z_{j})+c_{2}+\lambda_{\ell})^{-1}\Big)^{k}u_{\ell}\|_{L^{2}(0,1;E_{0})}\\\nonumber
&\leq&K_{0}\sum_{k=0}^{\infty}\|\sum_{\ell=1}^{\eta}\epsilon_{\ell}\Big(\chi_{j}(B(z_{j})-B)(B(z_{j})+c_{2})^{-1}(B(z_{j})+c_{2})(A+B(z_{j})+c_{2})^{-1}\\
&& \times(A+B(z_{j})+c_{2})(A+B(z_{j})+c_{2}+\lambda_{\ell})^{-1}\Big)^{k}u_{\ell}\|_{L^{2}(0,1;E_{0})}\\
&\leq&K_{0}\Big(\sum_{k=0}^{\infty}(\varepsilon C_{0}(K_{0}+1))^{k}\Big)\|\sum_{\ell=1}^{\eta}\epsilon_{\ell}u_{\ell}\|_{L^{2}(0,1;E_{0})},
\end{eqnarray*}
where
$$
C_{0}=\max_{j\in\{1,\dots,N\}}\|(B(z_{j})+c_{2})(A+B(z_{j})+c_{2})^{-1}\|_{\mathcal{L}(E_{0})},
$$
and we have taken $\varepsilon>0$ sufficiently small.

{\em Step 5 {\em(}Right inverse{\em)}}. Let $\{\phi_{j}\}_{j\in\{1,\dots,N\}}$ be a partition of unity subordinate to the covering $\{U_{j}\}_{j\in\{1,\dots,N\}}$ of $\mathbb{K}$. Then, by using
$$
\phi_{j}(A+B+c_{2}+\lambda)(A_{j}+\lambda)^{-1}=\phi_{j}(A_{j}+\lambda)(A_{j}+\lambda)^{-1}=\phi_{j},
$$
$j\in\{1,\dots,N\}$, $\lambda\in S_{\theta_{A+B}}$, together with the identity $I=\sum_{j=1}^{N}\phi_{j}$, we deduce that 
\begin{equation}\label{identidecomp}
I=\sum_{j=1}^{N}\phi_{j}(A+B+c_{2}+\lambda)(A_{j}+\lambda)^{-1}.
\end{equation}
Recall that multiplication by $\phi_{j}$ induces a bounded map on $\mathcal{D}(A)\cap E_{1}$. Furthermore, the commutator $[A+B+c_{2}+\lambda,\phi_{j}]=[A,\phi_{j}]$ is a ($\mu-1$)-order differential operator with $W^{1,\infty}(\mathbb{K};\mathbb{C})$-coefficients, and hence it induces a bounded map from $\mathcal{D}(A)\cap E_{1}$ to $E_{0}$. Thus, by \eqref{identidecomp} we find
$$
I=\sum_{j=1}^{N}\big([\phi_{j},A+B+c_{2}+\lambda]+(A+B+c_{2}+\lambda)\phi_{j}\big)(A_{j}+\lambda)^{-1},
$$
i.e.
\begin{equation}\label{erd}
I+\sum_{j=1}^{N}[A,\phi_{j}](A_{j}+\lambda)^{-1}=(A+B+c_{2}+\lambda)\sum_{j=1}^{N}\phi_{j}(A_{j}+\lambda)^{-1}.
\end{equation}
Since each $A_{j}$ is sectorial, by \cite[(I.2.5.2) and (I.2.9.6)]{Am} and \cite[Theorem 5.6.9]{HNVW}, for any $\vartheta\in(0,1)$ and $\psi\in(0,\vartheta)$ we have
\begin{eqnarray*}
\mathcal{D}(A)\cap E_{1}\hookrightarrow \mathcal{D}(A_{j}^{\vartheta})\hookrightarrow(E_{0},\mathcal{D}(A)\cap E_{1})_{\psi,q}\hookrightarrow(E_{0},\mathcal{D}(A))_{\psi,q}\hookrightarrow H_{q}^{\psi\mu-\varepsilon}(\mathbb{K};X_{0}),
\end{eqnarray*}
for all $\varepsilon>0$ sufficiently small. Note that, in the case of $\mathbb{K}=\mathbb{V}$, by choosing a covering of $\mathbb{V}$ by coordinate charts and then taking a subordinate partition of unity, \cite[Theorem 5.6.9]{HNVW} extends to vector valued Bessel potential spaces on $\mathbb{V}$. Hence, by choosing $(\mu-1)/\mu<\psi<\vartheta<1$ we find that $[A,\phi_{j}]A_{j}^{-\vartheta}\in\mathcal{L}(E_{0})$, for each $j\in\{1,\dots,N\}$. Thus, by Remark \ref{TaLnLem}, each of 
\begin{equation}\label{intermdecay}
\|[A,\phi_{j}](A_{j}+c_{3}+\lambda)^{-1}\|_{\mathcal{L}(E_{0})} =\|[A,\phi_{j}]A_{j}^{-\vartheta}A_{j}^{\vartheta}(A_{j}+c_{3}+\lambda)^{-1}\|_{\mathcal{L}(E_{0})} 
\end{equation}
can become arbitrary small, uniformly in $\lambda\in S_{\theta_{A+B}}$, by choosing $c_{3}>0$ sufficiently large. Therefore, for such $c_{3}>0$, by \eqref{erd} we get 
\begin{equation}\label{erd2}
(A+B+c_{4}+\lambda)\Big(\sum_{j=1}^{N}\phi_{j}(A_{j}+c_{3}+\lambda)^{-1}\Big)\Big(I+\sum_{j=1}^{N}[A,\phi_{j}](A_{j}+c_{3}+\lambda)^{-1}\Big)^{-1}=I,
\end{equation}
in $E_{0}$ for all $\lambda\in S_{\theta_{A+B}}$, where $c_{4}=c_{2}+c_{3}$, i.e. $A+B+c_{4}+\lambda$ has a right inverse that belongs to $\mathcal{L}(E_{0},\mathcal{D}(A)\cap E_{1})$.\\
{\em Step 6 {\em(}Left inverse{\em)}}. Take $\lambda\in S_{\theta_{A+B}}$, $u\in\mathcal{D}(A)\cap E_{1}$, $f\in E_{0}$ and consider the equation
$$
(A+B+c_{4}+\lambda)u=f.
$$ 
By multiplying with $\phi_{j}$, $j\in\{1,\dots,N\}$, we get 
$$
(A+B+c_{4}+\lambda)\phi_{j}u=\phi_{j}f+[A+B+c_{4}+\lambda,\phi_{j}]u=\phi_{j}f+[A,\phi_{j}]u.
$$
Note that $A+B+c_{4}+\lambda=A_{j}+c_{3}+\lambda+(1-\chi_{j})(B-B(z_{j}))$ and $(1-\chi_{j})\phi_{j}=0$. Hence, applying the resolvent $(A_{j}+c_{3}+\lambda)^{-1}$ we obtain 
$$
\phi_{j}u=(A_{j}+c_{3}+\lambda)^{-1}(\phi_{j}f+[A,\phi_{j}]u),
$$
where by summing up we find 
\begin{equation}\label{llk}
u=\sum_{j=1}^{N}(A_{j}+c_{3}+\lambda)^{-1}\phi_{j}f+\sum_{j=1}^{N}(A_{j}+c_{3}+\lambda)^{-1}[A,\phi_{j}]u.
\end{equation}

Each of the commutators $[A,\phi_{j}]$ induces a bounded map from $\mathcal{D}(A)\cap E_{1}$ to $H_{q}^{1}(\mathbb{K};X_{0})$. Furthermore, by \cite[(I.2.5.2) and (I.2.9.6)]{Am} and \cite[Theorem 5.6.9]{HNVW}, we have that $H_{q}^{1}(\mathbb{K};X_{0})\hookrightarrow \mathcal{D}(A_{j}^{\rho})$, for any $\rho\in(0,1/\mu)$ and $j\in\{1,\dots,N\}$. Hence, $A_{j}^{\rho}[A,\phi_{j}]\in\mathcal{L}(\mathcal{D}(A)\cap E_{1},E_{0})$ for each $j\in\{1,\dots,N\}$. In addition, by \eqref{kkk} we have
\begin{eqnarray*}
\lefteqn{(A+B(z_{j})+c_{2})(A_{j}+c_{3}+\lambda)^{-1}A_{j}^{-\rho}}\\
&=&(A+B(z_{j})+c_{2})(A+B(z_{j})+c_{4}+\lambda)^{-1}A_{j}^{-\rho}\\
&&+(A+B(z_{j})+c_{2})(A+B(z_{j})+c_{4}+\lambda)^{-1}\\
&&\times\Big(\sum_{k=0}^{\infty}\Big(\chi_{j}(B(z_{j})-B)(A+B(z_{j})+c_{4}+\lambda)^{-1}\Big)^{k}\Big)\\
&&\times\chi_{j}(B(z_{j})-B)(A+B(z_{j})+c_{2})^{-1}(A+B(z_{j})+c_{2})(A+B(z_{j})+c_{4}+\lambda)^{-1}A_{j}^{-\rho}.
\end{eqnarray*}
Thus, by Lemma \ref{Lem1} we find that each of the norms
$$
\|\sum_{j=1}^{N}(A_{j}+c_{3}+\lambda)^{-1}A_{j}^{-\rho}A_{j}^{\rho}[A,\phi_{j}]\|_{\mathcal{L}(\mathcal{D}(A)\cap E_{1})}
$$
becomes arbitrary small, uniformly in $\lambda$, by taking $c_{3}$ sufficiently large. By \eqref{llk} we conclude that for $c_{4}$ large enough the operator $A+B+c_{4}+\lambda$ admits a left inverse which belongs to $\mathcal{L}(E_{0},\mathcal{D}(A)\cap E_{1})$.\\
{\em Step 7 {\em(}$R$-sectoriality{\em)}}. Due to \eqref{erd2}, the resolvent of $A+B+c_{4}$ is given by
$$
(A+B+c_{4}+\lambda)^{-1}=\Big(\sum_{j=1}^{N}\phi_{j}(A_{j}+c_{3}+\lambda)^{-1}\Big)\Big(I+\sum_{j=1}^{N}[A,\phi_{j}](A_{j}+c_{3}+\lambda)^{-1}\Big)^{-1},
$$
for all $\lambda\in S_{\theta_{A+B}}$. Using this expression, $R$-sectoriality for $A+B+c_{4}$ will follow by the $R$-sectoriality of each $A_{j}+c_{3}\in\mathcal{R}(K_{j},\theta_{A+B})$, $j\in\{1,\dots,N\}$. More precisely
\begin{eqnarray}\nonumber
\lefteqn{\|\sum_{\ell=1}^{\eta}\epsilon_{\ell}\lambda_{\ell}(A+B+c_{4}+\lambda_{\ell})^{-1}u_{\ell}\|_{L^{2}(0,1;E_{0})}}\\\nonumber
&\leq& C_{1}\sum_{j=1}^{N}\|\sum_{\ell=1}^{\eta}\epsilon_{\ell}\lambda_{\ell}(A_{j}+c_{3}+\lambda_{\ell})^{-1}\\\nonumber
&&\times\Big(\sum_{k=0}^{\infty}\Big(\sum_{i=1}^{N}[\phi_{i},A](A_{i}+c_{3}+\lambda_{\ell})^{-1}\Big)^{k}\Big)u_{\ell}\|_{L^{2}(0,1;E_{0})}\\\nonumber
&\leq& C_{1}NK\|\sum_{\ell=1}^{\eta}\epsilon_{\ell}\Big(\sum_{k=0}^{\infty}\Big(\sum_{j=1}^{N}[\phi_{j},A](A_{j}+c_{3}+\lambda_{\ell})^{-1}\Big)^{k}\Big)u_{\ell}\|_{L^{2}(0,1;E_{0})}\\\label{finalrsec}
&\leq& C_{1}NK\sum_{k=0}^{\infty}\|\sum_{\ell=1}^{\eta}\epsilon_{\ell}\Big(\sum_{j=1}^{N}[A,\phi_{j}](A_{j}+c_{3}+\lambda_{\ell})^{-1}\Big)^{k}u_{\ell}\|_{L^{2}(0,1;E_{0})},
\end{eqnarray}
where
$$
C_{1}=\max_{j\in\{1,\dots,N\}}\|\phi_{j}\cdot\|_{\mathcal{L}(E_{0})} \quad \text{and} \quad K= \max_{j\in\{1,\dots,N\}}K_{j}.
$$
For $k\geq1$, the $k$-th term on the right-hand side of \eqref{finalrsec} is estimated by a sum of $N^{k}$ terms the form
\begin{equation}\label{extraterm}
\|\sum_{\ell=1}^{\eta}\epsilon_{\ell}[A,\phi_{j_{1}}](A_{j_{1}}+c_{3}+\lambda_{\ell})^{-1}\dots [A,\phi_{j_{k}}](A_{j_{k}}+c_{3}+\lambda_{\ell})^{-1}u_{\ell}\|_{L^{2}(0,1;E_{0})},
\end{equation}
where $j_{i}\in\{1,\dots,N\}$, $i\in\{1,\dots,k\}$. Similarly to \eqref{intermdecay} we have
\begin{eqnarray*}
\lefteqn{[A,\phi_{j_{i}}](A_{j_{i}}+c_{3}+\lambda_{\ell})^{-1}}\\
&=&[A,\phi_{j_{i}}]A^{-\vartheta}A^{\vartheta}(A_{j_{i}}+c_{3})^{-1}(A_{j_{i}}+c_{3})(A_{j_{i}}+c_{3}+\lambda_{\ell})^{-1}, \quad i\in\{1,\dots,k\}.
\end{eqnarray*}
Moreover, for any $\varepsilon>0$ choose $c_{3}>0$ large enough such that
$$
\max_{j\in\{1,\dots,N\}}\|[A,\phi_{j}]A^{-\vartheta}A^{\vartheta}(A_{j}+c_{3})^{-1}\|_{\mathcal{L}(E_{0})}<\varepsilon,
$$
where we have used Remark \ref{TaLnLem}. Now \eqref{extraterm}, due to \eqref{Stheta}, can be estimated by 
$$
\Big(\varepsilon K\Big(1+\frac{2}{S(\theta_{A+B})}\Big)\Big)^{k}\|\sum_{\ell=1}^{\eta}\epsilon_{\ell}u_{\ell}\|_{L^{2}(0,1;E_{0})},
$$
Since $\varepsilon>0$ is arbitrary, $R$-sectoriality for $A+B+c_{4}$, for sufficiently large values of $c_{4}$, follows by the above estimate and \eqref{finalrsec}. \mbox{\ } \hfill $\Box$

\begin{remark} {\em {\bf(i)} In Theorem \ref{t1}, when $\mathbb{K}=\mathbb{R}^{n}$, the condition of $B(\cdot)$ being constant on $\mathbb{R}^{n}\backslash \mathcal{K}$ can be relaxed to certain decay property of $B(\cdot)$, after modifying the above proof. In particular, the theorem holds true when $B(z)=B_{0}+B_{1}(z)$, where $B_{0}, B_{1}(z) \in \mathcal{L}(X_{1},X_{0})$, $z\in\mathbb{R}^{n}$, $B_{1}(\cdot)\in C(\mathbb{R}^{n};\mathcal{L}(X_{1},X_{0}))$ and $\|B_{1}(z)\|_{\mathcal{L}(X_{1},X_{0})}\rightarrow 0$ when $z\rightarrow\infty$.\\
{\bf(ii)} Theorem \ref{t1}, as a non-commuting version of the classical Kalton-Weis result \cite[Theorem 6.3]{KaW}, can be compared with \cite[Theorem 3.1]{PS}. Since it concerns Bochner spaces, it is less general than \cite[Theorem 3.1]{PS}, but, on the other hand, it is free of the Da Prato-Grisvard and Labbas-Terreni commutation assumption \cite[(3.1)]{PS} and \cite[(3.2)]{PS}, respectively.}
\end{remark}

Concerning the domain of the perturbed operator in Theorem \ref{t1}, we show the following result, which is also for later use. Denote by $[\cdot,\cdot]_{\eta}$, $\eta\in(0,1)$, the complex interpolation functor of exponent $\eta$. 

\begin{proposition}[Mixed derivative]\label{mxtderPr}
Let $n\in\mathbb{N}$, $q\in(1,\infty)$, $r\geq0$ and $\rho\in(0,1)$. Moreover, let $X_{1}\hookrightarrow X_{0}$ be a densely and continuously injected complex Banach couple, such that $X_{0}$ and $X_{1}$ are UMD. Then
\begin{eqnarray*}
H_{q}^{r+2}(\mathbb{R}^{n};X_{0})\cap H_{q}^{r}(\mathbb{R}^{n};X_{1}) &\hookrightarrow& [H_{q}^{r}(\mathbb{R}^{n};[X_{0},X_{1}]_{\rho}),H_{q}^{r+2}(\mathbb{R}^{n};[X_{0},X_{1}]_{\rho})]_{1-\rho}\\
(\text{with equivalent norms}) &=& H_{q}^{r+2-2\rho}(\mathbb{R}^{n};[X_{0},X_{1}]_{\rho}).
\end{eqnarray*}
\end{proposition}
\begin{proof}
Since the UMD property is preserved under interpolation (see e.g. \cite[Theorem III.4.5.2]{Am}), by \cite[(I.2.5.2)]{Am} and \cite[Theorem 14.7.12]{HNVW2}, we have
\begin{eqnarray*}
H_{q}^{r+2}(\mathbb{R}^{n};X_{0})\cap H_{q}^{r}(\mathbb{R}^{n};X_{1}) &\hookrightarrow& [H_{q}^{r+2}(\mathbb{R}^{n};X_{0}),H_{q}^{r}(\mathbb{R}^{n};X_{1})]_{\rho}\\
&=& H_{q}^{r+2-2\rho}(\mathbb{R}^{n};[X_{0},X_{1}]_{\rho})\\
&=& [H_{q}^{r}(\mathbb{R}^{n};[X_{0},X_{1}]_{\rho}),H_{q}^{r+2}(\mathbb{R}^{n};[X_{0},X_{1}]_{\rho})]_{1-\rho}.
\end{eqnarray*}
\end{proof}

We conclude the linear theory with the following elementary result.

\begin{lemma}\label{suminter}
Let $X$, $Y$, $Z$ be Banach spaces, all continuously embedded in the same Hausdorff topological vector space, such that $Y\hookrightarrow X$ and $X\cap Z=Y\cap Z=\{0\}$. For any $\theta\in(0,1)$ and $p\in (1,\infty)$ we have $\{X\oplus Z,Y\oplus Z\}=\{X,Y\}\oplus Z$, where $\{\cdot,\cdot\}\in \{(\cdot,\cdot)_{\theta,p}, [\cdot,\cdot]_{\theta}\}$.
\end{lemma}
\begin{proof}
Clearly, $\{X,Y\}\hookrightarrow \{X\oplus Z,Y\oplus Z\}$ and $Z\hookrightarrow \{X\oplus Z,Y\oplus Z\}$, so that
$$
\{X,Y\}\oplus Z\hookrightarrow \{X\oplus Z,Y\oplus Z\}.
$$
For the other direction, let $X\oplus Z\ni u+v\mapsto P(u+v)=u\in X$ be the natural projection. Since $P$ restricts to a bounded map $P:Y\oplus Z\rightarrow Y$, by \cite[Theorem 1.6 and Theorem 2.6]{Lunar18} we have that $P$ also defines a bounded map $P:\{X\oplus Z,Y\oplus Z\}\rightarrow \{X,Y\}$. Similarly, $I-P$ defines a bounded map $I-P:\{X\oplus Z,Y\oplus Z\}\rightarrow \{Z,Z\}=Z$. Hence, by writing any $w\in \{X\oplus Z,Y\oplus Z\}\hookrightarrow X\oplus Z$ as $Pw+(I-P)w$, we obtain 
$$
\{X\oplus Z,Y\oplus Z\}\hookrightarrow \{X,Y\}\oplus Z.
$$
\end{proof}

We describe now an abstract maximal $L^{d}$-regularity result for quasilinear parabolic equations. Let $d\in(1,\infty)$, $U$ be an open subset of $(X_{1},X_{0})_{1/d,d}$, $A(\cdot): U\rightarrow \mathcal{L}(X_{1},X_{0})$ and $F(\cdot,\cdot): U\times [0,T_{0}]\rightarrow X_{0}$, for some $T_{0}>0$. Consider the problem
\begin{eqnarray}\label{aqpp1}
u'(t)+A(u(t))u(t)&=&F(u(t),t)+G(t),\quad t\in(0,T),\\\label{aqpp2}
u(0)&=&u_{0},
\end{eqnarray}
where $T\in(0,T_{0}]$, $u_{0}\in U$ and $G\in L^{d}(0,T_{0};X_{0})$. If we denote by $C^{1-}(\cdot;\cdot)$ or $C^{1-,1-}(\cdot;\cdot)$ the space of Lipschitz continuous maps, then the following short time existence result holds true.
\begin{theorem}[{\rm Cl\'ement and Li, \cite[Theorem 2.1]{CL}}]\label{ClementLi}
Assume that:\\
{\em (H1)} $A(\cdot)\in C^{1-}(U;\mathcal{L}(X_{1},X_{0}))$.\\
{\em (H2)} $F(\cdot,\cdot)\in C^{1-,1-}(U\times [0,T_{0}];X_{0})$.\\
{\em (H3)} $A(u_{0})$ has maximal $L^{d}$-regularity.\\
Then, there exists a $T\in(0,T_{0}]$ and a unique $u\in H_{d}^{1}(0,T;X_{0})\cap L^{d}(0,T;X_{1})$ solving \eqref{aqpp1}-\eqref{aqpp2}. 
\end{theorem}

We close this section with the following embedding of the maximal $L^q$-regularity space, namely
\begin{equation}\label{interpemb}
H_{d}^{1}(0,T;X_{0})\cap L^{d}(0,T;X_{1})\hookrightarrow C([0,T];(X_{1},X_{0})_{\frac{1}{d},d}), \quad d\in(1,\infty), \, T>0,
\end{equation}
see e.g. \cite[Theorem III.4.10.2]{Am}.

\section{Cone differential operators and function spaces on manifolds with edges}

As explained in the introduction the Laplacian $\Delta_{\mathfrak{g}}$ is a second order cone differential operator. In this section we recall some basic facts from the related theory of cone pseudo-differential operators, following \cite{Schulze1}, \cite{Schulze2}, \cite{Schulze3} and \cite{Sei}. We also recall results concerning closed extensions and maximal regularity theory for the Laplacian; for more details and generalizations the reader is referred to \cite{CSS}, \cite{GKM}, \cite{GKM2}, \cite{GM}, \cite{Le}, \cite{RS1}, \cite{RS2}, \cite{RS}, \cite{RS3}, \cite{SS}, \cite{SS2}. Moreover, we shall discuss some embedding properties of Sobolev spaces which will be needed in the sequel.

In order to keep our ground space independent of $z\in\mathcal{V}$, we will fix the metric on $\mathcal{B}$, i.e. we will consider all operators on $\mathbb{B}=(\mathcal{B},\mathfrak{p})$. Denote by $\mathrm{Diff}^{k}(\cdot)$ the space of differential operators of order $k\in\mathbb{N}_{0}$ with smooth coefficients. A cone differential operator of order $\mu\in\mathbb{N}_{0}$ is any $\mu$-th order differential operator $A$ with smooth coefficients in the interior $\mathbb{B}^{\circ}$ of $\mathbb{B}$, such that its restriction to the collar neighborhood $(0,1)\times\partial\mathcal{B}$ admits the form 
\begin{equation}\label{Aconeww}
A=x^{-\mu}\sum_{k=0}^{\mu}a_{k}(x)(-x\partial_{x})^{k}, \quad \mbox{where} \quad a_{k}\in C^{\infty}([0,1);\mathrm{Diff}^{\mu-k}(\partial\mathbb{B})).
\end{equation}
If in addition to the usual homogeneous principal symbol $\sigma_{\psi}^{\mu}(A)\in C^{\infty}(T^{\ast}\mathbb{B}\backslash\{0\})$, the {\em rescaled principal symbol} $\widetilde{\sigma}_{\psi}^{\mu}(A)\in C^{\infty}((T^{\ast}\partial\mathbb{B}\times\mathbb{R})\backslash\{0\})$ of a cone differential operator, defined by 
$$
\widetilde{\sigma}_{\psi}^{\mu}(A)=\lim_{x\rightarrow0}x^{\mu}\sigma_{\psi}^{\mu}(A)(x,y,x^{-1}\tau,\xi),
$$
is also pointwise invertible, then the operator is called {\em cone-elliptic}; this is the case for the Laplacian $\Delta_{\mathfrak{g}}$.

Recall the fixed cut-off function $\omega$ and denote by $C_{c}^{\infty}(\cdot)$ the space of smooth compactly supported functions. Cone differential operators act naturally on scales of Mellin-Sobolev spaces.

\begin{definition}{\em (\cite[Definition 2.1.19]{Schulze2})}\label{MellSob}
Let $\gamma\in\mathbb{R}$ and consider the map 
$$
M_{\gamma}: C_{c}^{\infty}(\mathbb{R}_{+}\times\mathbb{R}^{\nu-1})\rightarrow C_{c}^{\infty}(\mathbb{R}^{\nu}) \quad \mbox{defined by} \quad u(x,y)\mapsto e^{(\gamma-\frac{\nu}{2})x}u(e^{-x},y). 
$$
Furthermore, take a covering $k_{j}:\Omega_{j}\subseteq\partial\mathcal{B} \rightarrow\mathbb{R}^{\nu-1}$, $j\in\{1,\dots,\ell\}$, $\ell\in\mathbb{N}$, of $\partial\mathcal{B}$ by coordinate charts and let $\{\varphi_{j}\}_{j\in\{1,\dots,\ell\}}$ be a subordinate partition of unity. For any $p\in(1,\infty)$ and $s\in\mathbb{R}$ let $\mathcal{H}^{s,\gamma}_p(\mathbb{B})$ be the space of all distributions $u$ on $\mathbb{B}^{\circ}$ such that 
$$
\|u\|_{\mathcal{H}^{s,\gamma}_p(\mathbb{B})}=\sum_{j=1}^{\ell}\|M_{\gamma}(1\otimes k_{j})_{\ast}(\omega\varphi_{j} u)\|_{H^{s}_{p}(\mathbb{R}^{\nu})}+\|(1-\omega)u\|_{H^{s}_{p}(\mathbb{B})}
$$
is defined and finite. The space $\mathcal{H}^{s,\gamma}_{p}(\mathbb{B})$, called {\em (weighted) Mellin-Sobolev space}, is independent of the choice of the cut-off function $\omega$, the covering $\{k_{j}\}_{j\in\{1,\dots,\ell\}}$ and the partition $\{\varphi_{j}\}_{j\in\{1,\dots,\ell\}}$. In particular, if $s\in \mathbb{N}_{0}$, then equivalently, $\mathcal{H}^{s,\gamma}_{p}(\mathbb{B})$ is the space of all functions $u$ in $H^s_{p,loc}(\mathbb{B}^\circ)$ such that near the boundary \eqref{MellSobint} holds.
\end{definition}

If $A$ is as in \eqref{Aconeww}, then it induces a bounded map
$$
A: \mathcal{H}^{s+\mu,\gamma+\mu}_p(\mathbb{B}) \rightarrow \mathcal{H}^{s,\gamma}_{p}(\mathbb{B}),
$$
for any $p\in(1,\infty)$ and $s,\gamma\in\mathbb{R}$. On the other hand, when a cone-elliptic operator $A$ is regarded as an unbounded operator in $\mathcal{H}^{s,\gamma}_p(\mathbb{B})$, $p\in(1,\infty)$, $s,\gamma\in\mathbb{R}$, with domain $C_{c}^{\infty}(\mathbb{B}^{\circ})$, then its minimal and maximal domains differ in general. More precisely, for any $z\in\mathcal{V}$, restricting to the Laplacian $\Delta_{\mathfrak{g}}$, the domain of its minimal extension (i.e. its closure) $\underline{\Delta}_{\mathfrak{g},\min,s}$ is given by 
$$
\mathcal{D}(\underline{\Delta}_{\mathfrak{g},\min,s})=\Big\{u\in \bigcap_{\varepsilon>0}\mathcal{H}^{s+2,\gamma+2-\varepsilon}_p(\mathbb{B}) \, |\, \Delta_{\mathfrak{g}} u\in \mathcal{H}^{s,\gamma}_p(\mathbb{B})\Big\},
$$
and the following embedding holds
$$
\mathcal{H}_{p}^{s+2,\gamma+2}(\mathbb{B})\hookrightarrow\mathcal{D}(\underline{\Delta}_{\mathfrak{g},\min,s})\hookrightarrow \bigcap_{\varepsilon>0}\mathcal{H}_{p}^{s+2,\gamma+2-\varepsilon}(\mathbb{B}).
$$
If in addition the {\em conormal symbol} of $\Delta_{\mathfrak{g}}$, i.e. the following family of differential operators
$$
\mathbb{C} \ni \lambda \mapsto \lambda^{2}-(n-1)\lambda + \Delta_{\mathfrak{h}(z,0)} \in \mathcal{L}(H_{2}^{2}(\partial\mathbb{B}),H_{2}^{0}(\partial\mathbb{B})),
$$
 is invertible on the line $\{\lambda\in\mathbb{C}\,|\, \mathrm{Re}(\lambda)= \frac{\nu}{2}-2-\gamma\}$, then we have precisely $\mathcal{D}(\underline{\Delta}_{\mathfrak{g},\min,s})=\mathcal{H}^{s+2,\gamma+2}_p(\mathbb{B})$, i.e.
$$
\mathcal{D}(\underline{\Delta}_{\mathfrak{g},\min,s})=\mathcal{H}_{p}^{s+2,\gamma+2}(\mathbb{B}) \quad \text{iff} \quad \lambda_{z,j}\neq -(\gamma+\nu/2)(\gamma+2-\nu/2), \,\, j\in\mathbb{N}_{0},
$$ 
where $\dots<\lambda_{z,1}<\lambda_{z,0}=0$ stands for the spectrum of $\Delta_{\mathfrak{h}(z,0)}$. It can be shown that the conormal symbol of $\Delta_{\mathfrak{g}}$ is meromorphically invertible in $\mathbb{C}$ with poles in
$$
q_{z,j}^{\pm}=\frac{\nu-2}{2}\pm\sqrt{\Big(\frac{\nu-2}{2}\Big)^{2}-\lambda_{z,j}}, \quad j\in\mathbb{N}_{0};
$$
these poles are always simple with the exception of $q_{z,0}^{+}=q_{z,0}^{-}=0$ being a double pole in case $\nu=2$.

Concerning the domain of the maximal extension $\underline{\Delta}_{\mathfrak{g},\max,s}$ of $\Delta_{\mathfrak{g}}$, which is defined as usual by $\mathcal{D}(\underline{\Delta}_{\mathfrak{g},\max,s})=\{u\in\mathcal{H}^{s,\gamma}_{p}(\mathbb{B}) \, |\, \Delta_{\mathfrak{g}} u\in \mathcal{H}^{s,\gamma}_{p}(\mathbb{B})\}$, the following decomposition holds
\begin{equation}\label{dmax1}
\mathcal{D}(\underline{\Delta}_{\mathfrak{g},\max,s})=\mathcal{D}(\underline{\Delta}_{\mathfrak{g},\min,s})\oplus\mathcal{E}_{\Delta_{\mathfrak{g}},\gamma}.
\end{equation}
The space $\mathcal{E}_{\Delta_{\mathfrak{g}},\gamma}$ is a finite dimensional space consisting of linear combinations of $C^{\infty}(\mathbb{B}^{\circ})$-functions that vanish on $\mathcal{B}\backslash([0,1)\times \partial\mathcal{B})$ and, in local coordinates on $(0,1)\times\partial\mathcal{B}$, they are of the form $\omega(x)c(y)x^{-\rho}\log^{k}(x)$, with $c\in C^{\infty}(\partial\mathbb{B})$, $\rho\in I_{\gamma}=[\frac{\nu-4}{2}-\gamma,\frac{\nu}{2}-\gamma)$ and $k\in\{0,1\}$. The exponents $\rho$ and $k$ are determined explicitly by $\nu$, the metric $\mathfrak{h}(z,0)$ and the weight $\gamma$. If in particular $\mathfrak{h}(z,x)$ is constant in $x$ when $x$ is close to zero, then we have precisely 
\begin{equation}\label{Espace}
\mathcal{E}_{\Delta_{\mathfrak{g}},\gamma}=\bigoplus_{q_{z,j}^{\pm}\in I_{\gamma}} \mathcal{E}_{\Delta_{\mathfrak{g}},\gamma,q_{z,j}^{\pm}},
\end{equation}
where the elements of $\mathcal{E}_{\Delta_{\mathfrak{g}},\gamma,q_{z,j}^{\pm}}$ are of the form $\omega(x)c(y)x^{-q_{z,j}^{\pm}}\log^{k}(x)$ with $c\in \mathrm{Ker}(\lambda_{z,j}-\Delta_{\mathfrak{h}(z,0)})\subset C^{\infty}(\partial\mathbb{B})$ and where $k$ runs up to the multiplicity of $q_{z,j}^{\pm}$; for further details see \cite[Section 6.2]{SS}. As a consequence, there are several closed extensions of $\Delta_{\mathfrak{g}}$ in $\mathcal{H}^{s,\gamma}_p(\mathbb{B})$, also called {\em realizations}; each one corresponds to a subspace of $\mathcal{E}_{\Delta_{\mathfrak{g}},\gamma}$ in \eqref{dmax1}.

Under certain choice of the weight $\gamma$, the space $\mathbb{C}_{\omega}$ becomes a subspace of $\mathcal{E}_{\Delta_{\mathfrak{g}},\gamma}$ and the realization of the Laplacian with domain $\mathcal{H}_{p}^{s+2,\gamma+2}(\mathbb{B})\oplus\mathbb{C}_{\omega}$ satisfies the property of maximal $L^{q}$-regularity. 

\begin{theorem}[{\rm Schrohe and Seiler, \cite[Theorem 6.7]{SS}}]\label{SecLapOrigin}
Let $p\in(1,\infty)$, $s\geq0$ and let $\gamma$ satisfying \eqref{gamma2} with $\lambda_{1}$ replaced by $\lambda_{z,1}$. For any $z\in \mathcal{V}$ consider the closed extension $\underline{\Delta}_{\mathfrak{g},z,s}$ of the Laplacian $\Delta_{\mathfrak{g}}$ in $\mathcal{H}_{p}^{s,\gamma}(\mathbb{B})$ with domain $\mathcal{H}_{p}^{s+2,\gamma+2}(\mathbb{B})\oplus\mathbb{C}_{\omega}$. Then, for any $c>0$ and $\theta\in[0,\pi)$ we have $c-\underline{\Delta}_{\mathfrak{g},z,s}\in\mathcal{H}^{\infty}(\theta)$.
\end{theorem}
\begin{proof}
By \cite[Theorem 6.7]{SS}, for each $c>0$ and $\theta\in[0,\pi)$ we have $c-\underline{\Delta}_{\mathfrak{g},z,s}\in\mathcal{H}^{\infty}(\theta)$ in $\mathcal{H}_{p}^{s,\gamma}((\mathcal{B},\mathfrak{g}))$. Then, the result follows by the equivalence of the norms $\|\cdot\|_{\mathcal{H}_{p}^{s,\gamma}((\mathcal{B},\mathfrak{g}))}$ and $\|\cdot\|_{\mathcal{H}_{p}^{s,\gamma}(\mathbb{B})}$: when $s\in\mathbb{N}_{0}$ this equivalence follows by the definition of Mellin-Sobolev spaces and when $s\geq0$ it can be seen by interpolation, see e.g. \cite[Lemma 3.7]{RS2}. 
\end{proof}

We conclude the study of the cone Laplacian with the following boundedness of the family $\mathbb{V}\ni z \mapsto \lambda_{z,1}\in(-\infty,0)$.

\begin{lemma}\label{lambda1bound}
Let $\Delta_{\mathfrak{h}(z,0)}$ be the Laplacian on $\partial\mathcal{B}$ induced by $\mathfrak{h}(z,0)$. If $\lambda_{z,1}<0$ is the greatest nonzero eigenvalue of $\Delta_{\mathfrak{h}(z,0)}$, then \eqref{lambda1} holds true.
\end{lemma}
\begin{proof}
Let $\{z_{k}\}_{k\in\mathbb{N}}$ be a sequence in $\mathbb{V}$ such that $\lambda_{z_{k},1}\rightarrow 0$ as $k\rightarrow\infty$. Let $\{z_{k_{j}}\}_{j\in\mathbb{N}}$ be a subsequence converging to some $z\in \mathbb{V}$. For any $\varepsilon>0$ there exists an $m\in\mathbb{N}$ such that $j\geq m$ implies 
$$
\|\Delta_{\mathfrak{h}(z_{k_{j}},0)}-\Delta_{\mathfrak{h}(z,0)}\|_{\mathcal{L}(H_{2}^{2}(\partial\mathbb{B}),H_{2}^{0}(\partial\mathbb{B}))}<\varepsilon.
$$
Denote $H=L^{2}(\partial\mathbb{B})$. By choosing $\varepsilon=1/(2\|(\Delta_{\mathfrak{h}(z,0)}-1)^{-1}\|_{\mathcal{L}(H)})$ and using the identity
$$
\Delta_{\mathfrak{h}(z_{k_{j}},0)}-I=(I+(\Delta_{\mathfrak{h}(z_{k_{j}},0)}-\Delta_{\mathfrak{h}(z,0)})(\Delta_{\mathfrak{h}(z,0)}-1)^{-1})(\Delta_{\mathfrak{h}(z,0)}-I),
$$
we deduce that there exists an $m_{0}\in\mathbb{N}$ such that $(\Delta_{\mathfrak{h}(z_{k_{j}},0)}-1)^{-1}$ exists in $\mathcal{L}(H)$ for all $j\geq m_{0}$ and satisfies 
$$
\|(\Delta_{\mathfrak{h}(z_{k_{j}},0)}-1)^{-1}\|_{\mathcal{L}(H)}\leq 2 \|(\Delta_{\mathfrak{h}(z,0)}-1)^{-1}\|_{\mathcal{L}(H)}.
$$
Hence, by the identity
$$
(\Delta_{\mathfrak{h}(z_{k_{j}},0)}-1)^{-1}-(\Delta_{\mathfrak{h}(z,0)}-1)^{-1}=(\Delta_{\mathfrak{h}(z_{k_{j}},0)}-1)^{-1}(\Delta_{\mathfrak{h}(z,0)}-\Delta_{\mathfrak{h}(z_{k_{j}},0)})(\Delta_{\mathfrak{h}(z,0)}-1)^{-1}
$$
we find that
$$
\lim_{j\rightarrow\infty}\|(\Delta_{\mathfrak{h}(z_{k_{j}},0)}-1)^{-1}-(\Delta_{\mathfrak{h}(z,0)}-1)^{-1}\|_{\mathcal{L}(H)}=0.
$$
In particular, by Cauchy-Schwarz inequality, for any $\varepsilon>0$ there exists an $m_{1}\in \mathbb{N}$ such that $j\geq m_{1}$ implies
\begin{equation}\label{diffest}
|\langle v, ((\Delta_{\mathfrak{h}(z_{k_{j}},0)}-1)^{-1}-(\Delta_{\mathfrak{h}(z,0)}-1)^{-1})v \rangle_{H}|\leq \varepsilon\|v\|_{H}^{2},
\end{equation}
for any $v\in H$, where $\langle \cdot,\cdot\rangle_{H}$ stands for the inner product in $H$. Note that for each $w\in\mathbb{V}$ the operator $(\Delta_{\mathfrak{h}(w,0)}-1)^{-1}$ is bounded self-adjoint and bounded from below by $-1$. Moreover, by the spectral mapping theorem for self-adjoint operators, if $\sigma(\Delta_{\mathfrak{h}(w,0)})=\{\dots,\lambda_{w,1},\lambda_{w,0}=0\}$ then $\sigma((\Delta_{\mathfrak{h}(w,0)}-1)^{-1})=\{-1,(\lambda_{w,1}-1)^{-1},\dots\}$. Furthermore, by the min-max principle for self-adjoint operators, see e.g. \cite[Theorem XIII.1]{ReSi}, we have 
$$
(\lambda_{w,1}-1)^{-1}=\sup_{u\in H\backslash\{0\}} \, \inf_{v\in H,\, \|v\|_{H}=1, \, \langle v,u\rangle_{H}=0}\langle v,(\Delta_{\mathfrak{h}(w,0)}-1)^{-1}v\rangle_{H}.
$$
Hence, if we consider the difference $(\lambda_{z_{k_{j}},1}-1)^{-1}-(\lambda_{z,1}-1)^{-1}$ and take into account \eqref{diffest}, we get a contradiction. 
\end{proof}

We continue with the study of function spaces on manifolds with edges and related function spaces. We start with some geometric properties of Mellin-Sobolev spaces and of the Mellin-Sobolev spaces valued Bessel potential spaces. 

\begin{remark}\label{UMDalpha}
{\em Let $p,q\in (1,\infty)$ and $s,\gamma,\eta\in\mathbb{R}$. We have the following isomorphisms 
$$
\mathcal{H}^{s,\gamma}_{p}(\mathbb{B})\cong\mathcal{H}^{0,\gamma}_{p}(\mathbb{B}) \cong L^{p}(\mathbb{B},d\mu),
$$
where $d\mu$ is the measure generated by the conical metric (the first isomorphism can be constructed, for example, using the cone calculus). Hence, $\mathcal{H}^{s,\gamma}_{p}(\mathbb{B})$ is UMD and has property $(\alpha)$ due to \cite[Proposition 4.2.15]{HNVW} and \cite[Proposition 7.5.3]{HNVW5}. Furthermore, the Bessel potential operators provide isomorphisms of $H_{q}^{\eta}(\mathbb{R}^{n};\mathcal{H}_{p}^{s,\gamma}(\mathbb{B}))$ and $L^{q}(\mathbb{R}^{n};\mathcal{H}_{p}^{s,\gamma}(\mathbb{B}))$, so that $H_{q}^{\eta}(\mathbb{R}^{n};\mathcal{H}_{p}^{s,\gamma}(\mathbb{B}))$ is UMD and has property $(\alpha)$ by \cite[Proposition 4.2.15]{HNVW} and \cite[Proposition 7.5.3]{HNVW2}. In particular, $\mathcal{H}^{s,\gamma}_{p}(\mathbb{B})$ and $H_{q}^{\eta}(\mathbb{R}^{n};\mathcal{H}_{p}^{s,\gamma}(\mathbb{B}))$ are {\em $\nu$-admissible} in the sense of \cite[p. 173]{Am2}. Similarly, both spaces $\mathcal{H}^{s,\gamma}_{p}(\mathbb{B})\oplus\mathbb{C}_{\omega}$, $H_{q}^{\eta}(\mathbb{R}^{n};\mathcal{H}_{p}^{s,\gamma}(\mathbb{B})\oplus\mathbb{C}_{\omega})$ are UMD, have the property $(\alpha)$ and are also $\nu$-admissible.} 
\end{remark}

Next we show some interpolation properties of our vector valued Bessel potential spaces. 

\begin{lemma}\label{Hprop}
Let $\rho\in(0,1)$ and for any $a_{0},a_{1}\in\mathbb{R}$ denote $a_{\rho}=(1-\rho)a_{0}+\rho a_{1}$. Let $p,q\in(1,\infty)$, $r,r_{0},r_{1},s\geq0$ and $s_{0},s_{1},\gamma, \gamma_{0}, \gamma_{1}\in\mathbb{R}$ satisfying $r_{0}\neq r_{1}$, $0\leq s_{0}\leq s_{1}$ and $\gamma_{0}\leq \gamma_{1}$. \\
{\bf(i)} For each $k\in\mathbb{N}_{0}$, up to equivalence of norms, we have
$$
H_{q}^{k}(\mathbb{R}^{n};\mathcal{H}_{p}^{s,\gamma}(\mathbb{B}))=W^{q,k}(\mathbb{R}^{n};\mathcal{H}_{p}^{s,\gamma}(\mathbb{B})).
$$
In particular 
$$
H_{q}^{0}(\mathbb{R}^{n};\mathcal{H}_{p}^{s,\gamma}(\mathbb{B}))=L^{q}(\mathbb{R}^{n};\mathcal{H}_{p}^{s,\gamma}(\mathbb{B})).
$$
{\bf(ii)} Up to equivalence of norms, we have
$$
[H_{q}^{r_{0}}(\mathbb{R}^{n};\mathcal{H}_{p}^{s_{0},\gamma_{0}}(\mathbb{B})),H_{q}^{r_{1}}(\mathbb{R}^{n};\mathcal{H}_{p}^{s_{1},\gamma_{1}}(\mathbb{B}))]_{\rho}=H_{q}^{r_{\rho}}(\mathbb{R}^{n};\mathcal{H}_{p}^{s_{\rho},\gamma_{\rho}}(\mathbb{B})).
$$
{\bf(iii)} If $\gamma_{0}<\nu/2<\gamma_{1}$, then for any $\varepsilon>0$ we have
\begin{eqnarray*}
\lefteqn{H_{q}^{r_{\rho}}(\mathbb{R}^{n};\mathcal{H}_{p}^{s_{\rho},\gamma_{\rho}}(\mathbb{B})\oplus \mathbb{C}_{\omega})}\\
&&\hspace{-10pt}\hookrightarrow [H_{q}^{r_{0}}(\mathbb{R}^{n};\mathcal{H}_{p}^{s_{0},\gamma_{0}}(\mathbb{B})),H_{q}^{r_{1}}(\mathbb{R}^{n};\mathcal{H}_{p}^{s_{1},\gamma_{1}}(\mathbb{B})\oplus \mathbb{C}_{\omega})]_{\rho}\hookrightarrow H_{q}^{r_{\rho}}(\mathbb{R}^{n};\mathcal{H}_{p}^{s_{\rho},\gamma_{\rho}-\varepsilon}(\mathbb{B})\oplus \mathbb{C}_{\omega}).
\end{eqnarray*}
If $\nu/2<\gamma_{0}\leq\gamma_{1}$, then
$$
 [H_{q}^{r_{0}}(\mathbb{R}^{n};\mathcal{H}_{p}^{s_{0},\gamma_{0}}(\mathbb{B})\oplus\mathbb{C}_{\omega}),H_{q}^{r_{1}}(\mathbb{R}^{n};\mathcal{H}_{p}^{s_{1},\gamma_{1}}(\mathbb{B})\oplus \mathbb{C}_{\omega})]_{\rho}=H_{q}^{r_{\rho}}(\mathbb{R}^{n};\mathcal{H}_{p}^{s_{\rho},\gamma_{\rho}}(\mathbb{B})\oplus \mathbb{C}_{\omega}).
$$
\end{lemma}
\begin{proof}
{\bf(i)} Follows by Remark \ref{UMDalpha} and \cite[Theorem 5.6.11]{HNVW}. \\
{\bf(ii)} Follows by Remark \ref{UMDalpha}, \cite[Theorem VII.4.5.5]{Am2} and \cite[Lemma 3.3 (iii)]{LR}.\\
{\bf(iii)} By Remark \ref{UMDalpha} and \cite[Theorem VII.4.5.5]{Am2} we have
\begin{eqnarray*}
[H_{q}^{r_{0}}(\mathbb{R}^{n};\mathcal{H}_{p}^{s_{0},\gamma_{0}}(\mathbb{B})),H_{q}^{r_{1}}(\mathbb{R}^{n};\mathcal{H}_{p}^{s_{1},\gamma_{1}}(\mathbb{B})\oplus \mathbb{C}_{\omega})]_{\rho}\\
&&\hspace{-60pt} =\,\, H_{q}^{r_{\rho}}(\mathbb{R}^{n};[\mathcal{H}_{p}^{s_{0},\gamma_{0}}(\mathbb{B}),\mathcal{H}_{p}^{s_{1},\gamma_{1}}(\mathbb{B})\oplus \mathbb{C}_{\omega}]_{\rho}).
\end{eqnarray*}
Hence, for the first case it suffices to show that 
\begin{equation}\label{suffembint}
\mathcal{H}_{p}^{s_{\rho},\gamma_{\rho}}(\mathbb{B})\oplus \mathbb{C}_{\omega} \hookrightarrow[\mathcal{H}_{p}^{s_{0},\gamma_{0}}(\mathbb{B}),\mathcal{H}_{p}^{s_{1},\gamma_{1}}(\mathbb{B})\oplus \mathbb{C}_{\omega}]_{\rho}\hookrightarrow \mathcal{H}_{p}^{s_{\rho},\gamma_{\rho}-\varepsilon}(\mathbb{B})\oplus \mathbb{C}_{\omega},
\end{equation}
for each $\varepsilon>0$. Due to \cite[Lemma 3.3 (iii)]{LR}, we have
$$
\mathcal{H}_{p}^{s_{\rho},\gamma_{\rho}}(\mathbb{B})=[\mathcal{H}_{p}^{s_{0},\gamma_{0}}(\mathbb{B}),\mathcal{H}_{p}^{s_{1},\gamma_{1}}(\mathbb{B})]_{\rho}\hookrightarrow [\mathcal{H}_{p}^{s_{0},\gamma_{0}}(\mathbb{B}),\mathcal{H}_{p}^{s_{1},\gamma_{1}}(\mathbb{B})\oplus \mathbb{C}_{\omega}]_{\rho}.
$$
Moreover
$$
\mathbb{C}_{\omega} \hookrightarrow \mathcal{H}_{p}^{s_{0},\gamma_{0}}(\mathbb{B})\cap \big(\mathcal{H}_{p}^{s_{1},\gamma_{1}}(\mathbb{B})\oplus\mathbb{C}_{\omega}\big)\hookrightarrow [\mathcal{H}_{p}^{s_{0},\gamma_{0}}(\mathbb{B}),\mathcal{H}_{p}^{s_{1},\gamma_{1}}(\mathbb{B})\oplus \mathbb{C}_{\omega}]_{\rho}.
$$
Hence, the first embedding in \eqref{suffembint} follows immediately. Concerning the second one, up to norm equivalence we can assume that $(\mathcal{B},\mathfrak{g})$ has straight conical tips, i.e. that $\mathfrak{h}(z,x)$ in \eqref{metricg} is constant in $x$ when $x$ is close to zero. Then, from \eqref{Deltaz} we see that $\Delta_{\mathfrak{g}}$ has constant in $x$ coefficients close to the boundary $\partial\mathcal{B}$. Moreover, it induces the following bounded maps
$$
\Delta_{\mathfrak{g}}:\mathcal{H}_{p}^{s_{0},\gamma_{0}}(\mathbb{B})\rightarrow \mathcal{H}_{p}^{s_{0}-2,\gamma_{0}-2}(\mathbb{B})
$$
and
$$
\Delta_{\mathfrak{g}}:\mathcal{H}_{p}^{s_{1},\gamma_{1}}(\mathbb{B})\oplus \mathbb{C}_{\omega}\rightarrow \mathcal{H}_{p}^{s_{1}-2,\gamma_{1}-2}(\mathbb{B}).
$$
Hence, by (complex) interpolation, $\Delta_{\mathfrak{g}}$ also induces a bounded map
$$
\Delta_{\mathfrak{g}}: [\mathcal{H}_{p}^{s_{0},\gamma_{0}}(\mathbb{B}),\mathcal{H}_{p}^{s_{1},\gamma_{1}}(\mathbb{B})\oplus \mathbb{C}_{\omega}]_{\rho}\rightarrow [\mathcal{H}_{p}^{s_{0}-2,\gamma_{0}-2}(\mathbb{B}),\mathcal{H}_{p}^{s_{1}-2,\gamma_{1}-2}(\mathbb{B})]_{\rho}=\mathcal{H}_{p}^{s_{\rho}-2,\gamma_{\rho}-2}(\mathbb{B}),
$$
where at the last step we used \cite[Lemma 3.3 (iii)]{LR}. We deduce that the space $[\mathcal{H}_{p}^{s_{0},\gamma_{0}}(\mathbb{B}),\mathcal{H}_{p}^{s_{1},\gamma_{1}}(\mathbb{B})\oplus \mathbb{C}_{\omega}]_{\rho}$ embeds to the maximal domain of $\Delta_{\mathfrak{g}}$ in $\mathcal{H}_{p}^{s_{\rho}-2,\gamma_{\rho}-2}(\mathbb{B})$, i.e. 
$$
[\mathcal{H}_{p}^{s_{0},\gamma_{0}}(\mathbb{B}),\mathcal{H}_{p}^{s_{1},\gamma_{1}}(\mathbb{B})\oplus \mathbb{C}_{\omega}]_{\rho}\hookrightarrow\mathcal{H}_{p}^{s_{\rho},\gamma_{\rho}-\varepsilon}(\mathbb{B})\oplus \mathcal{E}_{\Delta_{\mathfrak{g}},\gamma_{\rho}},
$$
for all $\varepsilon>0$. Due to \eqref{Espace} we have
$$
\mathcal{E}_{\Delta_{\mathfrak{g}},\gamma_{\rho}}=\bigoplus_{q_{z,j}^{\pm}\in I_{\gamma_{\rho}}} \mathcal{E}_{\Delta_{\mathfrak{g}},\gamma_{\rho},q_{z,j}^{\pm}}.
$$
Suppose an element $u$ of $ \mathcal{E}_{\Delta_{\mathfrak{g}},\gamma_{\rho},q_{z,j}^{\pm}}$, $q_{z,j}^{\pm}\neq0$, would belong to $[\mathcal{H}_{p}^{s_{0},\gamma_{0}}(\mathbb{B}),\mathcal{H}_{p}^{s_{1},\gamma_{1}}(\mathbb{B})\oplus \mathbb{C}_{\omega}]_{\rho}$. By (complex) interpolation, the operator $\omega(x)x\partial_{x}$ maps $[\mathcal{H}_{p}^{s_{0},\gamma_{0}}(\mathbb{B}),\mathcal{H}_{p}^{s_{1},\gamma_{1}}(\mathbb{B})\oplus \mathbb{C}_{\omega}]_{\rho}$ to
$$
[\mathcal{H}_{p}^{s_{0}-1,\gamma_{0}}(\mathbb{B}),\mathcal{H}_{p}^{s_{1}-1,\gamma_{1}}(\mathbb{B})]_{\rho}=\mathcal{H}_{p}^{s_{\rho}-1,\gamma_{\rho}}(\mathbb{B}).
$$
On the other hand, by direct calculation, $\omega(x)x\partial_{x}u$ does not belong to $\mathcal{H}_{p}^{s_{\rho}-1,\gamma_{\rho}}(\mathbb{B})$, resulting in a contradiction.

Concerning the second case, similarly, by Remark \ref{UMDalpha} and \cite[Theorem VII.4.5.5]{Am2} we have
\begin{eqnarray*}
 [H_{q}^{r_{0}}(\mathbb{R}^{n};\mathcal{H}_{p}^{s_{0},\gamma_{0}}(\mathbb{B})\oplus\mathbb{C}_{\omega}),H_{q}^{r_{1}}(\mathbb{R}^{n};\mathcal{H}_{p}^{s_{1},\gamma_{1}}(\mathbb{B})\oplus \mathbb{C}_{\omega})]_{\rho}\\
&&\hspace{-80pt} =H_{q}^{r_{\rho}}(\mathbb{R}^{n};[\mathcal{H}_{p}^{s_{0},\gamma_{0}}(\mathbb{B})\oplus\mathbb{C}_{\omega},\mathcal{H}_{p}^{s_{1},\gamma_{1}}(\mathbb{B})\oplus \mathbb{C}_{\omega})]_{\rho}).
\end{eqnarray*}
Hence, the result follows by Lemma \ref{suminter} and \cite[Lemma 3.3 (iii)]{LR}.
\end{proof}

\begin{corollary}\label{sharpmixder}
Let $p,q\in(1,\infty)$, $r,s\geq0$ and $\gamma\in \mathbb{R}$. Then, for any $\rho\in(0,1)$ and $\varepsilon>0$ we have
$$
H_{q}^{r+2}(\mathbb{R}^{n};\mathcal{H}_{p}^{s,\gamma}(\mathbb{B}))\cap H_{q}^{r}(\mathbb{R}^{n};\mathcal{H}_{p}^{s+2,\gamma+2}(\mathbb{B})\oplus\mathbb{C}_{\omega})\hookrightarrow H_{q}^{r+2(1-\rho)}(\mathbb{R}^{n};\mathcal{H}_{p}^{s+2\rho,\gamma+2\rho-\varepsilon}(\mathbb{B})\oplus\mathbb{C}_{\omega}).
$$
\end{corollary}
\begin{proof}
Follows from Proposition \ref{mxtderPr} and \eqref{suffembint}.
\end{proof}

If $X$ is a Banach space denote by $BUC(\mathbb{R}^{n};X)$ the space of bounded uniformly continuous functions $u:\mathbb{R}^{n}\rightarrow X$ and by $BUC^{k}(\mathbb{R}^{n};X)$, $k\in\mathbb{N}$, the subspace of the $k$-times continuously
differentiable functions $u$ with $\partial^{\alpha}u\in BUC(\mathbb{R}^{n};X)$, $|\alpha|\leq k$.

\begin{lemma}\label{ProbHR} Let $p,q\in(1,\infty)$ and $J\in\{\{0\},\mathbb{C}_{\omega}\}$.\\
{\bf (i)} Let $r_{1}\geq r_{0}\geq0$, $s_{1}\geq s_{0}\geq0$ and $\gamma_{1},\gamma_{0}\in\mathbb{R}$ such that $\gamma_{1}\geq\gamma_{0}$. Then, the following continuous embedding holds
$$
H_{q}^{r_{1}}(\mathbb{R}^{n};\mathcal{H}_{p}^{s_{1},\gamma_{1}}(\mathbb{B})\oplus J)\hookrightarrow H_{q}^{r_{0}}(\mathbb{R}^{n};\mathcal{H}_{p}^{s_{0},\gamma_{0}}(\mathbb{B})\oplus J).
$$
{\bf (ii)} Let $k\in\mathbb{N}_{0}$, $r>k+n/q$ and $s>\nu/p$. Then
\begin{equation}\label{subsetcont}
H_{q}^{r}(\mathbb{R}^{n};\mathcal{H}_{p}^{s,\gamma}(\mathbb{B})\oplus J)\hookrightarrow BUC^{k}(\mathbb{R}^{n};\mathcal{H}_{p}^{s,\gamma}(\mathbb{B})\oplus J) \subset BUC^{k}(\mathbb{R}^{n};C(\mathbb{B}^{\circ})).
\end{equation}
Moreover, there exists a constant $C>0$ depending only on $\nu$, $p$ and $s$, such that for each $u\in BUC^{k}(\mathbb{R}^{n};\mathcal{H}_{p}^{s,\gamma}(\mathbb{B})\oplus J)$, $k_{j}\in\mathbb{N}_{0}$, $j\in\{1,\dots,n\}$, with $k_{1}+\dots+k_{n}\leq k$, and $(z_{1}\dots,z_{n},x,y)=(z,x,y)\in\mathbb{R}^{n}\times[0,1)\times\partial\mathcal{B}$ we have
\begin{equation}\label{ineqtrace}
|(\partial_{z_{1}}^{k_{1}}\dots \partial_{z_{n}}^{k_{n}}u_{\mathcal{H}})(z,x,y)|\leq Cx^{\gamma-\frac{\nu}{2}}\|(\partial_{z_{1}}^{k_{1}}\dots \partial_{z_{n}}^{k_{n}}u_{\mathcal{H}})(z,\cdot,\cdot)\|_{\mathcal{H}_{p}^{s,\gamma}(\mathbb{B})},
\end{equation}
where $u(z)=u_{\mathcal{H}}(z)+u_{J}(z)$ with $u_{\mathcal{H}}(z)\in \mathcal{H}_{p}^{s,\gamma}(\mathbb{B})$ and $u_{J}(z)\in J$, $z\in\mathbb{R}^{n}$. If in particular $\gamma>\nu/2$, then
\begin{equation}\label{contemb}
 BUC^{k}(\mathbb{R}^{n};\mathcal{H}_{p}^{s,\gamma}(\mathbb{B})\oplus J)\hookrightarrow BUC^{k}(\mathbb{R}^{n};C(\mathbb{B})).
\end{equation}
{\bf (iii)} If $k\in\mathbb{N}$, $k>n/q$, $s>\nu/p$ and $\gamma\geq\nu/2$, then, up to an equivalent norm, the space $H_{q}^{k}(\mathbb{R}^{n};\mathcal{H}_{p}^{s,\gamma}(\mathbb{B})\oplus J)$ is a Banach algebra.\\
{\bf (iv)} If $k\in\mathbb{N}$, $s>\nu/p$ and $\gamma\geq\nu/2$, then the space $BUC^{k}(\mathbb{R}^{n};\mathcal{H}_{p}^{s,\gamma}(\mathbb{B})\oplus J)$ is a Banach algebra up to an equivalent norm. In particular, the space $BUC^{k}(\mathbb{R}^{n};\mathcal{H}_{p}^{s,\gamma}(\mathbb{B})\oplus\mathbb{C}_{\omega})$ is also closed under holomorphic functional calculus. Furthermore, if $U$ is a bounded open subset of $BUC^{k}(\mathbb{R}^{n};\mathcal{H}_{p}^{s,\gamma}(\mathbb{B})\oplus\mathbb{C}_{\omega})$ consisting of functions $v$ such that $\mathrm{Re}(v)\geq\alpha>0$ for some fixed $\alpha$, then the subset $\{1/v\, |\, v\in U\}$ of $BUC^{k}(\mathbb{R}^{n};\mathcal{H}_{p}^{s,\gamma}(\mathbb{B})\oplus\mathbb{C}_{\omega})$ is also bounded; its bound can be estimated by the bound of $U$ and by the constant $\alpha$.\\
{\bf (v)} Let $k\in\mathbb{N}$, $r\in[0,k)$, $p_{0},p_{1}\in(1,\infty)$, $\sigma>0$, $s\in(-\sigma,\sigma)$ and $\gamma\in\mathbb{R}$. Then, there exists a $C_{0}>0$, depending only on $n$, $\nu$, $k$, $r$, $p_{0}$, $p_{1}$, $\sigma$, $s$ and $\gamma$, such that for each $u\in BUC^{k}(\mathbb{R}^{n};\mathcal{H}_{p_{1}}^{\sigma+\nu/p_{1},\nu/2}(\mathbb{B})\oplus J)$ and each $v\in H_{q}^{r}(\mathbb{R}^{n};\mathcal{H}_{p_{0}}^{s,\gamma}(\mathbb{B}))$ we have that $uv\in H_{q}^{r}(\mathbb{R}^{n};\mathcal{H}_{p_{0}}^{s,\gamma}(\mathbb{B}))$ and
$$
\|uv\|_{H_{q}^{r}(\mathbb{R}^{n};\mathcal{H}_{p_{0}}^{s,\gamma}(\mathbb{B}))}\leq C_{0} \|u\|_{BUC^{k}(\mathbb{R}^{n};\mathcal{H}_{p_{1}}^{\sigma+\frac{\nu}{p_{1}},\frac{\nu}{2}}(\mathbb{B})\oplus J)}\|v\|_{H_{q}^{r}(\mathbb{R}^{n};\mathcal{H}_{p_{0}}^{s,\gamma}(\mathbb{B}))}.
$$
\end{lemma}
\begin{proof}
{\bf (i)} Follows by \cite[Theorem VII.5.6.5 (ii)]{Am2} combined with \cite[Remark 2.1 (b)]{Sei}.\\
{\bf (ii)} The first embedding in \eqref{subsetcont} follows by \cite[Theorem VII.4.1.4 (ii)]{Am2}. The second inclusion in \eqref{subsetcont} as well as \eqref{ineqtrace} and \eqref{contemb} follow by \cite[Corollary 2.9.]{RS1} (see also \cite[Corollary 1.2.9]{Le}). \\
{\bf (iii)} Follows by Lemma \ref{Hprop} (i), \cite[Corollary VII.6.2.4]{Am2} and \cite[Lemma 3.2]{RS2}.\\
{\bf (iv)} The Banach algebra property follows by the definition of $BUC^{k}$-spaces and by \cite[Lemma 3.2]{RS2}. Furthermore, if $u\in BUC^{k}(\mathbb{R}^{n};\mathcal{H}_{p}^{s,\gamma}(\mathbb{B})\oplus\mathbb{C}_{\omega})$ is pointwise invertible, then by \cite[Lemma 6.2 and Lemma 6.3]{RS2}, we have that $\partial_{z_{1}}^{k_{1}}\dots \partial_{z_{n}}^{k_{n}}(1/u)\in BUC(\mathbb{R}^{n};\mathcal{H}_{p}^{s,\gamma}(\mathbb{B})\oplus\mathbb{C}_{\omega})$ for each $k_{j}\in\mathbb{N}_{0}$, $j\in\{1,\dots,n\}$, with $k_{1}+\dots+k_{n}\leq k$, where $(z_{1}\dots,z_{n})\in\mathbb{R}^{n}$. Hence, $1/u\in BUC^{k}(\mathbb{R}^{n};\mathcal{H}_{p}^{s,\gamma}(\mathbb{B})\oplus\mathbb{C}_{\omega})$. Therefore, if $f$ is a $\mathbb{C}$-valued holomorphic function in a neighborhood of $\mathrm{Ran}(u)$ and $\Gamma$ a simple closed path around $\mathrm{Ran}(-u)$ in the area of holomorphicity of $f$, then by the formula
$$
f(u)=\frac{1}{2\pi i}\int_{\Gamma}f(-\lambda)(u+\lambda)^{-1}d\lambda
$$
we deduce that $f(u)\in BUC^{k}(\mathbb{R}^{n};\mathcal{H}_{p}^{s,\gamma}(\mathbb{B})\oplus\mathbb{C}_{\omega})$. The above argument also shows the boundedness of the set $\{1/v\, |\, v\in U\}$.\\
{\bf (v)} Follows by Remark \ref{UMDalpha}, \cite[Theorems VII.2.6.5 and VII.6.5.1]{Am2}, \cite[p. 155]{Am2} and \cite[Corollary 3.3]{RS2}.
\end{proof}

We close this section with some properties of the function spaces on manifolds with edges.

\begin{lemma} \label{propF}
Let $\rho\in(0,1)$, and, for any $a_{0},a_{1}\in\mathbb{R}$ denote $a_{\rho}=(1-\rho)a_{0}+\rho a_{1}$. Moreover, let $p,q\in(1,\infty)$, $r,r_{0},r_{1}\geq0$ and $s,s_{0},s_{1},\gamma, \gamma_{0}, \gamma_{1}\in\mathbb{R}$ satisfying $r_{0}\neq r_{1}$, $0\leq s_{0}\leq s_{1}$ and $\gamma_{0}\leq \gamma_{1}$.\\
{\bf(i)} The space $\mathcal{H}_{q,p,\gamma}^{r,s}(\mathbb{F})_{\oplus J}$ is UMD, has the property $(\alpha)$ and is also $\nu$-admissible in the sense of \cite[p. 173]{Am2}.\\
{\bf(ii)} If $k\in\mathbb{N}_{0}$, then $\mathcal{H}_{q,p,\gamma}^{k,s}(\mathbb{F})_{\oplus J}=W^{q,k}(\mathbb{V};\mathcal{H}_{p}^{s,\gamma}(\mathbb{B})\oplus J)$, with equivalent norms. \\
{\bf(iii)} We have
$$
[\mathcal{H}_{q,p,\gamma_{0}}^{r_{0},s_{0}}(\mathbb{F}),\mathcal{H}_{q,p,\gamma_{1}}^{r_{1},s_{1}}(\mathbb{F})]_{\rho}=\mathcal{H}_{q,p,\gamma_{\rho}}^{r_{\rho},s_{\rho}}(\mathbb{F}).
$$
If in addition $\gamma_{0}<\nu/2<\gamma_{1}$, then for any $\varepsilon>0$ we have
$$
\mathcal{H}_{q,p,\gamma_{\rho}}^{r_{\rho},s_{\rho}}(\mathbb{F})_{\oplus \mathbb{C}_{\omega}} \hookrightarrow [\mathcal{H}_{q,p,\gamma_{0}}^{r_{0},s_{0}}(\mathbb{F}),\mathcal{H}_{q,p,\gamma_{1}}^{r_{1},s_{1}}(\mathbb{F})_{\oplus \mathbb{C}_{\omega}}]_{\rho}\hookrightarrow \mathcal{H}_{q,p,\gamma_{\rho}-\varepsilon}^{r_{\rho},s_{\rho}}(\mathbb{F})_{\oplus \mathbb{C}_{\omega}}.
$$
While, if $\nu/2<\gamma_{0}\leq \gamma_{1}$, then
$$
[\mathcal{H}_{q,p,\gamma_{0}}^{r_{0},s_{0}}(\mathbb{F})_{\oplus\mathbb{C}_{\omega}},\mathcal{H}_{q,p,\gamma_{1}}^{r_{1},s_{1}}(\mathbb{F})_{\oplus\mathbb{C}_{\omega}}]_{\rho}=\mathcal{H}_{q,p,\gamma_{\rho}}^{r_{\rho},s_{\rho}}(\mathbb{F})_{\oplus\mathbb{C}_{\omega}}.
$$
{\bf(iv)} If $r_{1}\geq r_{0}\geq0$, then
$$
\mathcal{H}_{q,p,\gamma_{1}}^{r_{1},s_{1}}(\mathbb{F})_{\oplus J}\hookrightarrow \mathcal{H}_{q,p,\gamma_{0}}^{r_{0},s_{0}}(\mathbb{F})_{\oplus J}.
$$
{\bf(v)} If $r>n/q$ and $s>\nu/p$, then
$$
\mathcal{H}_{q,p,\gamma}^{r,s}(\mathbb{F})_{\oplus J}\hookrightarrow C(\mathbb{V}; \mathcal{H}_{p}^{s,\gamma}(\mathbb{B})\oplus J)\subset C(\mathbb{F}^{\circ}).
$$
Furthermore, for each $u\in \mathcal{H}_{q,p,\gamma}^{r,s}(\mathbb{F})_{\oplus J}$ there exists a $c\in C(\mathbb{F})$ such that, in local coordinates $(z,x,y)\in\mathcal{V}\times[0,1)\times\mathcal{Y}$ we have $c=c(z)$, $u(z,\cdot,\cdot)-\omega(\cdot) c(z)\in \mathcal{H}_{p}^{s,\gamma}(\mathbb{B})$ and
$$
|u(z,x,y)-\omega(x) c(z)|\leq Cx^{\gamma-\nu/2}\|u(z,\cdot,\cdot)-\omega(\cdot) c(z)\|_{\mathcal{H}_{p}^{s,\gamma}(\mathbb{B})},
$$
for some constant $C>0$ depending only on $\nu$, $p$ and $s$. In particular, $c\equiv 0$ when $J=\{0\}$. Consequently, if in addition $\gamma>\nu/2$, then
$$
\mathcal{H}_{q,p,\gamma}^{r,s}(\mathbb{F})_{\oplus J}\hookrightarrow C(\mathbb{F}).
$$
{\bf(vi)} If $k\in\mathbb{N}$, $k>n/q$, $s>\nu/p$ and $\gamma\geq\nu/2$, then, up to an equivalent norm, the space $\mathcal{H}_{q,p,\gamma}^{k,s}(\mathbb{F})_{\oplus J}$ is a Banach algebra. \\
{\bf(vii)} Let $r_{1}>n/q$, $s_{1}>\nu/p$, $r_{1}-n/q>k>r_{0}\geq0$, for certain $k\in\mathbb{N}$, and $s_{0}\in(\nu/p-s_{1},s_{1}-\nu/p)$. Moreover, let $V\subset \mathcal{H}_{q,p,\nu/2}^{r_{1},s_{1}}(\mathbb{F})_{\oplus \mathbb{C}_{\omega}}$ be a bounded set and $f$ be a complex valued function analytic on a neighbourhood of the closure of $\cup_{v\in V}\mathrm{Ran}(v)$. Then, for each $u\in \mathcal{H}_{q,p,\gamma}^{r_{0},s_{0}}(\mathbb{F})$ and $v\in V$ we have $uf(v)\in \mathcal{H}_{q,p,\gamma}^{r_{0},s_{0}}(\mathbb{F})$. In addition, there exists a constant $C_{V,f}>0$, depending only on $n$, $\nu$, $k$, $r_{0}$, $r_{1}$, $s_{0}$, $s_{1}$, $\gamma$, $q$, $p$, $V$ and $f$, such that
\begin{equation}\label{functcalinv1}
\|uf(v)\|_{\mathcal{H}_{q,p,\gamma}^{r_{0},s_{0}}(\mathbb{F})}\leq C_{V,f}\|u\|_{\mathcal{H}_{q,p,\gamma}^{r_{0},s_{0}}(\mathbb{F})}
\end{equation}
and
\begin{equation}\label{functcalinv2}
\|u(f(v_{1})-f(v_{2}))\|_{\mathcal{H}_{q,p,\gamma}^{r_{0},s_{0}}(\mathbb{F})}\leq C_{V,f}\|v_{1}-v_{2}\|_{ \mathcal{H}_{q,p,\nu/2}^{r_{1},s_{1}}(\mathbb{F})_{\oplus \mathbb{C}_{\omega}}}\|u\|_{\mathcal{H}_{q,p,\gamma}^{r_{0},s_{0}}(\mathbb{F})},
\end{equation}
for all $u\in \mathcal{H}_{q,p,\gamma}^{r_{0},s_{0}}(\mathbb{F})$ and $v,v_{1},v_{2}\in V$.
\end{lemma}
\begin{proof}
{\bf(i)}-{\bf(vi)} follow by Remark \ref{UMDalpha}, Lemma \ref{Hprop} and Lemma \ref{ProbHR}. \\
{\bf(vii)} Due to (v), we can choose a simple closed path $\Gamma$ in $\mathbb{C}$ around the closure of $\cup_{v\in V}\mathrm{Ran}(v)$. For each $u\in \mathcal{H}_{q,p,\gamma}^{r_{0},s_{0}}(\mathbb{F})$ and $v\in V$ we write 
$$
uf(v)=\frac{u}{2\pi i}\int_{\Gamma}\frac{f(\lambda)}{\lambda-v}d\lambda, 
$$
so that
$$
\|uf(v)\|_{\mathcal{H}_{q,p,\gamma}^{r_{0},s_{0}}(\mathbb{F})} \leq (2\pi)^{-1}\int_{\Gamma}\sum_{j=1}^{N}\|(\kappa_{j})_{\ast}(\phi_{j}u(\lambda-v)^{-1})\|_{H_{q}^{r_{0}}(\mathbb{R}^{n};\mathcal{H}_{p}^{s_{0},\gamma}(\mathbb{B}))} |f(\lambda)|d\lambda,
$$
where $\{\phi_{j}\}_{j\in\{1,\dots,N\}}$ and $\{\kappa_{j}\}_{j\in\{1,\dots,N\}}$ are as in \eqref{nrmHF}. Then, \eqref{functcalinv1} follows by (ii), (iv) and (v) of Lemma \ref{ProbHR}. Concerning \eqref{functcalinv2}, it follows similarly by the identity 
$$
\phi_{j}u(f(v_{1})-f(v_{2}))=(v_{1}-v_{2})\frac{\phi_{j}u}{2\pi i}\int_{\Gamma}\frac{f(\lambda)}{(\lambda-v_{1})(\lambda-v_{2})}d\lambda.
$$
\end{proof}

\section{The Laplacian on manifolds with edges}

We start by freezing the coefficients to the components of the Laplacian $\Delta_{\mathbb{F}}$. If $k_{j}:V_{j}\subset \mathcal{V}\rightarrow \mathbb{R}^{n}$ is a coordinate map as in \eqref{nrmHF} and $z\in V_{j}$, let $\Delta_{\mathfrak{m},z}(j)$ be the differential operator on $\mathbb{R}^{n}$ obtained from the local expression of $\Delta_{\mathfrak{m}}$ on $V_{j}$ by freezing its coefficients in $\kappa_{j}(z)$; for convenience of notation, in the sequel we shall simply write $\Delta_{\mathfrak{m},z}$ rather than $\Delta_{\mathfrak{m},z}(j)$.

\begin{proposition}\label{p2}
Let $p,q\in(1,\infty)$, $r,s\geq0$, $\gamma\in\mathbb{R}$, $\theta\in[0,\pi)$ and fix $z\in V_{j}$ for some $j\in\{1,\dots,N\}$, where $V_{j}\subset \mathcal{V}$ is as in \eqref{nrmHF}. Let $\underline{\Delta}_{\mathfrak{m},z,r}$ be the closed extension of $\Delta_{\mathfrak{m},z}$ in $H_{q}^{r}(\mathbb{R}^{n};\mathcal{H}_{p}^{s,\gamma}(\mathbb{B}))$ with domain $H_{q}^{r+2}(\mathbb{R}^{n};\mathcal{H}_{p}^{s,\gamma}(\mathbb{B}))$. Then, there exists a $c>0$ such that $c-\underline{\Delta}_{\mathfrak{m},z,r}\in \mathcal{RH}^{\infty}(\theta)$.
\end{proposition}
\begin{proof}
It suffices to show that $c-\underline{\Delta}_{\mathfrak{m},z,r} \in \mathcal{H}^{\infty}(\theta)$. Then, the result will follow by Remark \ref{UMDalpha} and \cite[Theorem 5.3]{KaW}. We write $\Delta_{\mathfrak{m},z}={}_0\Delta_{\mathfrak{m},z}+{}_1\Delta_{\mathfrak{m},z}$, where ${}_0\Delta_{\mathfrak{m},z}$ stands for the principal part of $\Delta_{\mathfrak{m},z}$, and denote $E_{\eta}=H_{q}^{\eta}(\mathbb{R}^{n};\mathcal{H}_{p}^{s,\gamma}(\mathbb{B}))$, for any $\eta\geq0$. We split the proof in several steps according to different values of $r$. 

{\em Case $r=0$}. Note that $-{}_0\Delta_{\mathfrak{m},z}$ is a second order {\em homogeneous elliptic operator} in the sense of \cite[(8.3)-(8.5)]{Ha}. Hence, since $\mathcal{H}_{p}^{s,\gamma}(\mathbb{B})$ is UMD (see Remark \ref{UMDalpha}), by \cite[Theorem 8.2.1]{Ha} and Lemma \ref{Hprop} (i), {\em the part} $(-{}_0\Delta_{\mathfrak{m},z})_{q}$ of $-{}_0\Delta_{\mathfrak{m},z}$ in $E_{0}$ has domain $E_{2}$; see \cite[p. 219]{Ha} for the definition of the part $(-{}_0\Delta_{\mathfrak{m},z})_{q}$ of $-{}_0\Delta_{\mathfrak{m},z}$ in $E_{0}$. Denote $(-{}_0\Delta_{\mathfrak{m},z})_{q}$ simply by ${}_0\underline{\Delta}_{\mathfrak{m},z}$. Therefore, by \cite[Theorem 8.2.7]{Ha}, there exists a $c>0$ such that $c- {}_0\underline{\Delta}_{\mathfrak{m},z}\in \mathcal{H}^{\infty}(\theta)$; this result also follows by \cite[Theorem 5.5]{DHP}. In particular, $c-{}_0\underline{\Delta}_{\mathfrak{m},z}\in \mathcal{P}(\theta)$. Hence, by \cite[(I.2.5.2) and (I.2.9.6)]{Am}, for $c_{0}$ sufficiently large and each $\rho\in(1/2,1)$ we have
$$
\mathcal{D}((c_{0}-{}_0\underline{\Delta}_{\mathfrak{m},z})^{\rho}) \hookrightarrow (E_{0},E_{2})_{\frac{1}{2},1}.
$$
Thus, by \cite[Theorem 8.2.1]{Ha} and Lemma \ref{Hprop} (i), we obtain
$$
\mathcal{D}((c_{0}-{}_0\underline{\Delta}_{\mathfrak{m},z})^{\rho}) \hookrightarrow E_{1},
$$
which implies that 
\begin{equation}\label{fracpowerest}
{}_1\Delta_{\mathfrak{m},z}(c_{0}-{}_0\underline{\Delta}_{\mathfrak{m},z})^{-\rho} \in \mathcal{L}(E_{0}).
\end{equation}
As a consequence, due to \cite[Lemma III.4.8.1]{Am}, by possibly increasing $c$ and using Remark \ref{TaLnLem}, we get that $c-\underline{\Delta}_{\mathfrak{m},z,0} \in \mathcal{P}(\theta)$. Then, the result follows by Remark \ref{TaLnLem}, \eqref{fracpowerest} and the perturbation result \cite[Proposition 5.5.3]{Ha}.

{\em Remark}. Denote $A_{\eta}=c-\underline{\Delta}_{\mathfrak{m},z,\eta}$, where $c$ is the constant from the previous step. We claim that $S_{\theta}\subset \rho(-A_{\eta})$ and 
$$
(A_{\eta}+\lambda)^{-1}=(A_{0}+\lambda)^{-1}|_{E_{\eta}} \quad \text{for each} \quad \lambda\in S_{\theta}.
$$
To this end, it suffices to show that the identities
$$
 (A_{0}+\lambda)^{-1}(A_{\eta}+\lambda)=I \quad \text{and} \quad (A_{\eta}+\lambda)(A_{0}+\lambda)^{-1}=I 
$$
hold on $E_{\eta+2}$ and $E_{\eta}$ respectively. The first one is trivial. For the second one, let us first restrict to the case of $\eta\in(0,1]$. Let $u\in E_{2}$ such that $(c-\Delta_{\mathfrak{m},z}+\lambda)u\in E_{\eta}$. Then, ${}_0\Delta_{\mathfrak{m},z}u=(c+\lambda)u-{}_1\Delta_{\mathfrak{m},z}u \in E_{\eta}$, i.e. $u$ belongs to the maximal domain of ${}_0\Delta_{\mathfrak{m},z}$ in $E_{\eta}$. By ellliptic regularity, similarly to the scalar case, this maximal domain is $E_{\eta+2}$, see also \cite[Theorem 8.2.1]{Ha}. The assertion for higher $\eta$ follows by iteration. \\
Consequently, if $H_{0}^{\infty}(\theta)$ is as in Definition \ref{hinftycal}, then for any $f\in H_{0}^{\infty}(\theta)$ we have that
\begin{equation}\label{restrf}
f(A_{\eta})=f(A_{0})|_{E_{\eta}}.
\end{equation}

{\em Case $r/2\in\mathbb{N}_{0}$}. We proceed by induction. Assume that $A_{r}\in \mathcal{H}^{\infty}(\theta)$ for some $r$ with $r/2\in\mathbb{N}_{0}$. By \eqref{restrf}, for any $f\in H_{0}^{\infty}(\theta)$ and $u\in E_{r+2}$, we have
\begin{eqnarray*}
\|f(A_{r+2})u\|_{E_{r+2}} &=& \|f(A_{r})u\|_{E_{r+2}}\\
&\leq& C_{1}\big(\|f(A_{r})u\|_{E_{r}}+\|A_{r}f(A_{r})u\|_{E_{r}}\big)\\
&=&C_{1}\big(\|f(A_{r})u\|_{E_{r}}+\|f(A_{r})A_{r}u\|_{E_{r}}\big)\\
&\leq& C_{2}(\sup_{\lambda\in\mathbb{C}\backslash S_{\theta}}|f(\lambda)|)\big(\|u\|_{E_{r}}+\|A_{r}u\|_{E_{r}}\big)\\
&\leq& C_{3}(\sup_{\lambda\in\mathbb{C}\backslash S_{\theta}}|f(\lambda)|)\|u\|_{E_{r+2}},
\end{eqnarray*}
for certain $C_{1}, C_{2}, C_{3}>0$. Note that the above argument also shows sectoriality for $A_{r+2}$ at the same sector as $A_{r}$. Therefore, $A_{r+2}\in \mathcal{H}^{\infty}(\theta)$.

{\em Case $r\geq0$}. Follows from the previous step and interpolation, i.e. by Lemma \ref{Hprop} (ii), \eqref{restrf} and \cite[Theorem 2.6]{Lunar18}.
\end{proof}

\begin{remark}\label{remarkhinfty}{\em 
Let $\underline{\Delta}_{\mathfrak{g},z,s}$ be the closed extension of $\Delta_{\mathfrak{g}}$ from Theorem \ref{SecLapOrigin}. By Remark \ref{UMDalpha} and \cite[Theorem 5.3]{KaW}, for any $c>0$ and $\theta\in[0,\pi)$ we have $c-\underline{\Delta}_{\mathfrak{g},z,s}\in\mathcal{RH}^{\infty}(\theta)$.} 
\end{remark}

Based on the above two results, Theorem \ref{t1} provides us $R$-sectorial local approximations for the leading term of the edge Laplacian $\Delta_{\mathfrak{e}}$. 

\begin{proposition}\label{rsecforfreezing}
Let $p,q\in(1,\infty)$, $s\geq0$, $\gamma$ be as in \eqref{gamma2}, fix $z\in\mathcal{V}$ and let $\underline{\Delta}_{\mathfrak{m},z,r}$ be the operator defined in Proposition \ref{p2}. Moreover let $\underline{\Delta}_{\mathfrak{g},z,s}$ be the closed extension of $\Delta_{\mathfrak{g}}$ from Theorem \ref{SecLapOrigin} and let $\underline{\Delta}_{\mathfrak{g},z,s}$ denote its natural extension to $H_{q}^{r}(\mathbb{R}^{n};\mathcal{H}_{p}^{s,\gamma}(\mathbb{B}))$ with domain $H_{q}^{r}(\mathbb{R}^{n};\mathcal{H}_{p}^{s+2,\gamma+2}(\mathbb{B})\oplus\mathbb{C}_{\omega})$. Then the operator $\underline{\Delta}_{\mathfrak{m},z,0}+\underline{\Delta}_{\mathfrak{g},z,s}$ with domain $H_{q}^{2}(\mathbb{R}^{n};\mathcal{H}_{p}^{s,\gamma}(\mathbb{B}))\cap L^{q}(\mathbb{R}^{n};\mathcal{H}_{p}^{s+2,\gamma+2}(\mathbb{B})\oplus\mathbb{C}_{\omega})$ in $L^{q}(\mathbb{R}^{n};\mathcal{H}_{p}^{s,\gamma}(\mathbb{B}))$ is closed and moreover, for any $\theta\in [0,\pi)$ there exists a $c>0$ such that $c-\underline{\Delta}_{\mathfrak{m},z,0}-\underline{\Delta}_{\mathfrak{g},z,s} \in \mathcal{R}(\theta)$.
\end{proposition}
\begin{proof}
The result follows by Theorem \ref{t1}, Remark \ref{UMDalpha}, Proposition \ref{p2} and Remark \ref{remarkhinfty}. Note that in our case the coefficients of $A$ and the map $B(\cdot)$ in Theorem \ref{t1} are both constant.
\end{proof}

{\bf Proof of Theorem \ref{thsec33}.} Denote 
$$
X_{0}^{r}=\mathcal{H}_{q,p,\gamma}^{r,s}(\mathbb{F}), \quad X_{1}^{r}=\mathcal{H}_{q,p,\gamma}^{r+2,s}(\mathbb{F})\cap \mathcal{H}_{q,p,\gamma+2}^{r,s+2}(\mathbb{F})_{\oplus\mathbb{C}_{\omega}}, 
$$
$$
E_{0}^{r}=H_{q}^{r}(\mathbb{R}^{n};\mathcal{H}_{p}^{s,\gamma}(\mathbb{B})), \quad E_{1}^{r}=H_{q}^{r+2}(\mathbb{R}^{n};\mathcal{H}_{p}^{s,\gamma}(\mathbb{B}))\cap H_{q}^{r}(\mathbb{R}^{n};\mathcal{H}_{p}^{s+2,\gamma+2}(\mathbb{B})\oplus\mathbb{C}_{\omega}),
$$
and let $\underline{\Delta}_{\mathfrak{e},r}$ be the operator $\Delta_{\mathfrak{e}}:X_{1}^{r}\rightarrow X_{0}^{r}$. We split the proof in several steps according to different values of $r$.\\
{\em Case $r=0$}. Denote by $U_{\tau}(z)$ the geodesic ball in $\mathbb{V}$ of radius $\tau>0$, centered at $z\in \mathbb{V}$. Consider two open covers $\{U_{\tau}(z_{j})\}_{j\in\{1,\dots,N\}}$ and $\{U_{2\tau}(z_{j})\}_{j\in\{1,\dots,N\}}$ of $\mathbb{V}$ by coordinate charts, where $z_{1},\dots,z_{N}\in \mathcal{V}$, for some $\tau>0$ and $N\in\mathbb{N}$. Let $\{\chi_{j}\}_{j\in\{1,\dots,N\}}$ be a collection of smooth positive functions bounded by one, such that $\chi_{j}=1$ in $U_{\tau}(z_{j})$ and $\chi_{j}=0$ outside of $U_{2\tau}(z_{j})$, for each $j$. In addition, let $\{\phi_{j}\}_{j\in\{1,\dots,N\}}$ be a partition of unity subordinate to the covering $\{U_{\tau}(z_{j})\}_{j\in\{1,\dots,N\}}$ and $\{\psi_{j}\}_{j\in\{1,\dots,N\}}$ be a collection of smooth positive functions such that $\mathrm{supp}(\phi_{j})\subset \mathrm{supp}(\psi_{j})\subset U_{\tau}(z_{j})$ and $\psi_{j}=1$ on $\mathrm{supp}(\phi_{j})$, for each $j\in\{1,\dots,N\}$.

For each $j\in\{1,\dots,N\}$, let $A_{z_{j}}=\underline{\Delta}_{\mathfrak{m},z_{j},0}+\underline{\Delta}_{\mathfrak{g},z_{j},s} : E_{1}^{0} \rightarrow E_{0}^{0}$ be the operator defined in Proposition \ref{rsecforfreezing}. By Proposition \ref{rsecforfreezing}, there exists a $c_{0}>0$ such that $c_{0}-A_{z_{j}}\in \mathcal{R}(\theta)$, $j\in\{1,\dots,N\}$. By the same argument as in Step 2 and Step 3 of the proof of Theorem \ref{t1}, the spectral shift $c_{0}$ and the $R$-sectorial bounds of $c_{0}-A_{z_{j}}$ can be chosen uniformly bounded in $z_{j}\in \mathbb{V}$, i.e. there exist $c_{0},K_{0}>0$, with $c_{0}$ arbitrary large, such that $c_{0}-A_{z_{j}}\in \mathcal{R}(K_{0},\theta)$ for each $z_{j}\in \mathbb{V}$, $j\in\{1,\dots,N\}$.

For any $\delta, \vartheta\in(0,1)$, by Lemma \ref{Hprop} (iii) and Corollary \ref{sharpmixder} we have that
 $$
 [E_{0}^{0},E_{1}^{0}]_{\delta}\hookrightarrow H_{q}^{2\delta(1-\vartheta)}(\mathbb{R}^{n};\mathcal{H}_{p}^{s+2\delta \vartheta,\gamma+2\delta\vartheta-\varepsilon}(\mathbb{B})\oplus\mathbb{C}_{\omega}),
 $$
 for all $\varepsilon>0$. Hence, by choosing $2\delta(1-\vartheta)>1$, \cite[(I.2.5.2) and (I.2.9.6)]{Am} imply 
 $$
 \mathcal{D}((c_{0}-A_{z_{j}})^{\delta})\hookrightarrow H_{q}^{1}(\mathbb{R}^{n};\mathcal{H}_{p}^{s,\gamma}(\mathbb{B})),
 $$
 so that $ D_{z_{j}}(c_{0}-A_{z_{j}})^{-\delta} \in \mathcal{L}(E_{0}^{0})$, for each $j\in\{1,\dots,N\}$, where by $D_{z_{j}}$ we denote the operator arising from the local expression of $D$ after freezing its coefficients at the point $z_{j}$. Therefore, by writing 
 $$
 D_{z_{j}}(c_{0}+\underline{c}_{0}-A_{z_{j}})^{-1}=D_{z_{j}}(c_{0}-A_{z_{j}})^{-\delta}(c_{0}-A_{z_{j}})^{\delta}(c_{0}+\underline{c}_{0}-A_{z_{j}})^{-1} 
 $$
 for some $\underline{c}_{0}>0$, and then choosing $\underline{c}_{0}$ large enough, by the perturbation result \cite[Theorem 1]{KW}, \cite[Lemma 2.6]{RS2} and Remark \ref{TaLnLem} we get that $c_{1}-A_{z_{j}}-D_{z_{j}}\in \mathcal{R}(K_{1},\theta)$, for certain $c_{1},K_{1}>0$ and all $j\in\{1,\dots,N\}$.

Define 
\begin{equation}\label{eeqqppert}
A_{j}=\chi_{j}(c_{1}-\underline{\Delta}_{\mathfrak{e},0})+(1-\chi_{j})(c_{1}-A_{z_{j}}-D_{z_{j}})=c_{1}-A_{z_{j}}-D_{z_{j}}+\chi_{j}(A_{z_{j}}+D_{z_{j}}-\underline{\Delta}_{\mathfrak{e},0}), 
\end{equation}
 $j\in\{1,\dots,N\}$, with domain $E_{1}^{0}$ in $E_{0}^{0}$, where $\underline{\Delta}_{\mathfrak{e},0}$ is expressed in local coordinates in $U_{2\tau}(z_{j})$. By taking $\tau>0$ sufficiently small, and possibly $N$ large enough, due to \cite[Theorem 1]{KW}, by \eqref{deltaF}, \eqref{eeqqppert} and \cite[Lemma 2.6]{RS2} we find some $K\geq1$ such that $A_{j}\in\mathcal{R}(K,\theta)$ for each $j\in\{1,\dots,N\}$.

Take $\lambda\in S_{\theta}$, $u\in X_{1}^{0}$ and $f\in X_{0}^{0}$ such that
$$
(c_{1}-\underline{\Delta}_{\mathfrak{e},0}+\lambda)u=f.
$$
By multiplying with $\phi_{j}$ we obtain
$$
\phi_{j}f=(c_{1}-\underline{\Delta}_{\mathfrak{e},0}+\lambda)\phi_{j}u+[\underline{\Delta}_{\mathfrak{e},0},\phi_{j}]u,
$$
i.e.
$$
\phi_{j}f=(A_{j}+\lambda)\phi_{j}u+[\phi_{j},A_{j}]u.
$$
By applying the resolvent of $A_{j}$ we get
$$
(A_{j}+\lambda)^{-1}\phi_{j}f=\phi_{j}u+(A_{j}+\lambda)^{-1}[\phi_{j},A_{j}]u,
$$
where by multiplying with $\psi_{j}$ and summing up we find
\begin{equation}\label{lakjd}
\sum_{j=1}^{N}\psi_{j}(A_{j}+\underline{c}_{1}+\lambda)^{-1}\phi_{j}f=u+\sum_{j=1}^{N}\psi_{j}(A_{j}+\underline{c}_{1}+\lambda)^{-1}[\phi_{j},A_{j}]u,
\end{equation}
for any $\underline{c}_{1}>0$. By Corollary \ref{sharpmixder} we have that
\begin{equation}\label{domemb}
E_{1}^{0}\hookrightarrow H_{p}^{2\rho_{1}}(\mathbb{R}^{n};\mathcal{H}_{p}^{s+2\rho_{2},\gamma+2\rho_{2}}(\mathbb{B})\oplus\mathbb{C}_{\omega}),
\end{equation}
for all $0<\rho_{1}<1-\rho_{2}<1$. Let $\varepsilon>0$ be sufficiently small, and consider the following two choices for $\rho_{1}$, $\rho_{2}$:
\begin{equation}\label{phochoices}
(\rho_{1},\rho_{2})=(1/2+\varepsilon,\varepsilon) \quad \text{and} \quad (\rho_{1},\rho_{2})=(\varepsilon,1/2+\varepsilon).
\end{equation}
For each $j\in \{1,\dots,N\}$, by \eqref{deltaF}, $[\phi_{j},A_{j}]$ is sum of a first order differential operator with smooth coefficients on $\mathbb{R}^{n}$ and a smooth in $z\in\mathbb{V}$ family of first order differential operators with smooth coefficients on $\mathcal{B}$. Therefore, by \eqref{domemb} with the choices \eqref{phochoices}, for each $j\in \{1,\dots,N\}$, the map 
\begin{equation}\label{bndmap}
[\phi_{j},A_{j}]: X_{1}^{0} \rightarrow H_{p}^{2\varepsilon}(\mathbb{R}^{n};\mathcal{H}_{p}^{s+2\varepsilon,\gamma+2\varepsilon}(\mathbb{B})\oplus\mathbb{C}_{\omega})
\end{equation}
is bounded. On the other hand, for each $\rho\in(0,1)$, by Lemma \ref{Hprop} (iii) we have
$$
H_{p}^{2\rho}(\mathbb{R}^{n};\mathcal{H}_{p}^{s+2\rho+\varepsilon,\gamma+2\rho+\varepsilon}(\mathbb{B})\oplus\mathbb{C}_{\omega})\hookrightarrow [E_{0}^{0},H_{p}^{2}(\mathbb{R}^{n};\mathcal{H}_{p}^{s+2,\gamma+2}(\mathbb{B})\oplus\mathbb{C}_{\omega})]_{\rho} \hookrightarrow [E_{0}^{0},E_{1}^{0}]_{\rho}.
$$
By choosing $\rho=\varepsilon/2$ and taking into account \cite[(I.2.5.2) and (I.2.9.6)]{Am}, from \eqref{bndmap} we deduce that there exists an $a\in(0,1/2)$ such that, for each $j\in \{1,\dots,N\}$, the map
\begin{equation}\label{bndfracmap}
[\phi_{j},A_{j}]: X_{1}^{0}\rightarrow \mathcal{D}(A_{j}^{a})
\end{equation}
is bounded. Therefore, by writing
$$
(A_{j}+\underline{c}_{1}+\lambda)^{-1}[\phi_{j},A_{j}]=A_{j}^{-1}A_{j}(A_{j}+\underline{c}_{1}+\lambda)^{-1}A_{j}^{-a}A_{j}^{a}[\phi_{j},A_{j}]
$$
in \eqref{lakjd} and using Remark \ref{TaLnLem}, we can make the $\mathcal{L}(X_{1}^{0})$-norm of the second term on the right-hand side of \eqref{lakjd} arbitrary small, uniformly in $\lambda$, by choosing $\underline{c}_{1}$ sufficiently large. Hence, by denoting $c_{2}=c_{1}+\underline{c}_{1}$ and choosing $\underline{c}_{1}>0$ large enough, we deduce that $c_{2}-\underline{\Delta}_{\mathfrak{e},0}+\lambda$ has a left inverse $L$ that belongs to $\mathcal{L}(X_{0}^{0},X_{1}^{0})$.

By \eqref{lakjd} we have
\begin{eqnarray}\nonumber
(c_{2}-\underline{\Delta}_{\mathfrak{e},0}+\lambda)L&=&(c_{2}-\underline{\Delta}_{\mathfrak{e},0}+\lambda)\sum_{j=1}^{N}\psi_{j}(A_{j}+\underline{c}_{1}+\lambda)^{-1}(\phi_{j}+[A_{j},\phi_{j}]L)\\\nonumber
&=&\sum_{j=1}^{N}[\psi_{j},\underline{\Delta}_{\mathfrak{e},0}](A_{j}+\underline{c}_{1}+\lambda)^{-1}(\phi_{j}+[A_{j},\phi_{j}]L)\\\nonumber
&&+\sum_{j=1}^{N}\psi_{j}(c_{2}-\underline{\Delta}_{\mathfrak{e},0}+\lambda)(A_{j}+\underline{c}_{1}+\lambda)^{-1}(\phi_{j}+[A_{j},\phi_{j}]L)\\\label{rightinvrs}
&=&I+\sum_{j=1}^{N}[\psi_{j},\underline{\Delta}_{\mathfrak{e},0}](A_{j}+\underline{c}_{1}+\lambda)^{-1}(\phi_{j}+[A_{j},\phi_{j}]L),
\end{eqnarray}
where we have used the fact that 
$$
\sum_{j=1}^{N}[A_{j},\phi_{j}]=\sum_{j=1}^{N}[\phi_{j},\underline{\Delta}_{\mathfrak{e},0}]=0. 
$$
For each $b\in(0,\rho)$ and $j\in\{1,\dots,N\}$, by \cite[(I.2.5.2) and (I.2.9.6)]{Am}, Lemma \ref{Hprop} (iii) and \eqref{domemb}, we have
\begin{eqnarray*}
\lefteqn{\mathcal{D}(A_{j}^{b})\hookrightarrow [E_{0}^{0},E_{1}^{0}]_{\rho} \hookrightarrow [E_{0}^{0}, H_{q}^{2\rho_{1}}(\mathbb{R}^{n};\mathcal{H}_{p}^{s+2\rho_{2},\gamma+2\rho_{2}}(\mathbb{B})\oplus\mathbb{C}_{\omega})]_{\rho}}\\
&&\hookrightarrow H_{q}^{2\rho\rho_{1}-\varepsilon/2}(\mathbb{R}^{n};\mathcal{H}_{p}^{s+2\rho\rho_{2}-\varepsilon/2,\gamma+2\rho\rho_{2}-\varepsilon/2}(\mathbb{B})\oplus\mathbb{C}_{\omega}).
\end{eqnarray*}
Hence, by choosing $\rho_{1}$, $\rho_{2}$ according to \eqref{phochoices}, setting $\rho=1-\varepsilon$ and using the expression of $[\psi_{j},\underline{\Delta}_{\mathfrak{e},0}]$ due to \eqref{deltaF}, we obtain that, for each $j\in\{1,\dots,N\}$, $[\psi_{j},\underline{\Delta}_{\mathfrak{e},0}]A_{j}^{-b}$ is a bounded operator on $E_{0}^{0}$; in particular 
$$
[\psi_{j},\underline{\Delta}_{\mathfrak{e},0}]A_{j}^{-b}: E_{0}^{0}\rightarrow H_{q}^{\varepsilon/4}(\mathbb{R}^{n};\mathcal{H}_{p}^{s+\varepsilon/4,\gamma+\varepsilon/4}(\mathbb{B})\oplus\mathbb{C}_{\omega}), \quad j\in\{1,\dots,N\}.
$$ 
As a consequence, by writing
$$
[\psi_{j},\underline{\Delta}_{\mathfrak{e},0}](A_{j}+\underline{c}_{1}+\lambda)^{-1}=[\psi_{j},\underline{\Delta}_{\mathfrak{e},0}]A_{j}^{-b}A_{j}^{b}(A_{j}+\underline{c}_{1}+\lambda)^{-1}
$$
and using Remark \ref{TaLnLem}, after taking $\underline{c}_{1}>0$ sufficiently large, we can make the $\mathcal{L}(X_{0}^{0})$-norm of the second term on the right-hand side of \eqref{rightinvrs} arbitrary small. Thus, $c_{2}-\underline{\Delta}_{\mathfrak{e},0}+\lambda$ has also a right inverse, that belongs to $\mathcal{L}(X_{0}^{0}, X_{1}^{0})$.

By \eqref{lakjd} we have
$$
(c_{2}-\underline{\Delta}_{\mathfrak{e},0}+\lambda)^{-1}=\sum_{k=0}^{\infty}Q^{k}(\lambda)R(\lambda), \quad \lambda\in S_{\theta},
$$
where
$$
Q(\lambda)=\sum_{j=1}^{N}\psi_{j}(A_{j}+\underline{c}_{1}+\lambda)^{-1}[A_{j},\phi_{j}] \quad \text{and} \quad R(\lambda)=\sum_{j=1}^{N}\psi_{j}(A_{j}+\underline{c}_{1}+\lambda)^{-1}\phi_{j}.
$$
Let $\lambda_{1},\dots,\lambda_{\eta}\in S_{\theta}\backslash\{0\}$, $u_{1},\dots,u_{\eta}\in X_{0}^{0}$, $\eta\in\mathbb{N}$, and let $\{\epsilon_{k}\}_{k\in\mathbb{N}}$ be the sequence of Rademacher functions. We have
\begin{equation}\label{rsecestfin}
\|\sum_{i=1}^{\eta}\epsilon_{i}\lambda_{i}(c_{2}-\underline{\Delta}_{\mathfrak{e},0}+\lambda_{i})^{-1}u_{i}\|_{L^{2}(0,1;X_{0}^{0})}\leq\sum_{k=0}^{\infty}\|\sum_{i=1}^{\eta}\epsilon_{i}\lambda_{i}Q^{k}(\lambda_{i})R(\lambda_{i})u_{i}\|_{L^{2}(0,1;X_{0}^{0})}.
\end{equation}
By \cite[Lemma 2.6]{RS2}, the $0$-th term of the infinite sum on the right-hand side of \eqref{rsecestfin} is
\begin{equation}\label{rserest1}
\leq C_{0}^{2}NK(1+2/S(\theta))\|\sum_{i=1}^{\eta}\epsilon_{i}u_{i}\|_{L^{2}(0,1;X_{0}^{0})},
\end{equation}
where $S(\theta)$ is defined in \eqref{Stheta} and
$$
C_{0}=\max_{j\in\{1,\dots,N\}}\{\|\psi_{j}\cdot\|_{\mathcal{L}(E_{0}^{0},X_{0}^{0})}, \|\phi_{j}\cdot\|_{\mathcal{L}(X_{0}^{0},E_{0}^{0})}\}.
$$
On the other hand, for each $k\in\mathbb{N}$ and $i\in\{1,\dots,\eta\}$, the term $Q^{k}(\lambda_{i})R(\lambda_{i})$ in \eqref{rsecestfin} is a sum of $N^{k+1}$ terms of the form
\begin{eqnarray*}
\lefteqn{\psi_{j_{1}}(A_{j_{1}}+\underline{c}_{1}+\lambda_{i})^{-1}[A_{j_{1}},\phi_{j_{1}}]}\\
&&\psi_{j_{2}}A_{j_{2}}^{-b}A_{j_{2}}^{b}(A_{j_{2}}+\underline{c}_{1})^{-1}(A_{j_{2}}+\underline{c}_{1})(A_{j_{2}}+\underline{c}_{1}+\lambda_{i})^{-1}[A_{j_{2}},\phi_{j_{2}}]\\
&&\dots\psi_{j_{k}}A_{j_{k}}^{-b}A_{j_{k}}^{b}(A_{j_{k}}+\underline{c}_{1})^{-1}(A_{j_{k}}+\underline{c}_{1})(A_{j_{k}}+\underline{c}_{1}+\lambda_{i})^{-1}[A_{j_{k}},\phi_{j_{k}}]\\
&&\psi_{\ell}A_{\ell}^{-1}A_{\ell}(A_{\ell}+\underline{c}_{1})^{-1}(A_{\ell}+\underline{c}_{1})(A_{\ell}+\underline{c}_{1}+\lambda_{i})^{-1}\phi_{\ell},
\end{eqnarray*}
where $j_{1},\dots,j_{k},\ell\in \{1,\dots,N\}$. By Remark \ref{TaLnLem}, we can choose $\underline{c}_{1}>0$ sufficiently large such that
$$
\|A_{\ell}^{b}(A_{\ell}+\underline{c}_{1})^{-1}\|_{\mathcal{L}(E_{0}^{0})}<\varepsilon,
$$ 
for any $\ell\in\{1,\dots,N\}$. Moreover, denote
$$
C_{1}=\sup_{j,\ell\in\{1,\dots,N\}}\|[A_{j},\phi_{j}]\psi_{\ell}A_{\ell}^{-1}\|_{\mathcal{L}(E_{0}^{0})} \quad \text{and} \quad C_{2}=\sup_{j,\ell\in\{1,\dots,N\}}\|[A_{j},\phi_{j}]\psi_{\ell}A_{\ell}^{-b}\|_{\mathcal{L}(E_{0}^{0})}.
$$
By using successively the $R$-sectoriality of $A_{\ell}$ and \cite[Lemma 2.6]{RS2}, the $k$-th term of the infinite sum on the right-hand side of \eqref{rsecestfin} is 
\begin{equation}\label{rserest2}
\leq (K+1)C_{1}C_{0}^{2}(\varepsilon C_{2})^{k-1}(N(K(1+2/S(\theta))+1))^{k+1}\|\sum_{i=1}^{\eta}\epsilon_{i}u_{i}\|_{L^{2}(0,1;X_{0}^{0})}.
\end{equation}
The result follows by \eqref{rsecestfin}, \eqref{rserest1} and \eqref{rserest2}, after taking $\varepsilon>0$ sufficiently small. \\
{\em Case $r\in\mathbb{N}_{0}$}. We proceed by induction. Assume that the result holds for some $r\in\mathbb{N}_{0}$. In particular, there exists a $c>0$ such that $c-\underline{\Delta}_{\mathfrak{e},r}\in\mathcal{R}(\theta)$. Let $d\in(1,\infty)$, $T>0$, $h\in L^{d}(0,T;X_{0}^{r+1})$ and consider the problem 
\begin{eqnarray}\label{aqlpind1}
v_{t}(t)+(c-\Delta_{\mathfrak{e}})v(t)&=&h(t), \quad t\in(0,T),\\\label{aqlpind2}
v(0)&=&0.
\end{eqnarray}
Since by Lemma \ref{propF} (i) the space $X_{0}^{r}$ is UMD, due to Theorem \ref{KaWeTh} there exists a unique 
\begin{equation}\label{regofv}
v\in H_{d}^{1}(0,T;X_{0}^{r})\cap L^{d}(0,T;X_{1}^{r})
\end{equation}
solving \eqref{aqlpind1}-\eqref{aqlpind2}. 

For each $j\in\{1,\dots,N\}$, let $(s_{1},\dots,s_{n})$ be local coordinates in $U_{2\tau}(z_{j})$ and denote $M_{j,k}=\phi_{j}\partial_{s_{k}}$, $(j,k)\in\{1,\dots,N\}\times \{1,\dots,n\}$. By applying $M_{j,k}$ to \eqref{aqlpind1}-\eqref{aqlpind2}, we obtain
\begin{eqnarray}\label{aqlpifff}
\hspace{30pt}(M_{j,k}v)_{t}(t)+(c-\Delta_{\mathfrak{e}})(M_{j,k}v)(t)&=&[M_{j,k},\Delta_{\mathfrak{e}}]v(t)+(M_{j,k}h)(t), \quad t\in(0,T),\\\label{aqlphhhj}
(M_{j,k}v)(0)&=&0.
\end{eqnarray}
By \eqref{deltaF}, in local coordinates $(z,x,y)\in \mathcal{V}\times(0,1]\times\mathcal{Y}$, we have
$$
[M_{j,k},\Delta_{\mathfrak{e}}]=[M_{j,k},\Delta_{\mathfrak{m}}]+x^{-2}[M_{j,k},\frac{x\partial_{x}\det[\mathfrak{h}(z,x)]}{2\det[\mathfrak{h}(z,x)]}(x\partial_{x})+\Delta_{\mathfrak{h}(z,x)}]+[M_{j,k},D],
$$ 
where in the case of $\nu=1$, the middle term on the right-hand side of the above equation is zero. Therefore, $[M_{j,k},\Delta_{\mathfrak{e}}]$ can be written as a sum of a second order differential operator with smooth coefficients on $\mathbb{R}^{n}$ and a smooth in $z\in\mathbb{R}^{n}$ family of second order cone differential operators on $\mathcal{B}$. Hence, by \eqref{regofv} we obtain
$$
[M_{j,k},\Delta_{\mathfrak{e}}]v\in L^{d}(0,T;X_{0}^{r}).
$$
Thus, the right hand side of \eqref{aqlpifff} belongs to $ L^{q}(0,T;X_{0}^{r})$ and, due to the assumption, we conclude that
$$
M_{j,k}v\in H_{d}^{1}(0,T;X_{0}^{r})\cap L^{d}(0,T;X_{1}^{r}).
$$
Combined with \eqref{regofv} we find that
$$
v\in H_{d}^{1}(0,T;X_{0}^{r+1})\cap L^{d}(0,T;X_{1}^{r+1}),
$$
i.e. $c-\underline{\Delta}_{\mathfrak{e},r+1}$ has maximal $L^{d}$-regularity. Then, the result follows by the characterization \cite[Theorem 4.2]{Weis}.\\
{\em Case $r\geq0$}. We proceed by interpolation. First, concerning the solution \eqref{regofv} of \eqref{aqlpind1}-\eqref{aqlpind2}, we recall the following approach. Denote by $B_{r}$ the operator $\partial_{t}$ in $L^{d}(0,T;X_{0}^{r})$ with domain 
$$
\mathcal{D}(B_{r})=\{u\in H_{d}^{1}(0,T;X_{0}^{r})\, |\, u(0)=0\}.
$$ 
Since $X_{0}^{r}$ is UMD, by \cite[Theorem 8.5.8]{Ha} for each $\phi\in(0,\pi/2)$ we have that $B_{r}\in\mathcal{H}^{\infty}(\phi)$. Since $B_{r}$ and $\Lambda_{r}=c-\underline{\Delta}_{\mathfrak{e},r}$ are resolvent commuting in the sense of \cite[(III.4.9.1)]{Am}, by \cite[Theorem 6.3]{KaW} the inverse $(B_{r}+\Lambda_{r})^{-1}$ exists as a bounded map from $L^{d}(0,T;X_{0}^{r})$ to $H_{d}^{1}(0,T;X_{0}^{r})\cap L^{d}(0,T;X_{1}^{r})$. In particular
$$
v=(B_{r}+\Lambda_{r})^{-1}h.
$$
Let $r=k+\xi$, where $k\in\mathbb{N}_{0}$ and $\xi\in(0,1)$. We have that 
$$
(B_{k}+\Lambda_{k})^{-1}\in \mathcal{L}(L^{d}(0,T;X_{0}^{k}),H_{d}^{1}(0,T;X_{0}^{k})\cap L^{d}(0,T;X_{1}^{k}))
$$
and
$$
(B_{k}+\Lambda_{k})^{-1}|_{L^{d}(0,T;X_{0}^{k+1})}\in \mathcal{L}(L^{d}(0,T;X_{0}^{k+1}),H_{d}^{1}(0,T;X_{0}^{k+1})\cap L^{d}(0,T;X_{1}^{k+1})),
$$
where we have used uniqueness of solution for \eqref{aqlpind1}-\eqref{aqlpind2}. Therefore, by \cite[Theorem 2.6]{Lunar18}, we also have 
\begin{equation}\label{invermapint}
(B_{k}+\Lambda_{k})^{-1}\in \mathcal{L}(Y_{0},Y_{1})
\end{equation}
where
$$
Y_{0}=[L^{d}(0,T;X_{0}^{k}),L^{d}(0,T;X_{0}^{k+1})]_{\xi} 
$$
and
$$ 
Y_{1}=[H_{d}^{1}(0,T;X_{0}^{k})\cap L^{d}(0,T;X_{1}^{k}),H_{d}^{1}(0,T;X_{0}^{k+1})\cap L^{d}(0,T;X_{1}^{k+1})]_{\xi}.
$$
Clearly,
\begin{equation}\label{Y1embed}
Y_{1}\hookrightarrow [H_{d}^{1}(0,T;X_{0}^{k}),H_{d}^{1}(0,T;X_{0}^{k+1})]_{\xi}\cap [L^{d}(0,T;X_{1}^{k}),L^{d}(0,T;X_{1}^{k+1})]_{\xi}.
\end{equation}
Let $\ell,i\in\{0,1\}$ and $j\in\{k,k+1\}$. The restriction operator
$$
R_{T}:H_{d}^{\ell}(\mathbb{R};X_{i}^{j})\rightarrow H_{d}^{\ell}(0,T;X_{i}^{j}), \quad \text{defined by} \quad u\mapsto u|_{[0,T]},
$$
is a bounded map. Hence, by Lemma \ref{propF} (i), \cite[Theorem VII.4.5.5]{Am2} and \cite[Theorem 2.6]{Lunar18}, it induces a bounded map 
\begin{equation}\label{RTdeff}
R_{T}:H_{d}^{\ell}(\mathbb{R};[X_{i}^{k},X_{i}^{k+1}]_{\xi})\rightarrow [H_{d}^{\ell}(0,T;X_{i}^{k}),H_{d}^{\ell}(0,T;X_{i}^{k+1})]_{\xi}.
\end{equation}
On the other hand, let $d_{T}\in C^{\infty}(\mathbb{R};[0,1])$ such that $d_{T}=1$ on $[0,T]$ and $d_{T}=0$ outside $[-\varepsilon,T+\varepsilon]$, for some $\varepsilon\in(0,1)$. Then, the multiplication operator
$$
d_{T}\cdot :H_{d}^{\ell}(-1,T+1;X_{i}^{j}) \rightarrow H_{d}^{\ell}(\mathbb{R};X_{i}^{j}), \quad \text{defined by} \quad u\mapsto d_{T}u,
$$
is bounded. By Lemma \ref{propF} (i), \cite[Theorem VII.4.5.5]{Am2} and \cite[Theorem 2.6]{Lunar18}, it induces a bounded map
$$
d_{T}:[H_{d}^{\ell}(-1,T+1;X_{i}^{k}),H_{d}^{\ell}(-1,T+1;X_{i}^{k+1})]_{\xi}\rightarrow H_{d}^{\ell}(\mathbb{R};[X_{i}^{k},X_{i}^{k+1}]_{\xi}).
$$
If $u\in [H_{d}^{\ell}(0,T;X_{i}^{k}),H_{d}^{\ell}(0,T;X_{i}^{k+1})]_{\xi}$ denote by $E_{T}u$ the extension of $u$ from $[0,T]$ to $[-1,T+1]$ by zero if $\ell=0$ and by $\lim_{t\rightarrow0^{+}}u(t)$ on the left and $\lim_{t\rightarrow T^{-}}u(t)$ on the right of $[0,T]$ if $\ell=1$. By writing $u=R_{T}d_{T}E_{T}u$, we conclude that $R_{T}$ in \eqref{RTdeff} is a surjection. Consequently,
\begin{equation}\label{0Tint}
[H_{d}^{\ell}(0,T;X_{i}^{k}),H_{d}^{\ell}(0,T;X_{i}^{k+1})]_{\xi}=H_{d}^{\ell}(0,T;[X_{i}^{k},X_{i}^{k+1}]_{\xi}).
\end{equation}
Moreover, by Lemma \ref{propF} (iii) we have
\begin{equation}\label{intXsp1}
[X_{0}^{k},X_{0}^{k+1}]_{\xi}=X_{0}^{k+\xi}
\end{equation}
and
\begin{eqnarray}\nonumber
[X_{1}^{k},X_{1}^{k+1}]_{\xi}&\hookrightarrow& [\mathcal{H}_{q,p,\gamma}^{k+2,s}(\mathbb{F}),\mathcal{H}_{q,p,\gamma}^{k+3,s}(\mathbb{F})]_{\xi}\cap [\mathcal{H}_{q,p,\gamma+2}^{k,s+2}(\mathbb{F})_{\oplus\mathbb{C}_{\omega}},\mathcal{H}_{q,p,\gamma+2}^{k+1,s+2}(\mathbb{F})_{\oplus\mathbb{C}_{\omega}}]_{\xi}\\\label{intXsp2}
&=&\mathcal{H}_{q,p,\gamma}^{k+2+\xi,s}(\mathbb{F})\cap \mathcal{H}_{q,p,\gamma+2}^{k+\xi,s+2}(\mathbb{F})_{\oplus\mathbb{C}_{\omega}}=X_{1}^{k+\xi}.
\end{eqnarray}
Thus, by \eqref{invermapint}, \eqref{Y1embed}, \eqref{0Tint}, \eqref{intXsp1} and \eqref{intXsp2} we obtain 
$$
(B_{k}+\Lambda_{k})^{-1}\in \mathcal{L}(L^{d}(0,T;X_{0}^{k+\xi}), H_{d}^{1}(0,T;X_{0}^{k+\xi})\cap L^{d}(0,T;X_{1}^{k+\xi})),
$$
which completes the proof. \mbox{\ } \hfill $\Box$

\section{Application: The porous medium equation on manifolds with edges}

In this section we will prove Theorem \ref{pmeonF}. Existence, uniqueness and regularity for the PME will follow by Theorem \ref{ClementLi}. Setting $u^{m}=v$ in \eqref{PME1}-\eqref{PME2} we obtain
\begin{eqnarray}\label{PME3}
\hspace{7mm} v'(t)-mv^{\frac{m-1}{m}}\Delta_{\mathfrak{e}}v(t) &=& \p(v(t),t),\quad t\in(0,T],\\\label{PME4}
\hspace{7mm} v(0)&=&u_{0}^{m},
\end{eqnarray}
where $\p(v,t)=mv^{(m-1)/m}f(v^{1/m},t)$. We start by showing maximal $L^{d}$-regularity for the linearization of \eqref{PME3}-\eqref{PME4}. For any $x\geq0$ denote 
$$
[x]=\bigg\{\begin{array}{lll} x & \text{if} & x\in \mathbb{N}_{0} \\ 
\min\{k\, |\, k\in\mathbb{N} \,\,\text{and} \,\, x<k\} & \text{if} & x\notin\mathbb{N}_{0}.
\end{array}
$$

\begin{theorem}\label{maxreglinterm}
Let $p,q\in(1,\infty)$, $r,s\geq0$, $\gamma$ be as in \eqref{gamma2}, $\delta>0$ and 
$$
u\in \mathcal{H}_{q,p,\frac{\nu}{2}+\delta}^{[r]+1+\frac{n}{q}+\delta,[s]+\frac{\nu}{p}+\delta}(\mathbb{F})_{\oplus \mathbb{C}_{\omega}} 
$$ 
satisfying $u\geq \alpha >0$, for some constant $\alpha>0$. Then, there exist $\theta_{0}\in(\pi/2,\pi)$ and $c_{0}>0$ such that $c_{0}-u\underline{\Delta}_{\mathfrak{e}}\in\mathcal{R}(\theta_{0})$, where $\underline{\Delta}_{\mathfrak{e}}$ is the Laplacian \eqref{DeltaF}.
\end{theorem}
\begin{proof}
Note that, due to Lemma \ref{propF} (v), we have $u\in C(\mathbb{F})$, and therefore the assumption $u\geq \alpha$ makes sense. Denote $X_{0}^{r,s}=\mathcal{H}_{q,p,\gamma}^{r,s}(\mathbb{F})$, $X_{1}^{r,s}=\mathcal{H}_{q,p,\gamma}^{r+2,s}(\mathbb{F})\cap \mathcal{H}_{q,p,\gamma+2}^{r,s+2}(\mathbb{F})_{\oplus\mathbb{C}_{\omega}}$ and the Laplacian $\underline{\Delta}_{\mathfrak{e}}$ defined in \eqref{DeltaF} by $\Delta_{r,s}$. We start with the following observation. Let $\theta\in [0,\pi)$, $\lambda_{1},\dots,\lambda_{\xi}\in S_{\theta}\backslash\{0\}$, $\xi\in\mathbb{N}$, $x_{1},\dots,x_{\xi}\in X_{0}^{r,s}$ and let $\{\epsilon_{j}\}_{j\in\mathbb{N}}$ be the sequence of the Rademacher functions. For any $\widetilde{w}\in \mathbb{F}$, by Theorem \ref{thsec33} and Kahane's contraction principle \cite[Proposition 2.5]{KW1}, there exists a $C_{0}\geq1$ such that 
\begin{eqnarray*}
\lefteqn{\|\sum_{j=1}^{\xi}\epsilon_{j}\lambda_{j}(c_{1}-u(\widetilde{w})\Delta_{r,s}+\lambda_{j})^{-1}x_{j}\|_{L^{2}(0,1;X_{0}^{r,s})}}\\
&=&\|\sum_{j=1}^{\xi}\epsilon_{j}\frac{c_{1}-cu(\widetilde{w})+\lambda_{j}}{u(\widetilde{w})}\\
&&\times\Big(c-\Delta_{r,s}+\frac{c_{1}-cu(\widetilde{w})+\lambda_{j}}{u(\widetilde{w})}\Big)^{-1}\frac{\lambda_{j}}{c_{1}-cu(\widetilde{w})+\lambda_{j}}x_{j}\|_{L^{2}(0,1;X_{0}^{r,s})}\\
&\leq&C_{0}\|\sum_{j=1}^{\xi}\epsilon_{j}\frac{\lambda_{j}}{c_{1}-cu(\widetilde{w})+\lambda_{j}}x_{j}\|_{L^{2}(0,1;X_{0}^{r,s})}\\
&\leq& 2C_{0}\sup_{\lambda\in S_{\theta}\backslash\{0\}}\Big(\frac{\lambda}{c_{1}-cu(\widetilde{w})+\lambda}\Big)\|\sum_{j=1}^{\xi}\epsilon_{j}x_{j}\|_{L^{2}(0,1;X_{0}^{r,s})},
\end{eqnarray*}
where $c>0$ is as in Theorem \ref{thsec33} and $c_{1}$ satisfies $c_{1}>c\max_{w\in\mathbb{F}}(u(w))$. This implies that $c_{1}-u(\widetilde{w})\Delta_{r,s}: X_{1}^{r,s}\rightarrow X_{0}^{r,s}$ belongs to $\mathcal{R}(\theta)$ and its $R$-sectorial bound can be chosen independently of $\widetilde{w}\in \mathbb{F}$. Next, we split the proof in several steps according to different values of $r$ and $s$.\\
{\em Case $r=s=0$}. We follow similar steps as in the case $s=0$ in the proof of \cite[Theorem 6.1]{RS2}; however, the adaptation of the proof requires several important modifications, which make the full demonstration necessary. Let $\mathfrak{r}$ be a smooth non-singular Riemannian metric on $\mathcal{F}$ and denote by $U_{\tau}(w)$ the open geodesic ball of radius $\tau>0$ in $(\mathcal{F},\mathfrak{r})$, centered at $w\in \mathcal{F}$. Consider an open cover $\{U_{\tau}(w_{j})\}_{j\in\{1,\dots,N\}}$ of $(\mathcal{F},\mathfrak{r})$ by coordinate charts, where $w_{j}\in\partial\mathcal{F}$, $j\in \{1,\dots,N_{0}\}$, and $w_{j}\in\mathcal{F}\backslash\partial\mathcal{F}$ when $j\in\{N_{0}+1,\dots,N\}$, for some $N,N_{0}\in\mathbb{N}$ with $1<N_{0}<N$. We have $w_{j}=(z_{j},0,y_{j})\in\partial\mathcal{F}$, $j\in \{1,\dots,N_{0}\}$, for some $z_{j}\in \mathcal{V}$ and $y_{j}\in \mathcal{Y}$. Assume that each of $\{\overline{U_{3\tau/2}(w_{j})}\}$, $j\in\{N_{0}+1,\dots,N\}$, does not intersect the boundary of $(\mathcal{F},\mathfrak{r})$. Let $\chi:\mathbb{R}\rightarrow [0,1]$ be a smooth non-increasing function that equals to $1$ on $(-\infty,1/2]$ and $0$ on $[3/4,\infty)$. For any $j\in\{1,\dots,N\}$ define
$$
u_{j}(w)=\chi\Big(\frac{d(w,w_{j})}{2\tau}\Big)u(w)+\big(1-\chi\Big(\frac{d(w,w_{j})}{2\tau}\Big)\big)u(w_{j}), \quad w\in\mathbb{F},
$$
where $d(\cdot,\cdot)$ stands for the geodesic distance in $(\mathcal{F},\mathfrak{r})$, and let 
$$
c_{1}-u_{j}\Delta_{0,0}=c_{1}-u(w_{j})\Delta_{0,0}+(u(w_{j})-u_{j})\Delta_{0,0}: X_{1}^{0,0}\rightarrow X_{0}^{0,0}.
$$
By choosing $\tau$ sufficiently small, and possibly $N$ large enough, we can make the $\mathcal{L}(X_{1}^{0,0},X_{0}^{0,0})$-norm of each $(u(w_{j})-u_{j})\Delta_{0,0}$ arbitrary small. Therefore, by the perturbation result \cite[Theorem 1]{KW}, there exists a $K\geq1$ such that $c_{1}-u_{j}\Delta_{0,0}\in\mathcal{R}(K,\theta)$, for each $j\in\{1,\dots,N\}$.

Let $\{\varphi_{j}\}_{j\in\{1,\dots,N\}}$ be a partition of unity subordinate to $\{U_{\tau}(w_{j})\}_{j\in\{1,\dots,N\}}$, and let $\{\psi_{j}\}_{j\in\{1,\dots,N\}}$ be a collection of smooth functions on $\mathbb{F}$ such that, for each $j\in\{1,\dots,N\}$, $\psi_{j}$ is supported on $U_{\tau}(w_{j})$ and satisfies $\psi_{j}=1$ on $\mathrm{supp}(\varphi_{j})$. We choose each $\varphi_{j}$, $j\in\{1,\dots,N_{0}\}$, to be independent of $(x,y)\in [0,\varepsilon_{0}]\times \mathcal{Y}$, for some $\varepsilon_{0}\in(0,1)$, i.e. to depend only on $z\in\mathcal{V}$ close $\partial\mathcal{F}$. Moreover, let $f\in X_{1}^{0,0}$, $g\in X_{0}^{0,0}$, $\lambda\in S_{\theta}$, $c_{2}>0$ and consider the equation
$$
(\lambda+c_{3}-u\Delta_{0,0})f=g,
$$ 
where $c_{3}=c_{1}+c_{2}$. Following similar steps to those leading up to equation \eqref{lakjd}, we obtain
\begin{equation}\label{Lftinv}
f=\sum_{j=1}^{N}\psi_{j}(\lambda+c_{3}-u_{j}\Delta_{0,0})^{-1}\varphi_{j}g+\sum_{j=1}^{N}\psi_{j}(\lambda+c_{3}-u_{j}\Delta_{0,0})^{-1}u[\varphi_{j},\Delta_{0,0}]f.
\end{equation}

In local coordinates $(z,x,y)\in \mathcal{V}\times[0,1/2]\times \mathcal{Y}$ we have
\begin{eqnarray}\nonumber
[\Delta_{0,0},\varphi_{j}]&=&[\Delta_{\mathfrak{m}},\varphi_{j}]+2(\partial_{x}\varphi_{j})\partial_{x}+(\partial_{x}^{2}\varphi_{j})+[D,\varphi_{j}]\\\label{expcomut}
&&+\Bigg\{\begin{array}{lll} 0 & \text{if}& \nu=1\\
x^{-1}\Big(\nu-1+\frac{x\partial_{x}\det[\mathfrak{h}(z,x)]}{2\det[\mathfrak{h}(z,x)]}\Big)(\partial_{x}\varphi_{j})+x^{-2}[\Delta_{\mathfrak{h}(z,x)},\varphi_{j}] & \text{if}& \nu>1,
\end{array}
\end{eqnarray}
where $\Delta_{\mathfrak{m}}$, $D$, $\mathfrak{h}(z,x)$ and $\Delta_{\mathfrak{h}(z,x)}$ are defined in \eqref{deltaF}-\eqref{Deltaz}. Due to the choice of $\varphi_{j}$, $j\in\{1,\dots,N_{0}\}$, and the fact that $[\Delta_{\mathfrak{m}},\varphi_{j}]$ is a first order differential operator with smooth coefficients on $\mathbb{V}$, by \eqref{domemb} with the choices \eqref{phochoices}, we have that each $[\Delta_{0,0},\varphi_{j}]$ induces a bounded map from $X_{1}^{0,0}$ to $\mathcal{H}_{q,p,\gamma+2\varepsilon}^{2\varepsilon,2\varepsilon}(\mathbb{F})_{\oplus\mathbb{C}_{\omega}}$, for certain $\varepsilon>0$ sufficiently small. Furthermore, by Lemma \ref{propF} (vii), for such $\varepsilon$, multiplication by $u$ induces a bounded map on $\mathcal{H}_{q,p,\gamma+2\varepsilon}^{2\varepsilon,2\varepsilon}(\mathbb{F})_{\oplus\mathbb{C}_{\omega}}$. On the other hand, by Lemma \ref{propF} (iii), for any $\eta\in(0,1)$ we have
$$
\mathcal{H}_{q,p,\gamma+2\eta+\varepsilon}^{2\eta+\varepsilon,2\eta+\varepsilon}(\mathbb{F})_{\oplus\mathbb{C}_{\omega}} \hookrightarrow [\mathcal{H}_{q,p,\gamma}^{0,0}(\mathbb{F}),\mathcal{H}_{q,p,\gamma+2}^{2,2}(\mathbb{F})_{\oplus \mathbb{C}_{\omega}}]_{\eta}\hookrightarrow [X_{0}^{0,0},X_{1}^{0,0}]_{\eta}.
$$
Therefore, by choosing $\eta=\varepsilon/2$ and using \cite[(I.2.5.2) and (I.2.9.6)]{Am}, we deduce that 
$$
u[\Delta_{0,0},\varphi_{j}]\in \mathcal{L}(X_{1}^{0,0},\mathcal{D}((c_{1}-u_{j}\Delta_{0,0})^{a})),\quad j\in\{1,\dots,N\},
$$
for certain $a\in(0,1/2)$. Consequently, we can write
\begin{eqnarray*}
\lefteqn{(\lambda+c_{3}-u_{j}\Delta_{0,0})^{-1}u[\varphi_{j},\Delta_{0,0}]}\\
&=&(\lambda+c_{3}-u_{j}\Delta_{0,0})^{-1}(c_{1}-u_{j}\Delta_{0,0})^{-a}(c_{1}-u_{j}\Delta_{0,0})^{a}u[\varphi_{j},\Delta_{0,0}]
\end{eqnarray*}
and, due to Remark \ref{TaLnLem}, we can see that the $\mathcal{L}(X_{1}^{0,0})$-norm of the second term on the right-hand side of \eqref{Lftinv} becomes arbitrary small, uniformly in $\lambda$, by choosing $c_{2}$ large enough. Hence, if we choose $c_{2}$ sufficiently large, we deduce that $c_{3}-u\Delta_{0,0}+\lambda$ has a left inverse $L$ that belongs to $\mathcal{L}(X_{0}^{0,0},X_{1}^{0,0})$.

By \eqref{Lftinv}, similarly to \eqref{rightinvrs}, we have
\begin{equation}\label{rightnverse}
(\lambda+c_{3}-u\Delta_{0,0})L=I+\sum_{j=1}^{N}u[\psi_{j},\Delta_{0,0}](\lambda+c_{3}-u_{j}\Delta_{0,0})^{-1}(\varphi_{j}+u[\varphi_{j},\Delta_{0,0}]L).
\end{equation}
For each $b\in(\eta,1)$ and $j\in\{1,\dots,N\}$, by \cite[(I.2.5.2) and (I.2.9.6)]{Am} we have $\mathcal{D}((c_{1}-u_{j}\Delta_{0,0})^{b})\hookrightarrow [X_{0}^{0,0},X_{1}^{0,0}]_{\eta}$, so that, by Lemma \ref{Hprop} (iii) and \eqref{domemb}, $[\psi_{j},\Delta_{0,0}]$ maps $\mathcal{D}((c_{1}-u_{j}\Delta_{0,0})^{b})$ continuously to 
\begin{eqnarray*}
\lefteqn{[\psi_{j},\Delta_{0,0}]\big([ H_{q}^{0}(\mathbb{R}^{n};\mathcal{H}_{p}^{0,\gamma}(\mathbb{B})),H_{q}^{2\rho_{1}}(\mathbb{R}^{n};\mathcal{H}_{p}^{2\rho_{2},\gamma+2\rho_{2}}(\mathbb{B})\oplus\mathbb{C}_{\omega})]_{\eta})}\\
&&\hspace{14pt}\hookrightarrow [\psi_{j},\Delta_{0,0}]\big(H_{q}^{2\eta\rho_{1}-\varepsilon}(\mathbb{R}^{n};\mathcal{H}_{p}^{2\eta\rho_{2}-\varepsilon,\gamma+2\eta\rho_{2}-\varepsilon}(\mathbb{B})\oplus\mathbb{C}_{\omega})\big).
\end{eqnarray*}
Taking $\varepsilon\in(0,1/8)$, $\eta=1-\varepsilon$ and using the following two choices for $\rho_{1}$, $\rho_{2}$:
$$
(\rho_{1},\rho_{2})=(\varepsilon,1-2\varepsilon), \quad (\rho_{1},\rho_{2})=(1-2\varepsilon,\varepsilon),
$$
together with \eqref{expcomut}, we deduce that $[\psi_{j},\Delta_{0,0}]$ maps the domain $\mathcal{D}((c_{1}-u_{j}\Delta_{0,0})^{b})$ continuously to $H_{q}^{\varepsilon/2}(\mathbb{R}^{n};\mathcal{H}_{p}^{\varepsilon/2,\gamma+\varepsilon/2}(\mathbb{B})\oplus\mathbb{C}_{\omega})\hookrightarrow X_{0}^{0,0}$. Consequently, by Lemma \ref{ProbHR} (ii) and (v), each map 
\begin{equation}\label{bpowermap}
u[\psi_{j},\Delta_{0,0}]:\mathcal{D}((c_{1}-u_{j}\Delta_{0,0})^{b})\rightarrow X_{0}^{0,0}
\end{equation}
is bounded. Therefore, by writing 
\begin{eqnarray*}
\lefteqn{u[\psi_{j},\Delta_{0,0}](\lambda+c_{3}-u_{j}\Delta_{0,0})^{-1}}\\
&=&u[\psi_{j},\Delta_{0,0}](c_{1}-u_{j}\Delta_{0,0})^{-b}(c_{1}-u_{j}\Delta_{0,0})^{b}(\lambda+c_{3}-u_{j}\Delta_{0,0})^{-1}
\end{eqnarray*}
on the right-hand side of \eqref{rightnverse} and using Remark \ref{TaLnLem}, we see that, after choosing $c_{2}$ sufficiently large, $c_{3}-u\Delta_{0,0}+\lambda$ has also a right inverse that belongs to $\mathcal{L}(X_{0}^{0,0},X_{1}^{0,0})$. In addition, by \eqref{Lftinv} we express
\begin{equation}\label{invnewmannser}
(c_{3}-u\Delta_{0,0}+\lambda)^{-1}=\sum_{k=0}^{\infty}Q^{k}(\lambda)R(\lambda), \quad \lambda\in S_{\theta},
\end{equation}
where
$$
Q(\lambda)=\sum_{j=1}^{N}\psi_{j}(\lambda+c_{3}-u_{j}\Delta_{0,0})^{-1}u[\varphi_{j},\Delta_{0,0}]
$$
and
$$ 
R(\lambda)=\sum_{j=1}^{N}\psi_{j}(\lambda+c_{3}-u_{j}\Delta_{0,0})^{-1}\varphi_{j}, \quad \lambda\in S_{\theta}.
$$

Concerning $R$-sectoriality for $c_{3}-u\Delta_{0,0}$, based on \eqref{invnewmannser} we have
\begin{equation}\label{expforrsec}
\|\sum_{j=1}^{\xi}\epsilon_{j}\lambda_{j}(c_{3}-u\Delta_{0,0}+\lambda_{j})^{-1}x_{j}\|_{L^{2}(0,1;X_{0}^{0,0})}\leq\sum_{k=0}^{\infty}\|\sum_{j=1}^{\xi}\epsilon_{j}\lambda_{j}Q^{k}(\lambda_{j})R(\lambda_{j})x_{j}\|_{L^{2}(0,1;X_{0}^{0,0})}.
\end{equation}
By \cite[Lemma 2.6]{RS2}, the $0$-th term of the infinite sum on the right-hand side of \eqref{expforrsec} is estimated
\begin{equation}\label{kkskhfhf}
\leq C_{1}^{2}NK(1+2/S(\theta))\|\sum_{j=1}^{\xi}\epsilon_{j}x_{j}\|_{L^{2}(0,1;X_{0}^{0})},
\end{equation}
where $S(\theta)$ is defined in \eqref{Stheta} and
$$
C_{1}=\max_{j\in\{1,\dots,N\}}\{\|\psi_{j}\cdot\|_{\mathcal{L}(X_{0}^{0,0})}, \|\varphi_{j}\cdot\|_{\mathcal{L}(X_{0}^{0,0})}\}.
$$
Furthermore, for each $k\in\mathbb{N}$ and $j\in\{1,\dots,\xi\}$, the term $Q^{k}(\lambda_{j})R(\lambda_{j})$ in \eqref{expforrsec}, after replacing $c_{2}$ with $c_{2}+c_{4}$, for some $c_{4}>0$, is a sum of $N^{k+1}$ terms of the form
\begin{eqnarray*}
\lefteqn{\psi_{j_{1}}(A_{j_{1}}+c_{4}+\lambda_{j})^{-1}u[\varphi_{j_{1}},\Delta_{0,0}]}\\
&&\psi_{j_{2}}A_{j_{2}}^{-b}A_{j_{2}}^{b}(A_{j_{2}}+c_{4})^{-1}(A_{j_{2}}+c_{4})(A_{j_{2}}+c_{4}+\lambda_{j})^{-1}u[\varphi_{j_{2}},\Delta_{0,0}]\\
&&\dots\psi_{j_{k}}A_{j_{k}}^{-b}A_{j_{k}}^{b}(A_{j_{k}}+c_{4})^{-1}(A_{j_{k}}+c_{4})(A_{j_{k}}+c_{4}+\lambda_{j})^{-1}u[\varphi_{j_{k}},\Delta_{0,0}]\\
&&\psi_{\ell}A_{\ell}^{-1}A_{\ell}(A_{\ell}+c_{4})^{-1}(A_{\ell}+c_{4})(A_{\ell}+c_{4}+\lambda_{j})^{-1}\varphi_{\ell},
\end{eqnarray*}
where $j_{1},\dots,j_{k},\ell\in \{1,\dots,N\}$, $A_{\ell}=c_{3}-u_{\ell}\Delta_{0,0}$ and we have taken into account \eqref{bpowermap}. Due to Remark \ref{TaLnLem}, let $c_{4}$ be sufficiently large such that $\|A_{\ell}^{b}(A_{\ell}+c_{4})^{-1}\|_{X_{0}^{0,0}}<\varepsilon$ for each $\ell\in \{1,\dots,N\}$. Then, the $k$-th term of the infinite sum on the right-hand side of \eqref{expforrsec} becomes
\begin{equation}\label{jahagfs}
\leq \varepsilon^{k-1}C_{1}^{2}C_{2}^{k}N^{k+1}(1+K(1+2/S(\theta)))^{k+2},
\end{equation}
where
$$
C_{2}=\max_{j,\ell\in\{1,\dots,N\}}\{\|u[\varphi_{j},\Delta_{0,0}]\psi_{\ell}A_{\ell}^{-b}\|_{X_{0}^{0,0}}, \|u[\varphi_{j},\Delta_{0,0}]\psi_{\ell}A_{\ell}^{-1}\|_{X_{0}^{0,0}}\}.
$$
Hence, $R$-sectoriality for $c_{3}+c_{4}-u\Delta_{0,0}$ follows from \eqref{expforrsec}, \eqref{kkskhfhf} and \eqref{jahagfs}, after choosing $\varepsilon$ small enough.

{\em Case $r,s\in\mathbb{N}_{0}$}. Proceeding by induction, we assume that for some $r,s\in\mathbb{N}_{0}$ there exist $\theta_{1}\in(\pi/2,\pi)$ and $c_{5}>0$ such that $c_{5}-u\Delta_{r,s}\in\mathcal{R}(\theta_{1})$. Denoting by $(\widetilde{r},\widetilde{s})$ either $(r+1,s)$ or $(r,s+1)$, we will show that $c_{6}-u\Delta_{\widetilde{r},\widetilde{s}}\in\mathcal{R}(\theta_{2})$, for some $\theta_{2}\in(\pi/2,\pi)$ and $c_{6}>0$, provided that $u\in \mathcal{H}_{q,p,\nu/2+\delta}^{\widetilde{r}+1+n/p +\varepsilon,\widetilde{s}+\nu/p+\varepsilon}(\mathbb{F})_{\oplus \mathbb{C}_{\omega}}$. Let $(\zeta_{1},\dots,\zeta_{n+\nu})$ be local coordinates in $U_{\tau}(w_{j})$, $j\in\{1,\dots,N\}$, where, when $j\in\{1,\dots,N_{0}\}$ we choose 
$$
(\zeta_{1},\dots,\zeta_{n+\nu})=\bigg\{ \begin{array}{lll}(z_{1},\dots,z_{n},x) \in \mathbb{R}^{n}\times[0,1) & \text{if} & \nu=1\\
(z_{1},\dots,z_{n},x,y_{1},\dots,y_{\nu-1}) \in \mathbb{R}^{n}\times[0,1) \times \mathcal{Y} & \text{if} & \nu>1. \end{array}
$$
Let $W_{j,k}=\varphi_{j}\partial_{\zeta_{k}}$, $j\in\{1,\dots,N\}$, $k\in\{1,\dots,n\}$ when $(\widetilde{r},\widetilde{s})=(r+1,s)$ and $k\in\{n+1,\dots,n+\nu\}$ when $(\widetilde{r},\widetilde{s})=(r,s+1)$. Furthermore, when $j\in\{1,\dots,N_{0}\}$ and $(\widetilde{r},\widetilde{s})=(r,s+1)$, by $\partial_{\zeta_{n+1}}$ we mean $x\partial_{x}$. Since by Lemma \ref{propF} (i) the space $X_{0}^{r,s}$ is UMD, due to Theorem \ref{KaWeTh}, let
\begin{equation}\label{vregqs}
v\in H_{d}^{1}(0,T;X_{0}^{r,s})\cap L^{d}(0,T;X_{1}^{r,s})
\end{equation}
be the unique solution of 
\begin{eqnarray*}
v_{t}(t)+(c_{5}-u\Delta_{r,s})v(t)&=&h(t), \quad t\in(0,T),\\
v(0)&=&0,
\end{eqnarray*}
where $h\in L^{d}(0,T;X_{0}^{\widetilde{r},\widetilde{s}})$ and $T>0$. If we apply $W_{j,k}$, $j\in\{1,\dots,N\}$, $k\in\{1,\dots,n\}$ or $k\in\{n+1,\dots,n+\nu\}$, to the above equation we get
\begin{eqnarray}\nonumber
(W_{j,k}v)_{t}(t)+(c_{5}-u\Delta_{r,s})(W_{j,k}v)(t)&&\\\label{wacteq1}
&=&[W_{j,k},u\Delta_{r,s}]v(t)+(W_{j,k}h)(t), \quad t\in(0,T),\\\label{wacteq2}
(W_{j,k}v)(0)&=&0.
\end{eqnarray}
We have
\begin{equation}\label{commutatorud}
[W_{j,k},u\Delta_{r,s}]=\varphi_{j}(\partial_{\zeta_{k}}u)\Delta_{r,s}+u[\varphi_{j}\partial_{\zeta_{k}},\Delta_{r,s}],
\end{equation}
and, when $j\in\{1,\dots,N_{0}\}$, by \eqref{deltaF} we get
\begin{eqnarray}\nonumber
\lefteqn{[\varphi_{j}\partial_{\zeta_{k}},\Delta_{r,s}]=[\varphi_{j}\partial_{\zeta_{k}},\Delta_{\mathfrak{m}}]+[\varphi_{j}\partial_{\zeta_{k}},D]-(\partial_{x}^{2}\varphi_{j})\partial_{\zeta_{k}}-2(\partial_{x}\varphi_{j})\partial_{x}\partial_{\zeta_{k}}}\\\label{poastieo}
&&\hspace{3mm}+\varphi_{j}\partial_{\zeta_{k}}\Big(\frac{x\partial_{x}\det[\mathfrak{h}(z,x)]}{2\det[\mathfrak{h}(z,x)]}\Big)\frac{\partial_{x}}{x}+\varphi_{j}\Big(\nu-1+\frac{x\partial_{x}\det[\mathfrak{h}(z,x)]}{2\det[\mathfrak{h}(z,x)]}\Big)\partial_{\zeta_{k}}(x^{-1})\partial_{x}\\\nonumber
&&-(\partial_{x}\varphi_{j})\Big(\nu-1+\frac{x\partial_{x}\det[\mathfrak{h}(z,x)]}{2\det[\mathfrak{h}(z,x)]}\Big)\frac{\partial_{\zeta_{k}}}{x}+\varphi_{j}\partial_{\zeta_{k}}(x^{-2})\Delta_{\mathfrak{h}(z,x)}+\frac{1}{x^{2}}[\varphi_{j}\partial_{\zeta_{k}},\Delta_{\mathfrak{h}(z,x)}],
\end{eqnarray}
where, in the case of $\nu=1$, the last two lines in the above expression have to be ignored. 

Concerning the first term on the right-hand side of \eqref{commutatorud}, by Lemma \ref{propF} (vii), multiplication by $\varphi_{j}(\partial_{\zeta_{k}}u)$ induces a bounded map from $X_{0}^{r,s}$ to itself, so that this term is a bounded map from $X_{1}^{r,s}$ to $X_{0}^{r,s}$. Furthermore, by the choice of $\{\phi_{j}\}_{j\in\{1,\dots,N\}}$, Corollary \ref{sharpmixder} with $\rho=1/2$ and \eqref{poastieo}, we see that $[\varphi_{j}\partial_{\zeta_{k}},\Delta_{r,s}]$ induces a bounded map from $X_{1}^{r,s}$ to $X_{0}^{r,s}$. Hence, due to Lemma \ref{propF} (vii), the second term on the right-hand side of \eqref{commutatorud} is also a bounded map from $X_{1}^{r,s}$ to $X_{0}^{r,s}$. Therefore, $[W_{j,k},u\Delta_{r,s}]v\in L^{d}(0,T;X_{0}^{r,s})$. In addition, $W_{j,k}h\in L^{d}(0,T;X_{0}^{r,s})$, so that, by the assumption and \eqref{wacteq1}-\eqref{wacteq2}, we obtain $W_{j,k}v\in H_{d}^{1}(0,T;X_{0}^{r,s})\cap L^{d}(0,T;X_{1}^{r,s})$. This together with \eqref{vregqs} imply $v\in H_{d}^{1}(0,T;X_{0}^{\widetilde{r},\widetilde{s}})\cap L^{d}(0,T;X_{1}^{\widetilde{r},\widetilde{s}})$, i.e. $c_{5}-u\Delta_{\widetilde{r},\widetilde{s}}$ has maximal $L^{d}$-regularity, and hence $c_{6}-u\Delta_{\widetilde{r},\widetilde{s}}\in\mathcal{R}(\theta_{2})$, for some $\theta_{2}\in(\pi/2,\pi)$ and $c_{6}>0$, due to \cite[Theorem 4.2]{Weis}.

{\em Case $r,s\geq0$}. The proof is the same as the {\em Case $r\geq0$} in the proof of Theorem \ref{thsec33}.
\end{proof}

Using Lemma \ref{Hprop} (ii), (iii) and Corollary \ref{sharpmixder}, the following embedding is valid.

\begin{lemma}\label{tracespace} 
Let $p,q,d\in(1,\infty)$, $r,s\geq0$ and $\gamma\in\mathbb{R}$. Then for any $\rho\in(0,1)$ and $\varepsilon>0$ we have
$$
(\mathcal{H}_{q,p,\gamma}^{r+2,s}(\mathbb{F})\cap \mathcal{H}_{q,p,\gamma+2}^{r,s+2}(\mathbb{F})_{\oplus\mathbb{C}_{\omega}},\mathcal{H}_{q,p,\gamma}^{r,s}(\mathbb{F}))_{\frac{1}{d},d}\hookrightarrow \mathcal{H}_{q,p,\gamma+2\rho(1-\frac{1}{d})-\varepsilon}^{r+2(1-\rho)(1-\frac{1}{d})-\varepsilon,s+2\rho(1-\frac{1}{d})-\varepsilon}(\mathbb{F})_{\oplus\mathbb{C}_{\omega}}.
$$
\end{lemma}

{\bf Proof of Theorem \ref{pmeonF}.} We apply Theorem \ref{ClementLi} to \eqref{PME3}-\eqref{PME4} with $X_{0}=\mathcal{H}_{q,p,\gamma}^{r,s}(\mathbb{F})$, $X_{1}=\mathcal{H}_{q,p,\gamma}^{r+2,s}(\mathbb{F})\cap \mathcal{H}_{q,p,\gamma+2}^{r,s+2}(\mathbb{F})_{\oplus\mathbb{C}_{\omega}}$, $A(v)u=-mv^{\frac{m-1}{m}}\Delta u$ and $F(v,t)=\p(v,t)$. We have 
\begin{equation}\label{u0inH}
u_{0}^{m},u_{0}^{\frac{m-1}{m}}\in X_{1}\hookrightarrow (X_{1},X_{0})_{1/d,d}.
\end{equation}
Furthermore, by Lemma \ref{propF} (v) and Lemma \ref{tracespace} with $\rho=(1-\delta)/2$, for all $\varepsilon>0$ small enough we have
\begin{equation}\label{embedtrace}
(X_{1},X_{0})_{\frac{1}{d},d}\hookrightarrow \mathcal{H}_{q,p,\gamma+(1-\delta)(1-\frac{1}{d})-\varepsilon}^{r+(1+\delta)(1-\frac{1}{d})-\varepsilon,s+(1-\delta)(1-\frac{1}{d})-\varepsilon}(\mathbb{F})_{\oplus\mathbb{C}_{\omega}}\hookrightarrow C(\mathbb{F}).
\end{equation}
Let $U$ be an open ball in $(X_{1},X_{0})_{1/d,d}$ centered at $u_{0}^{m}$. For any $v\in U$ we have
$$
|u_{0}^{m}-v|\leq \|u_{0}^{m}-v\|_{C(\mathbb{F})}\leq C_{1}\|u_{0}^{m}-v\|_{(X_{1},X_{0})_{\frac{1}{d},d}},
$$ 
for certain $C_{1}>0$. Therefore, we can choose the radius of $U$ sufficiently small such that $\mathrm{Re}(v)>a$ for all $v\in U$ and certain $a>0$. Moreover, let $\Gamma_{1}$ be a simple closed path in $\{z\in \mathbb{C}\, |\, \mathrm{Re}(z)>0\}$ around $\cup_{v\in U}\mathrm{Ran}(v)$. We check the assumptions (H1)-(H3) of Theorem \ref{ClementLi}.

(H1) If $v_{1},v_{2}\in U$, then for all sufficiently small $\varepsilon>0$ we have
\begin{eqnarray*}
\lefteqn{\|A(v_{1})-A(v_{2})\|_{\mathcal{L}(X_{1},X_{0})}\leq C_{2}\|(v_{1}^{\frac{m-1}{m}}-v_{2}^{\frac{m-1}{m}})\cdot\|_{\mathcal{L}(X_{0})}}\\
&&\leq C_{3}\|v_{1}-v_{2}\|_{\mathcal{H}_{q,p,\gamma+(1-\delta)(1-\frac{1}{d})-\varepsilon}^{r+(1+\delta)(1-\frac{1}{d})-\varepsilon,s+(1-\delta)(1-\frac{1}{d})-\varepsilon}(\mathbb{F})_{\oplus\mathbb{C}_{\omega}}}\leq C_{4} \|v_{1}-v_{2}\|_{(X_{1},X_{0})_{\frac{1}{d},d}},
\end{eqnarray*}
for certain $C_{2},C_{3}, C_{4}>0$, where we have used Lemma \ref{propF} (vii) and \eqref{embedtrace}. 

(H2) If $t_{1},t_{2}\in[0,T_{0}]$, by Lemma \ref{propF} (vii) and \eqref{embedtrace} we estimate
\begin{eqnarray*}
\lefteqn{\|\p(v_{1},t_{1}) -\p(v_{2},t_{2})\|_{X_{0}}=\|\frac{1}{2\pi i}\int_{\Gamma_{1}}\Big(\frac{\p(\lambda,t_{1})}{\lambda-v_{1}}-\frac{\p(\lambda,t_{2})}{\lambda-v_{2}}\Big)d\lambda\|_{X_{0}}}\\
&\leq& \frac{1}{2\pi}\int_{\Gamma_{1}}\Big(|\p(\lambda,t_{1})|\|\frac{v_{1}-v_{2}}{(\lambda-v_{1})(\lambda-v_{2})}\|_{X_{0}}+|\p(\lambda,t_{1})-\p(\lambda,t_{2})|\|\frac{1}{\lambda-v_{2}}\|_{X_{0}}\Big)d\lambda\\
&\leq& C_{5}(\|v_{1}-v_{2}\|_{X_{0}}+|t_{1}-t_{2}|\|1_{\mathbb{F}}\|_{X_{0}})\leq C_{6}(\|v_{1}-v_{2}\|_{(X_{1},X_{0})_{\frac{1}{d},d}}+|t_{1}-t_{2}|),
\end{eqnarray*}
for some $C_{5},C_{6}>0$, where $1_{\mathbb{F}}$ stands for the constant function equal to one on $\mathbb{F}$.

(H3) Maximal $L^{q}$-regularity for the linearized term follows by Theorem \ref{KaWeTh}, Lemma \ref{propF} (i), Theorem \ref{maxreglinterm} and \eqref{u0inH}. 

We conclude that there exists a $T\in(0,T_{0}]$ and a unique 
\begin{equation}\label{vreg}
v\in H_{d}^{1}(0,T;X_{0})\cap L^{d}(0,T;X_{1})\hookrightarrow C([0,T];(X_{1},X_{0})_{\frac{1}{d},d})
\end{equation}
solving \eqref{PME3}-\eqref{PME4}. 

By \eqref{embedtrace}-\eqref{vreg} we get $v\in C([0,T];C(\mathbb{F}))$. Choose $T>0$ sufficiently small such that $\mathrm{Re}(v(t))>\beta$ for all $t\in [0,T]$ and certain $\beta>0$. Then for any $t_{1},t_{2}\in[0,T]$ by Lemma \ref{propF} (vii) we have
$$
\|v^{\frac{1}{m}}(t_{1})-v^{\frac{1}{m}}(t_{2})\|_{X_{0}}\leq C_{7}\|1_{\mathbb{F}}\|_{X_{0}}\|v(t_{1})-v(t_{2})\|_{\mathcal{H}_{q,p,\gamma+(1-\delta)(1-\frac{1}{d})-\varepsilon}^{r+(1+\delta)(1-\frac{1}{d})-\varepsilon,s+(1-\delta)(1-\frac{1}{d})-\varepsilon}(\mathbb{F})_{\oplus\mathbb{C}_{\omega}}},
$$
for some $C_{7}>0$ and all $\varepsilon>0$ sufficiently small. This together with \eqref{embedtrace}-\eqref{vreg} implies 
\begin{equation}\label{extracontreg}
u\in C([0,T];X_{0}).
\end{equation}
Moreover, by using Lemma \ref{propF} (vii) and \eqref{embedtrace}-\eqref{vreg} we obtain
$$
m^{d}\int_{0}^{T}\|u'(t)\|_{X_{0}}^{d}dt=\int_{0}^{T}\|v^{\frac{1-m}{m}}(t)v'(t)\|_{X_{0}}^{d}dt\leq C_{8}\int_{0}^{T}\|v'(t)\|_{X_{0}}^{d}dt,
$$
for some $C_{8}>0$, which, together with \eqref{extracontreg}, implies that $u\in H_{d}^{1}(0,T;X_{0})$. We conclude that $u$ satisfies \eqref{regofu1}-\eqref{regofu2}. 

Suppose now that there exists a second solution $w$ satisfying \eqref{regofu1}-\eqref{regofu2}. Then
$$
\int_{0}^{T}\|(w^{m})'(t)\|_{X_{0}}^{d}dt=m^{d}\int_{0}^{T}\|(w^m)^{\frac{m-1}{m}}(t)w'(t)\|_{X_{0}}^{d}dt\leq C_{9}\int_{0}^{T}\|w'(t)\|_{X_{0}}^{d}dt,
$$
for certain $C_{9}>0$, where we have used \eqref{regofu2} and Lemma \ref{propF} (vii). We conclude that $w^{m}$ is a second solution of \eqref{PME3}-\eqref{PME4} satisfying the regularity \eqref{vreg}, which is a contradiction. \mbox{\ } \hfill $\Box$

{\bf Acknowledgements:} We would like to express our gratitude to the anonymous referee for the careful reading of the manuscript and the thoughtful comments, including the definition of the function spaces on manifolds with edges and better proofs of Proposition \ref{mxtderPr} and Remark \ref{UMDalpha}.


\begin{thebibliography}{99}

\bibitem{Am} H. Amann. {\em Linear and quasilinear parabolic problems, Vol. I Abstract linear theory}. Monographs in Mathematics {\bf 89}, Birkh\"auser Verlag (1995).

\bibitem{Am2} H. Amann. {\em Linear and quasilinear parabolic problems, Vol. II Function spaces}. Monographs in Mathematics {\bf 106}, Birkh\"auser Verlag (2019).

\bibitem{Am4} H. Amann, M. Hieber, G. Simonett. {\em Bounded $H_{\infty}$-calculus for elliptic operators}. Differential Integral Equations {\bf 7}, no. 3--4, 613--653 (1994).

\bibitem{AP} D. Aronson, L. Peletier. {\em Large time behaviour of solutions of the porous medium equation in bounded domains}. Journal of Differential Equations {\bf 39}, no. 3, 378--412 (1981).

\bibitem{BaVe1} E. Bahuaud, B. Vertman. {\em Long-time existence of the edge Yamabe flow}. J. Math. Soc. Japan {\bf 71}, no. 2, 651--688 (2019).

\bibitem{BaVe2} E. Bahuaud, B. Vertman. {\em Yamabe flow on manifolds with edges}. Math. Nachr. {\bf 287}, no. 23, 127--159 (2014).

\bibitem{BG} M. Bonforte, G. Grillo. {\em Asymptotics of the porous media equation via Sobolev inequalities}. J. Funct. Anal. {\bf 225}, no. 1, 33--62 (2005).

\bibitem{CL} P. Cl\'ement, S. Li. {\em Abstract parabolic quasilinear equations and application to a groundwater flow problem}. Adv. Math. Sci. Appl. {\bf 3}, Special Issue, 17--32 (1993/94).

\bibitem{CSS} S. Coriasco, E. Schrohe, J. Seiler. {\em Differential operators on conic manifolds: Maximal regularity and parabolic equations}. Bull. Soc. Roy. Sci. Li\`ege {\bf 70}, no. 4-6, 207--229 (2001).

\bibitem{DG} G. Da Prato, P. Grisvard. {\em Sommes d'opérateurs linéaires et équations différentielles opérationnelles}. J. Math. Pures Appl. (9) {\bf 54}, no. 3, 305--387 (1975).

\bibitem{DHa} P. Daskalopoulos, R. Hamilton. {\em Regularity of the free boundary for the porous medium equation}. Journal of the American Mathematical Society {\bf 11}, no. 4, 899--965 (1998).

\bibitem{DHP} R. Denk, M. Hieber, J. Pr\"uss. {\em $R$-boundedness, Fourier multipliers and problems of elliptic and parabolic type}. Mem. Amer. Math. Soc. {\bf 166}, no. 788, (2003).

\bibitem{GKM} J. Gil, T. Krainer, G. Mendoza. {\em Geometry and spectra of closed extensions of elliptic cone operators}. Canad. J. Math. {\bf 59}, no. 4, 742--794 (2007).

\bibitem{GKM3} J. Gil, T. Krainer, G. Mendoza. {\em On the closure of elliptic wedge operators}. J. Geom. Anal. {\bf 23}, no. 4, 2035--2062 (2013).

\bibitem{GKM2} J. Gil, T. Krainer, G. Mendoza. {\em Resolvents of elliptic cone operators}. J. Funct. Anal. {\bf 241}, no. 1, 1--55 (2006).

\bibitem{GM} J. Gil, G. Mendoza. {\em Adjoints of elliptic cone operators}. Amer. J. Math. {\bf 125}, no. 2, 357--408 (2003).

\bibitem{GrMu} G. Grillo, M. Muratori. {\em Smoothing effects for the porous medium equation on Cartan-Hadamard manifolds}. Nonlinear Analysis {\bf 131}, 346--362 (2016).

\bibitem{GMP} G. Grillo, M. Muratori, F. Punzo. {\em The porous medium equation with measure data on negatively curved Riemannian manifolds}. J. Eur. Math. Soc. {\bf 20}, no. 11, 2769--2812 (2018).

\bibitem{GMV} G. Grillo, M. Muratori, J. L. V\'azquez. {\em The porous medium equation on Riemannian manifolds with negative curvature. The large-time behaviour}. Adv. Math. {\bf 314}, 328--377 (2017).

\bibitem{GMV2} G. Grillo, M. Muratori, J. L. V\'azquez. {\em The porous medium equation on Riemannian manifolds with negative curvature: the superquadratic case}. Math. Ann. {\bf 373}, no. 1-2, 119--153 (2019).

\bibitem{Ha} M. Haase. {\em The functional calculus for sectorial operators}. Operator theory: Advances and applications {\bf169}, Birkh\"auser Verlag (2006).

\bibitem{DH} R. Haller-Dintelmann, M. Hieber. {\em $H^{\infty}$-calculus for products of non-commuting operators}. Math. Z. {\bf251}, 85--100 (2005).

\bibitem{HNVW} T. Hyt\"onen, J. Neerven, M. Veraar, L. Weis. {\em Analysis in Banach spaces, Vol. I: Martingales and Littlewood-Paley theory}. Ergebnisse der Mathematik und ihrer Grenzgebiete. 3. Folge / A Series of Modern Surveys in Mathematics {\bf 63}, Springer Verlag (2016).

\bibitem{HNVW5} T. Hyt\"onen, J. Neerven, M. Veraar, L. Weis. {\em Analysis in Banach spaces, Vol. II: Probabilistic methods and operator theory}. Ergebnisse der Mathematik und ihrer Grenzgebiete. 3. Folge / A Series of Modern Surveys in Mathematics {\bf 67}, Springer Verlag (2017).

\bibitem{HNVW2} T. Hyt\"onen, J. Neerven, M. Veraar, L. Weis. {\em Analysis in Banach spaces, Vol. III: Harmonic analysis and spectral theory}. Ergebnisse der Mathematik und ihrer Grenzgebiete. 3. Folge / A Series of Modern Surveys in Mathematics {\bf 76}, Springer Verlag (2023).

\bibitem{KaW} N. Kalton, L. Weis. {\em The $H^{\infty}$-calculus and sums of closed operators}. Math. Ann. {\bf 321}, no. 2, 319--345 (2001).

\bibitem{KaSc} D. Kapanadze, B-W. Schulze. {\em Crack theory and edge singularities}. Mathematics and Its Applications {\bf 561}, Springer Verlag (2003).

\bibitem{KM} T. Krainer, G. Mendoza. {\em The Friedrichs extension for elliptic wedge operators of second order}. Adv. Differ. Equ. {\bf 23}, no. 3, 295--328 (2018).

\bibitem{KW1} P. C. Kunstmann, L. Weis. {\em Maximal $L_{p}$-regularity for parabolic equations, Fourier multiplier theorems and $H^\infty$-functional calculus}. Functional Analytic Methods for Evolution Equations, Lecture Notes in Mathematics {\bf 1855}, 65--311, Springer Verlag (2004).

\bibitem{KW} P. C. Kunstmann, L. Weis. {\em Perturbation theorems for maximal $L_p$-regularity}. Ann. Scuola Norm. Sup. Pisa Cl. Sci. (4) {\bf 30}, no. 2, 415--435 (2001).

\bibitem{Le} M. Lesch. {\em Operators of Fuchs type, conical singularities, and asymptotic methods}. Teubner-Texte zur Mathematik {\bf 136}, Teubner Verlag (1997).

\bibitem{LR} P. T. P. Lopes, N. Roidos. {\em Smoothness and long time existence for solutions of the Cahn-Hilliard equation on manifolds with conical singularities}. Monatshefte f\"ur Mathematik {\bf197}, 677--716 (2022).

\bibitem{Lunar18} A. Lunardi. {\em Interpolation theory}. Lecture Notes Scuola Normale Superiore {\bf 16}, Edizioni della Normale (2018).

\bibitem{LyVe} J. O. Lye, B. Vertman. {\em Long-time existence of Yamabe flow on singular spaces with positive Yamabe constant}. Analysis $\&$ PDE {\bf 16}, no. 2, 477--510 (2023).

\bibitem{Otto} F. Otto. {\em The geometry of dissipative evolution equations: The porous medium equation}. Comm. Partial Differential Equations {\bf 26}, no. 1-2, 101--174 (2001).

\bibitem{PS} J. Pr\"uss, G. Simonett. {\em $H^\infty$-calculus for the sum of non-commuting operators}. Transactions Amer. Math. Soc. {\bf 359}, no. 8, 3549--3565 (2007).

\bibitem{PS2} J. Pr\"uss, G. Simonett. {\em Moving interfaces and quasilinear parabolic evolution equations}. Monographs in Mathematics {\bf 105}, Birkh\"auser Verlag (2016).

\bibitem{ReSi} M. Reed, B. Simon. {\em Methods of modern mathematical physics IV, Analysis of operators}. Academic Press (1978).

\bibitem{Roi} N. Roidos. {\em On the inverse of the sum of two sectorial operators}. J. Funct. Anal. {\bf 265}, no. 2, 208--222 (2013).

\bibitem{RS1} N. Roidos, E. Schrohe. {\em Bounded imaginary powers of cone differential operators on higher order Mellin-Sobolev spaces and applications to the Cahn-Hilliard equation}. J. Differential Equations {\bf 257}, no. 3, 611--637 (2014).

\bibitem{RS2} N. Roidos, E. Schrohe. {\em Existence and maximal $L^{p}$-regularity of solutions for the porous medium equation on manifolds with conical singularities}. Comm. Partial Differential Equations {\bf 41}, no 9, 1441--1471 (2016).

\bibitem{RS} N. Roidos, E. Schrohe. {\em Smoothness and long time existence for solutions of the porous medium equation on manifolds with conical singularities}. Comm. Partial Differential Equations {\bf 43}, no. 10, 1456--1484 (2018).

\bibitem{RS3} N. Roidos, E. Schrohe. {\em The Cahn-Hilliard equation and the Allen-Cahn equation on manifolds with conical singularities}. Comm. Partial Differential Equations {\bf 38}, no. 5, 925--943 (2013).

\bibitem{RSS} N. Roidos, E. Schrohe, J. Seiler. {\em Bounded $H_{\infty}$-calculus for boundary value problems on manifolds with conical singularities}. Journal of Differential Equations {\bf 297}, 370--408 (2021).

\bibitem{RShao} N. Roidos, Y. Shao. {\em The fractional porous medium equation on manifolds with conical singularities I}. J. Evol. Equ. {\bf 22}, no. 1 (2022).

\bibitem{SS} E. Schrohe, J. Seiler. {\em Bounded $H_{\infty}$-calculus for cone differential operators}. J. Evol. Equ. {\bf 18}, no. 3, 1395--1425 (2018).

\bibitem{SS2} E. Schrohe, J. Seiler. {\em The resolvent of closed extensions of cone differential operators}. Can. J. Math. {\bf 57}, no. 4, 771--811 (2005).

\bibitem{Schulze1} B.-W. Schulze. {\em Pseudo-differential boundary value problems, conical singularities, and asymptotics}. Mathematical Topics {\bf 4}, Akademie Verlag (1994).

\bibitem{Schulze2} B.-W. Schulze. {\em Pseudo-differential operators on manifolds with singularities}. Studies in Mathematics and Its Applications {\bf 24}, North-Holland Publishing Co. (1991).

\bibitem{Schulze3} B.-W. Schulze. {\em Topologies and invertibility in operator spaces with symbolic structures}. Teubner-Texte Math. {\bf 111}, 259--288 (1989).

\bibitem{Sei} J. Seiler. {\em The cone algebra and a kernel characterization of Green operators}. Approaches to Singular Analysis, Operator Theory: Advances and Applications {\bf 125}, Birkh\"auser Verlag (2001).

\bibitem{Va0} J. L. V\'azquez. {\em Fundamental solution and long time behavior of the porous medium equation in hyperbolic space}. J. Math. Pures Appl. (9) {\bf 104}, no. 3, 454--484 (2015).

\bibitem{Va} J. L. V\'azquez. {\em The porous medium equation, mathematical theory}. Oxford Mathematical Monographs, Oxford University Press (2007).

\bibitem{Vertm} B.Vertman. {\em Ricci de Turck flow on singular manifolds}. J. Geom. Anal. {\bf 31}, no. 4, 3351--3404 (2020).

\bibitem{Weis} L. Weis. {\em Operator-valued Fourier multiplier theorems and maximal $L_{p}$-regularity}. Math. Ann. {\bf 319}, no. 4, 735--758 (2001).

\end{thebibliography}
\end{document}